\documentclass[11pt]{article}

\usepackage{amsthm,amsmath,amssymb,bbm}
\usepackage{natbib}
\usepackage{multirow}
\usepackage[pdftex]{graphicx}
\usepackage{subfigure}
\usepackage{makecell}
\usepackage{booktabs}
\usepackage{array}
\usepackage{tabularx}
\usepackage{tabulary}
\usepackage{caption}
\usepackage{booktabs}
\usepackage{url}
\usepackage{algorithm}
\usepackage{algorithmic}
\usepackage{bm}
\usepackage{wrapfig}
\usepackage{lipsum}
\usepackage{mathrsfs}
\usepackage{dsfont}
\usepackage{titling}

\makeatletter
\newcommand*{\rom}[1]{\expandafter\@slowromancap\romannumeral #1@}
\makeatother

\usepackage[usenames,dvipsnames,svgnames,table]{xcolor}
\usepackage[colorlinks,
linkcolor=red,
anchorcolor=blue,
citecolor=blue
]{hyperref}

\usepackage{xr}
\newcommand{\sgn}{\mathop{\mathrm{sign}}}

\usepackage{smile}

\newcommand*{\var}{\textnormal{var}}

\newcommand{\e}{\mathbb{E}}
\newcommand{\nn}{\nonumber}

\def\##1\#{\begin{align}#1\end{align}}
\def\$#1\${\begin{align*}#1\end{align*}}

\def\T{\intercal} 
\def\sn{\sum_{i=1}^n}
\def\Sb{\mathbf{S}}

\newcommand{\cov}{\mathrm{cov}}
\newcommand{\FDP}{{\rm FDP}}

\newcommand{\pr}{\mathbb{P}}

\newcommand{\wt}{\widetilde}

\newcommand{\bfsym}[1]{\ensuremath{\boldsymbol{#1}}}
       \def \bbeta    {\bfsym{\beta}}
\def \bgamma   {\bfsym{\gamma}}       \def \bdelta   {\bfsym{\delta}}

\newcommand{\B}{\mathrm{\tiny{b}}}
\newcommand{\mS}{\mathbb{S}}
\def \BH {{\rm BH}}

\newcommand{\Rom}[1]{\text{\uppercase\expandafter{\romannumeral #1\relax}}}

\usepackage{geometry}
\usepackage{enumitem}

\begin{document}

\title{Robust Inference via Multiplier Bootstrap}
 
\author{Xi Chen\thanks{Stern School of Business, New York University, New York, NY 10012, USA. E-mail: \href{mailto:xchen3@stern.nyu.edu}{\textsf{xchen3@stern.nyu.edu}}.}~~~and~~Wen-Xin Zhou\thanks{Department of Mathematics, University of California, San Diego, La Jolla, CA 92093, USA. E-mail:  \href{mailto:wez243@ucsd.edu}{\textsf{wez243@ucsd.edu}}. Supported in part by NSF Grant DMS-1811376.} }

\date{}
\maketitle

\vspace{-0.5in}

\begin{abstract}
This paper investigates the theoretical underpinnings of two fundamental statistical inference problems, the construction of confidence sets and large-scale simultaneous hypothesis testing,  in the presence of heavy-tailed data.
With heavy-tailed observation noise, finite sample properties of the least squares-based methods, typified by the sample mean, are suboptimal both theoretically and empirically. In this paper, we demonstrate that the adaptive Huber regression, integrated with the multiplier bootstrap procedure, provides a useful robust alternative to the method of least squares. Our theoretical and empirical results reveal the effectiveness of the proposed method, and highlight the importance of having inference methods that are robust to heavy tailedness.

\end{abstract}

\section{Introduction}
\label{sec1}

In classical statistical analysis, it is typically assumed that data are drawn from a Gaussian distribution. {Although the normality assumption has been widely adopted to facilitate methodological development and theoretical analysis,  Gaussian models could be an idealization of the complex random world.} The non-Gaussian, or even heavy-tailed, character of the distribution of empirical data has been repeatedly observed in various domains, ranging from genomics, medical imaging to economics and finance. New challenges are thus brought to classical statistical methods.
For linear models, regression estimators based on the least squares loss are suboptimal, both theoretically
and empirically, in the presence of heavy-tailed errors.
The necessity of robust alternatives to  least squares was first noticed by Peter Huber in his seminal work  ``Robust Estimation of a Location Parameter'' \citep{H1964}. Due to the growing complexity of modern data, the notion of robustness is becoming increasingly important in statistical analysis and finds its use in a wide range of applications. We refer to \cite{HR2009} for an overview of robust statistics.

Although the past a few decades have witnessed the active development of  rich statistical theory on robust estimation, {robust statistical inference for heavy-tailed data has always been a challenging problem on which the extant literature has been somewhat silent.
\cite{FHY2007}, \cite{DHJ2011} and \cite{LS2014} investigated robust inference that is confined to the Student's $t$-statistic. However, as pointed out by \cite{DLLO2016} (see Section 8 therein), sharp confidence estimation for heavy-tailed data in the finite sample set-up remains an open problem  and a general methodology is still lacking. To that end, this paper makes a further step in studying confidence estimation from a robust perspective. In particular, under linear model with heavy-tailed errors, we address two fundamental problems: (1) confidence set construction for regression coefficients, and (2) large-scale multiple testing with the guarantee of false discovery rate control. The developed techniques provide mathematical underpinnings for a class of robust statistical inference problems.} Moreover, sharp exponential-type bounds for the coverage probability of the confidence set are derived under weak moment assumptions.


\subsection{Confidence sets}

Consider the linear model $Y=   \bX^\T \btheta^* + \varepsilon$, where $Y \in \RR$ denotes the response variable, $\bX \in \RR^d$ is the (random) vector of covariates, $\btheta^* \in  \RR^d$ is the vector of regression coefficients and $\varepsilon$ represents the regression error satisfying $\EE(\varepsilon | \bX) = 0$ and $\sigma^2 = \EE(\varepsilon^2| \bX )<\infty$. Assume that we observe a random sample $(Y_1, \bX_1), \ldots, (Y_n, \bX_n)$ from $(Y, \bX)$:
\begin{align}\label{reg.model}
	Y_i =   \bX_i^\T \btheta^* + \varepsilon_i, \quad i=1,\ldots, n.
\end{align}
The intercept term is implicitly assumed in model \eqref{reg.model} by taking the first element of $\bX_i$ to be one so that the first element of $\btheta^*$ becomes the intercept.
{The least squares estimator and its variations have been widely adopted to estimate $\btheta^*$, which on many occasions achieve the minimax rate in terms of the mean squared error (MSE).}

{Although the MSE plays an important role in estimation, an estimator that is optimal in MSE might be suboptimal in terms of non-asymptotic deviation, which often relates to the construction of confidence intervals.} For example, in the mean estimation problem, although the sample mean has an optimal minimax mean squared error among all mean estimators, its deviation is worse for non-Gaussian samples than for Gaussian ones, and the worst-case deviation is suboptimal when the sampling distribution has heavy tails \citep{C2012}. More specifically, let $X_1,\ldots,X_n$ be independent random variables from $X$ with mean $\mu$ and variance $\sigma^2>0$. Consider the empirical mean $\hat{\mu}_n = (1/n)\sn X_i$, applying Chebyshev's inequality delivers a polynomial-type deviation bound
\begin{align*}
	  \PP\left(  | \hat{\mu}_n  - \mu | \geq \sigma\sqrt{\frac{1}{  \delta n }} \right) \leq \delta ~\mbox{ for any } \delta \in (0,1).
\end{align*}
In addition, if the distribution of $X$ is sub-Gaussian, i.e. $\EE \exp(\lambda X) \leq \exp(\sigma^2 \lambda^2/2)$ for all $\lambda$, then following the terminology in \cite{DLLO2016}, $\hat{\mu}_n$ becomes a sub-Gaussian estimator in the sense that
\begin{align*}
  \PP\left\{  | \hat{\mu}_n  - \mu | \geq \sigma\sqrt{ \frac{ 2\log(2/\delta) }{n} } \right\} \leq \delta .
\end{align*}
\cite{C2012} established a lower bound for the deviations of $\hat{\mu}_n$ when the sampling distribution is the least favorable in the class of all distributions with bounded variance: for any $\delta \in (0, e^{-1})$, there is some distribution with mean $\mu$ and variance $\sigma^2$ such that an independent sample of size $n$ drawn from it satisfies
\begin{align*}
 \PP\left\{	| \hat{\mu}_n - \mu | \geq \sigma \sqrt{\frac{1}{\delta n } } \bigg( 1 - \frac{e \delta  }{n} \bigg)^{(n-1)/2} \right\} \geq \delta .
\end{align*}
This shows that the deviation bound obtained from Chebyshev's inequality is essentially sharp under finite variance condition.
The limitation of least squares method arises also in the regression setting, which triggers an outpouring of interest in developing sub-Gaussian estimators, from univariate mean estimation to multivariate or even high dimensional problems,  for heavy-tailed data to achieve sharp deviation bounds from a non-asymptotic viewpoint. See, for example, \cite{BJL2015},  \cite{M2015, M2016},  \cite{HS2016}, \cite{DLLO2016}, \cite{CG2017}, \cite{G2017}, \cite{FLW2017}, \cite{SZF2016} and  \cite{LM2017}, among others. {In particular,  \cite{FLW2017}  and  \cite{SZF2016} proposed adaptive (regularized) Huber estimators with diverging robustification parameters (see Definition \ref{Huber.def} in Section \ref{sec2.1}), and derived exponential-type deviation bounds when the error distribution only has finite variance. Their key observation is that the robustification parameter should adapt to the sample size, dimensionality and noise level for optimal tradeoff between bias and robustness.}

All the aforementioned work studies robust estimation through concentration properties, that is, the robust estimator is tightly
concentrated around the true parameter  with high probability even when the sampling distribution has only a small number of finite moments.
In general, concentration inequalities loose constant factors and may result in confidence intervals too wide to be informative.
Therefore, an interesting and challenging open problem is how to construct tight confidence sets for $\btheta^*$ with finite samples of heavy-tailed data \citep{DLLO2016}.


{This paper addresses this open problem by developing a robust inference framework with non-asymptotic guarantees.} To illustrate the key idea, we focus on the classical setting where the parameter dimension $d$ is smaller than but is allowed to increase with the number of observations $n$. {Our approach integrates concentration properties of the adaptive Huber estimator (see Theorems \ref{br.thm} and \ref{wilks.thm}) and  the multiplier bootstrap method.}
The multiplier bootstrap, also known as the weighted bootstrap, is one of the most widely used resampling methods for constructing a confidence interval/set or for measuring the significance of a test. Its theoretical validity is typically guaranteed by the multiplier central limit theorem \citep{VW1996}. We refer to \cite{CB2005}, \cite{ABR2010}, \cite{CCK2013, CCK2014}, \cite{SZ2015} and \cite{Z2016} for the most recent progress in the theory and applications of the multiplier bootstrap. In particular, \cite{SZ2015} considered a multiplier bootstrap procedure for constructing likelihood-based confidence sets under a possible model misspecification. For a linear model with sub-Gaussian errors, their results validate the bootstrap procedure when applied to the ordinary least squares (OLS).  With heavy-tailed errors in the regression model \eqref{reg.model}, we demonstrate how the adaptive Huber regression and the multiplier bootstrap can be integrated to construct robust and sharp confidence sets for the true parameter $\btheta^*$ with a given coverage probability. The validity of the bootstrap procedure in situations with a limited sample size, growing dimensionality and heavy-tailed errors is established. In all these theoretical results, we provide non-asymptotic bounds for the errors of bootstrap approximation. See Theorems~\ref{boot.concentration.thm} and \ref{boot.wilks.thm} for finite sample properties of the bootstrap adaptive Huber estimator, including the deviation inequality, Bahadur representation and Wilks' expansion.

An alternative robust inference method is based on the asymptotic theory developed in \cite{ZBFL2017}; see, for example, Theorems~2.2 and 2.3 therein. Since the asymptotic distribution of either the proposed robust estimator itself or the excess risk depends on $\sigma^2$, a direct approach is to replace $\sigma^2$ by some sub-Gaussian variance estimator using Catoni's method \citep{C2012} or the median-of-means technique \citep{M2015}, with the advantage of being computationally fast. The disadvantage, however, is two-fold: first, constructing sub-Gaussian variance estimator involves another tuning parameter (for the problem of simultaneously testing $m$ regression models as discussed in the next section, variance estimation brings $m$ additional tuning parameters); secondly, because
the squared heavy-tailed data is highly right-skewed, using the method in \cite{C2012} or \cite{FLW2017} tends to underestimate the variance, and the median-of-means
method is numerically unstable for small or moderate samples. Both two methods were examined numerically in \cite{ZBFL2017}, while the multiplier bootstrap procedure,
albeit being more computationally intensive, demonstrates the most desirable finite sample performance.

\subsection{Simultaneous inference}

In addition to building confidence sets for an individual parameter vector, multiple hypothesis testing is another important
statistical problem with applications to many scientific fields, where thousands
of tests are performed simultaneously \citep{DV2008, E2010}. Gene microarrays comprise a prototypical example; there, each subject is automatically measured on tens of thousands of features.
Together, the large number of tests together with heavy tailedness bring new challenges to conventional statistical methods, which, in this scenario, often suffer from low power to detect important features and poor reproducibility. Robust alternatives are thus needed for conducting large-scale multiple inference for heavy-tailed data.

In this section, we consider the multiple response regression model
\begin{align}
	y_{ik} = \mu_k +   \bx_i^\T  \bbeta_k +  \varepsilon_{ik} , \ \ i=1,\ldots, n, \, k=1,\ldots, m,  \label{panel.data}
\end{align}
where $\mu_k$ is the intercept, $\bx_i = (x_{i1}, \ldots, x_{is})^\T , \bbeta_k = (\beta_{k1}, \ldots, \beta_{ks})^\T \in \RR^s$ are $s$-dimensional vectors of random covariates and regression coefficients, respectively, and $\varepsilon_{ik}$ is the regression error. Since our main focus here is the inference for intercepts, we decompose the parameter vector $\btheta^*$ in \eqref{reg.model} into two parts: the intercept $\mu_k$ and the slope $\bbeta_k$. Moreover, we use $\bx_i$  in \eqref{panel.data} to distinguish from $\bX_i$ in \eqref{reg.model}.
Write $\by_i=(y_{i1}, \ldots, y_{im})^\T \in \RR^m$ and let $\bmu=(\mu_1,\ldots, \mu_m)^\T \in \RR^m$ be the vector of intercepts. Based on random samples $\{ (\by_i, \bx_i) \}_{i=1}^n$ from model \eqref{panel.data}, our goal is to simultaneously test the hypotheses
\begin{align}
H_{0k} : \mu_k = 0  ~\mbox{ versus }~ H_{1k}: \mu_k \neq 0 , ~\mbox{ for } k=1,\ldots, m. \label{hypotheses}
\end{align}

An iconic example of model \eqref{panel.data} is the linear pricing model, which subsumes the capital asset pricing model (CAPM) \citep{S1964, L1965} and the Fama-French three-factor model \citep{FF1993}. The key implication from the multi-factor pricing theory is that for any asset $k$, the intercept $\mu_k$ should be zero. It is then important to investigate if such a pricing theory, also known as the ``mean-variance efficiency" pricing, can be validated by empirical data \citep{FLY2015}. According to the Berk and Green equilibrium \citep{BG2004}, inefficient pricing by the market may occur to a small proportion of exceptional assets, namely a very small fraction of the $\mu_k$'s are nonzero. To identify positive $\mu_k$'s by testing a large number of hypotheses simultaneously, \cite{BSW2010} and \cite{LD2017} developed FDR controlling procedures for data coming from model \eqref{panel.data}, which can be applied to mutual fund selection in empirical finance. We refer to  \cite{FKC2009},  \cite{DS2012}, \cite{FHG2012} and \cite{WZHO2017} for more examples from gene expression studies, where the goal is to identify features showing a biological signal of interest.

{Despite the extensive research and wide application of this problem, existing least squares-based methods with normal calibration could fail when applied to heavy-tailed data with a small sample size. To address this challenge, we develop a robust bootstrap procedure for large-scale simultaneous inference, which achieves good numerical performance for a small or moderate sample size. Theoretically, we prove its validity on controlling the false discover proportion (FDP) (see Theorem \ref{FDP.thm}).}

Finally, we  briefly comment on the computation issue. Fast computation of Huber regression is critical to our procedure since the multiplier bootstrap requires solving Huber loss minimization for at least hundreds of times. Ideally, a second order approach (e.g. Newton's method) is preferred. However, the second order derivative of Huber loss does not exist everywhere. To address this issue, we adopt the damped semismooth Netwon method \citep{QS1999}, which is a synergic integration of first and second order methods.
The details are provided in Appendix~\ref{app:C} of the supplemental material.

\subsection{Organization of the paper}

The rest of the paper proceeds as follows.
Section~\ref{sec2.1} presents a series of finite sample results of the adaptive Huber regression. Sections~\ref{sec.boot} and \ref{sec2.3} contain, respectively, the description of the bootstrap procedure for building confidence sets and theoretical guarantees.
Two data-driven schemes are proposed in Section~\ref{sec:adap} for choosing the tuning parameter in the Huber loss.
In Section~\ref{sec3}, we propose a robust bootstrap calibration method for multiple testing and investigate its theoretical property on controlling the FDP. The conclusions that are drawn in Sections~\ref{sec2} and \ref{sec3} are illustrated numerically in Section~\ref{sec.numerical}.  We conclude with a discussion in Section~\ref{sec.discuss}.
The supplementary material contains all the  proofs and additional simulation studies.

\subsection{Notation}
Let us summarize our notation. For every integer $k\geq 1$, we use $\RR^k$ to denote the the $k$-dimensional Euclidean space. The inner product of any two vectors $\bu=(u_1, \ldots, u_k)^\T, \bv=(v_1, \ldots ,v_k)^\T \in \RR^k$ is defined by $\bu^\T \bv = \langle \bu, \bv \rangle= \sum_{i=1}^k u_i v_i$.
We use the notation $\| \cdot \|_p , 1\leq p \leq \infty$ for the $\ell_p$-norms of vectors in $\RR^k$: $\| \bu \|_p = ( \sum_{i=1}^k | u_i |^p )^{1/p}$ and $\| \bu \|_\infty = \max_{1\leq i\leq k} |u_i|$. For $k\geq 2$, $\mathbb{S}^{k-1} = \{ \bu \in \RR^k : \| \bu \|_2 = 1 \}$ denotes the unit sphere in $\RR^k$. Throughout this paper, we use bold capital letters to represent matrices. For $k\geq 2$, $\bI_k$ represents the identity/unit matrix of size $k$. For any $k\times k$ symmetric matrix $\bA \in \RR^{k\times k}$, $\| \bA \|_2$ is the operator norm of $\bA$. We use $ \overline \lambda_{\bA} $ and $\underline \lambda_{\bA}$ to denote the largest and smallest eigenvalues of $\bA$, respectively. For any two real numbers $u$ and $v$, we write $u\vee v = \max(u,v)$ and $u \wedge v = \min(u,v)$.
For two sequences of non-negative numbers $\{ a_n \}_{n\geq 1}$ and $\{ b_n \}_{n\geq 1}$, $a_n \lesssim b_n$ indicates that there exists a constant $C>0$ independent of $n$ such that $a_n \geq Cb_n$; $a_n \gtrsim b_n$ is equivalent to $b_n \lesssim a_n$; $a_n \asymp b_n$ is equivalent to $a_n \lesssim b_n$ and $b_n \lesssim a_n$. For two numbers $C_1$ and $C_2$, we write $C_2=C_2(C_1)$ if $C_2$ depends only on $C_1$. For any set $\mathcal{S}$, we use ${\rm card}(\mathcal{S})$ and $|\mathcal{S}|$ to denote its cardinality, i.e. the number of elements in $\mathcal{S}$.

\section{Robust bootstrap confidence sets}
\label{sec2}

\subsection{Preliminaries}
\label{sec2.1}

First, we present some finite sample properties of the adaptive Huber estimator, which are of independent interest and also sharpen the results in \cite{SZF2016}.

Let us recall the definition of Huber loss,
\begin{definition} \label{Huber.def}
{\rm The Huber loss $\ell_\tau(\cdot)$ \citep{H1964} is defined as
\begin{align}   \label{huber.loss}
	\ell_\tau(u) =
	\left\{\begin{array}{ll}
	 u^2/2  ,    & \mbox{if } | u | \leq  \tau ,  \\
	\tau | u | -  \tau^2 /2 ,   &  \mbox{if }  | u | > \tau ,
	\end{array}  \right.
\end{align}
where $\tau >0$ is a tuning parameter and will be referred to as the {\it robustification parameter} that balances bias and robustness.}
\end{definition}
The Huber estimator is defined as
\begin{align}
	\hat{\btheta}_\tau \in \argmin_{\btheta \in \RR^d} 	\cL_\tau(\btheta)   ~\mbox{ with }~  \cL_\tau(\btheta)= \cL_{n,\tau}(\btheta) : =  \sn  \ell_\tau(Y_i - \bX_i^\T \btheta) . \label{huber.est}
\end{align}
The following theorem provides a sub-Gaussian-type deviation inequality and a non-asymptotic Bahadur representation for $\hat{\btheta}_\tau$. The proof is given in  the supplement. We first impose the moment conditions.

\begin{assumption}
\label{moment.cond}
{\rm
(i) There exists some constant $A_0>0$ such that for any $\bu \in \RR^d$ and $t\in \RR$, $\PP(|\langle \bu, \bZ \rangle |\geq A_0 \| \bu \|_2 \cdot  t) \leq 2\exp(-t^2)$, where $\bZ= \bSigma^{-1/2} \bX$ and $\bSigma = \EE(\bX \bX^\T)$ is positive definite. (ii) The regression error $\varepsilon$ satisfies $\EE(\varepsilon | \bX)=0$, $\EE(\varepsilon^2 | \bX)=\sigma^2$ and $ \EE(|\varepsilon|^{2+\delta} | \bX) \leq \upsilon_{2+\delta}$ almost surely for some $\delta \geq 0$.
}
\end{assumption}

Part~(i) of Condition~\ref{moment.cond} requires $\bX$ to be a sub-Gaussian vector. Via one-dimensional marginal, this generalizes the concept of sub-Gaussian random variables to higher dimensions. Typical examples include: (i) Gaussian and Bernoulli random vectors, (ii) spherical random vector\footnote{A random vector $\bX \in \RR^d$ is said to have a spherical distribution if it is uniformly distributed on the Euclidean sphere in $\RR^d$ with center at the origin and radius $\sqrt{d}$. }, (iii) random vector that is uniformly distributed on the Euclidean ball centered at the origin with radius $\sqrt{d}$, and (iv) random vector that is uniformly distributed on the unit cube $[-1,1]^d$. In all the above cases, the constant $A_0$ represents a dimension-free constant. We refer to Chapter~3.4 in \cite{V2018} for detailed discussions of sub-Gaussian distributions in higher dimensions.
Technically, this assumption is needed in order to derive an exponential-type concentration inequality for the quadratic form $\| \sn \ell'_\tau(\varepsilon_i) \bZ_i \|_2$, where
\begin{align}
	\bZ_i = \bSigma^{-1/2} \bX_i, \ \ i=1,\ldots, n .  \label{Zi.def}
\end{align}

\begin{theorem} \label{br.thm}
Assume Condition~\ref{moment.cond} holds. For any $t>0$ and $v \geq \upsilon_{2+\delta}^{1/(2+\delta)}$, the estimator $\hat{\btheta}_\tau$ given in \eqref{huber.est} with $\tau=v (\frac{n}{d+t})^{1/(2+\delta)}$ satisfies
\begin{align}
	\PP\Bigg\{ \| \bSigma^{1/2}(\hat{\btheta}_\tau - \btheta^*) \|_2 \geq  c_1  v \sqrt{\frac{d+t}{n}}  \Bigg\} \leq 2 e^{-t} \label{concentration.MLE} \\
	\mbox{ and }~  \PP\Bigg\{   \bigg\| \bSigma^{1/2}(\hat{\btheta}_\tau - \btheta^*)  - \frac{1}{n} \sn \ell'_\tau(\varepsilon_i) \bZ_i  \bigg\|_2 \geq  c_2 v \frac{d+t}{n} \Bigg\} \leq  3 e^{-t}  \label{Bahadur.representation}
\end{align}
as long as $n\geq c_3   (d+t) $, where $c_1$--$c_3$ are constants depending only on $A_0$.
\end{theorem}

The non-asymptotic results in Theorem~\ref{br.thm} reveal a new perspective for Huber's method: to construct sub-Gaussian estimators for linear regression with heavy-tailed errors, the tuning parameter in the Huber loss should adapt to the sample size, dimension and moments for optimal tradeoff between bias and robustness. The resulting estimator is therefore referred to as the {\it adaptive Huber estimator}. Specifically, Theorem~\ref{br.thm} provides the concentration property of the adaptive Huber estimator $\hat{\btheta}_\tau$ and the Fisher expansion for the difference $\hat{\btheta}_\tau - \btheta^*$.
It improves Theorem~2.1 in \cite{ZBFL2017} by sharpening the sample size scaling.
The classical asymptotic results can be easily derived from the obtained non-asymptotic expansions. In the following theorem, we further study the concentration property of the Wilks' expansion for the excess $ \cL_\tau(\btheta^* ) - \cL_\tau(\hat{\btheta}_\tau)$. This new result is directly related to the construction of confidence sets. See Theorem~\ref{boot.concentration.thm} below for its counterpart in the bootstrap world.

\begin{theorem} \label{wilks.thm}
Assume Condition~\ref{moment.cond} holds. Then for any $t>0$ and $v \geq \upsilon_{2+\delta}^{1/(2+\delta)}$, the estimator $\hat{\btheta}_\tau$ with $\tau=v (\frac{n}{d+t})^{1/(2+\delta)}$ satisfies that with probability at least $1- 3 e^{-t}$,
\begin{align}
\bigg|  \cL_\tau(\btheta^* ) - \cL_\tau(\hat{\btheta}_\tau)   - \frac{1}{2} \bigg\|  \frac{1}{\sqrt{n}}\sn \ell'_\tau(\varepsilon_i)  \bZ_i  \bigg\|_2^2  \bigg| \leq c_4 v^2 \frac{(d+t)^{3/2}}{\sqrt{n}} \label{Wilks.expansion} \\
\mbox{ and }~  \bigg| \sqrt{ 2\{  \cL_\tau(\btheta^* ) - \cL_\tau(\hat{\btheta}_\tau)\} } -   \bigg\|\frac{1}{\sqrt{n}} \sn \ell'_\tau(\varepsilon_i) \bZ_i   \bigg\|_2  \bigg| \leq c_5 v \frac{d+t}{\sqrt{n}} \label{sqrt.Wilks.expansion}
\end{align}
as long as $n \geq c_3  (d+t) $, where $c_4, c_5>0$ are constants depending on $A_0$.
\end{theorem}

\begin{remark}[{\sf On the robustification parameter $\tau$}] \label{rmk.order}
{\rm
Going through the proofs of Theorems~\ref{br.thm} and \ref{wilks.thm}, we see that the robustification parameter $\tau$ can be chosen as
\begin{align}
	\tau = v \{ n/(d+t) \}^{\eta} ~~\mbox{ for any }   \eta \in [1/(2+\delta), 1/2] ~\mbox{ and }~ v\geq \upsilon_{2+\delta}^{1/(2+\delta)}, \label{new.tau}
\end{align}
such that the conclusions \eqref{concentration.MLE}--\eqref{sqrt.Wilks.expansion} hold as long as $n \gtrsim d+t $. This implies that the existence of higher moments of $\varepsilon$ increases the flexibility of choosing $\tau$, whose order ranges from $(\frac{n}{d+t})^{1/(2+\delta)}$ to $(\frac{n}{d+t})^{1/2}$. In practice,  $\upsilon_{2+\delta}$ is unknown and thus brings difficulty in calibrating $\tau$. Guided by the theoretical results, in Section \ref{sec:adap} we propose a data-dependent procedure to choose $\tau$.}
\end{remark}

\begin{remark}[{\sf Sample size scaling}] \label{rmk.scaling}
{\rm
The deviation inequalities in Theorems~\ref{br.thm} and \ref{wilks.thm} hold under the scaling condition $n\gtrsim  d+t $, indicating that as many as $ d+t$ samples are required to guarantee the finite sample properties of the estimator. Similar conditions are also imposed for Proposition~2.4 in \cite{C2012} and Theorem~3.1 in \cite{AC2011}. In particular if $\EE( \varepsilon^2 )<\infty$, taking $t=  \log n$ and $\tau \asymp(\frac{n}{d+t})^{1/2}$, the corresponding estimator $\hat{\btheta}_\tau$ satisfies
$$
	 \hat{\btheta}_\tau = \btheta^* + \frac{1}{n} \sn \ell'_\tau(\varepsilon_i) \bSigma^{-1} \bX_i  + O\{ n^{-1} (d+\log n) \}
$$
with probability at least $1- O(n^{-1})$ under the scaling $n\gtrsim d$. From an asymptotic viewpoint, this implies that if the dimension $d$, as a function of $n$, satisfies $d  = o(n)$ as $n\to \infty$,
then for any deterministic vector $\bu \in \RR^d$, the distribution of the linear contrast $\bu^\T(\hat{\btheta}_\tau - \btheta^*)$ coincides with that of $ (1/n)\sn \ell'_\tau(\varepsilon_i) \bu^\T \bSigma^{-1} \bX_i$ asymptotically.
}
\end{remark}

\begin{remark}  \label{rmk.loss}
{\rm
To achieve sub-Gaussian behavior, the choice of loss function is not unique. An alternative loss function, which is obtained from the influence function proposed by \cite{CG2017}, is
\begin{align}   \label{catoni.loss}
	\rho_\tau(u) =
	\left\{\begin{array}{ll}
	  u^2/2 -  u^4/(24 \tau^2)  ,    & \mbox{if } | u | \leq  \sqrt{2} \, \tau ,  \\
	\frac{2\sqrt{2}}{3} \tau | u | -  \tau^2/2 ,   &  \mbox{if }  | u | > \sqrt{2} \, \tau .
	\end{array}  \right.
\end{align}
The function $\rho_\tau$ is convex, twice differentiable everywhere and  has bounded derivative that $|\rho_\tau'(u)|\leq (2\sqrt{2}/3 )\tau$ for all $u$.
By modifying the proofs of Theorems~\ref{br.thm} and \ref{wilks.thm}, it can be shown that the theoretical properties of the adaptive Huber estimator remain valid for the estimator that minimizes the empirical $\rho_\tau$-loss. Computationally, it can be solved via Newton's method.}
\end{remark}

\subsection{Multiplier bootstrap}
\label{sec.boot}

In this section, we go beyond estimation and focus on robust inference. According to \eqref{sqrt.Wilks.expansion}, the distribution of $2\{\cL_\tau(\btheta^*) - \cL_\tau(\hat{\btheta}_\tau)\}$ is close to that of $(1/n) \| \sn \xi_i  \bZ_i \|_2^2$ provided that $d^2/n$ is small, where $\xi_i = \ell_\tau'(\varepsilon_i)$. As we will see in the proof of Theorem~\ref{Boot.consistency} that, the truncated random variable $\xi_i$ has mean and variance {\it approximately} equal to $0$ and $\sigma^2$, respectively. Heuristically, the multivariate central limit theorem allows us to approximate the distribution of the normalized sum $n^{-1/2} \sn \xi_i  \bZ_i $ by $\mathcal{N}(\textbf{0}, \sigma^2 \bI_d)$. If this were true, then the distribution of $2\{ \cL_\tau (\btheta^*) - \cL_\tau(\hat{\btheta}_\tau) \}$ is close to the scaled chi-squared distribution $\sigma^2 \chi_d^2$ with $d$ degrees of freedom. This is in line with the asymptotic behavior of the likelihood ratio statistic that was studied in \cite{W1938}. With sample size sufficiently large, this result allows to construct confidence sets for $\btheta^*$ using quantiles of $\chi_d^2$: for any $\alpha \in (0,1)$,
\begin{align}
	\mathcal{C}^*_{\alpha}(\sigma) := \left\{ \btheta \in \RR^d: \cL_\tau(\btheta) - \cL_\tau(\hat{\btheta}_\tau) \leq  \sigma^2  \chi^2_{d, \alpha}/2 \right\} , \label{chi2.conf.set}
\end{align}
where $\chi^2_{d, \alpha}$ denotes the upper $\alpha$-quantile of $\chi_d^2$.
Estimating the residual variance $\sigma^2$ in the construction of $\mathcal{C}_\alpha^*(\sigma)$ is even more challenging when the errors are heavy-tailed. Moreover, as argued in \cite{SZ2015}, a possibly low speed of convergence of the likelihood ratio statistic makes the asymptotic Wilks' result hardly applicable to the case of small or moderate samples. Motivated by these two concerns, we have the following goal:
\begin{align}
&\mbox{propose a new method to construct confidence sets for $\btheta^*$ that is robust} \nn \\
 &\mbox{against heavy-tailed error distributions and performs well for a small or}  \nn \\
 &\mbox{moderate sample.} \nn
\end{align}

The results in Section~\ref{sec2.1} show that the adaptive Huber estimator provides a robust estimate of $\btheta^*$ in the sense that it admits sub-Gaussian-type deviations when the error distribution only has finite variance. To estimate the quantiles of the adaptive Huber estimator and to construct confidence set, we consider the use of multiplier bootstrap. Let $U_1, \ldots, U_n$ be independent and identically distributed (IID) random variables that are independent of the observed data $\mathcal{D}_n := \{ (Y_i, \bX_i) \}_{i=1}^n$ and satisfy
\begin{align}
	\EE(U_i) = 0, \quad \EE( U_i^2) =1 , \ \ i=1,\ldots , n. \label{weight.cond1}
\end{align}
With $W_i := 1 + U_i$ denoting the random weights, the bootstrap Huber loss and bootstrap Huber estimator are defined, respectively, as
\begin{align} \label{bootHuber.est}
	\mathcal{L}_\tau^\B(\btheta) = \sn W_i \, \ell_\tau(Y_i - \bX_i^\T \btheta) , \ \ \btheta \in \RR^d   \\
	\mbox{ and }~~	\hat{\btheta}^\B_\tau \in  \argmin_{ \btheta \in \RR^d: \| \btheta - \hat{\btheta}_\tau \|_2 \leq R } \mathcal{L}_\tau^\B(\btheta), \nn
\end{align}
where $R>0$ is a prespecified radius parameter. A simple observation is that $\EE^*\{\mathcal{L}_\tau^{\flat }(\btheta) \} = \cL_\tau(\btheta)$, where $\EE^*(\cdot) := \EE(\cdot\, |   \mathcal{D}_n )$ is the conditional expectation given the observed data $\mathcal{D}_n $. Therefore, $\hat{\btheta}_\tau \in \argmin_{\btheta \in \RR^d} \EE^*\{\mathcal{L}_\tau^\B(\btheta) \}$ and the difference $ \cL_\tau^\B(\hat{\btheta}_\tau) - \cL_\tau^{\flat}(\hat{\btheta}_\tau^{\flat })$ mimics $\cL_\tau(\btheta^*) - \cL_\tau(\hat{\btheta}_\tau)$.

\begin{algorithm}[!t]
    \caption{ {\small {\sf Huber Robust Confidence Set}}}
    \label{algo:huber_inference}
    {\textbf{Input:} Data $\{(Y_i, \bX_i)\}_{i=1}^n$, number of bootstrap samples $B$, Huber threshold $\tau$, radius parameter $R$, confidence level $1-\alpha$}
    \begin{algorithmic}[1]
      \STATE  Solve the Huber regression in \eqref{huber.est} and obtain $\hat{\btheta}_\tau$.
      \FOR{$b=1,2 \ldots, B$}
          \STATE Generate IID random weights $\{W_i\}_{i=1}^n$ satisfying $\EE(W_i)=1$ and $\var(W_i)=1$.
          \STATE Solve the weighted Huber regression in \eqref{bootHuber.est} and obtain the ``bootstrap'' Huber estimator. 
      \ENDFOR
      \STATE Define $\PP^*$ be the conditional probability over the random multipliers given the observed data $\mathcal{D}_n =\{ (Y_i, \bX_i) \}_{i=1}^n$, that is, $\PP^*(\cdot) = \PP(\cdot\,| \mathcal{D}_n)$.
      \STATE {Compute the upper $\alpha$-quantile of
      $ \cL^\B_\tau( \hat{\btheta}_\tau )  - \cL^\B_\tau( \hat \btheta_\tau^\B )  $:
      \begin{equation}\label{eq:boot_th}
          z^\B_{ \alpha} =\inf \{z \geq 0: \mathbb{P}^{*}\{  \cL^\B_\tau( \hat{\btheta}_\tau )  - \cL^\B_\tau( \hat \btheta_\tau^\B )   > z\} \leq \alpha \}.\nn
      \end{equation}
      }
    \end{algorithmic}
   \hspace{-1cm} \textbf{Output:} A confidence set of  $\btheta^*$ given by
    $\mathcal{C}_\alpha := \{\btheta \in \RR^d :   \cL_\tau( \btheta)  - \cL_\tau( \hat \btheta_\tau )   \leq z^\B_{ \alpha}  \}$.
\end{algorithm}

Based on this idea, we propose a Huber regression based inference procedure in Algorithm \ref{algo:huber_inference}, where the bootstrap threshold $z^\B_{\alpha}= z^\B_\alpha(\mathcal{D}_n)$ approximates
\begin{align}
z_\alpha := \inf \{ z \geq 0 :   \mathbb{P} \{  \cL_\tau( \btheta^*)  - \cL_\tau( \hat \btheta_\tau )  > z  \} \leq \alpha  \}. \label{def:zalpha}
\end{align}
Here $\mathbb{P}$ is the probability measure with respect to the underlying data generating process.


\subsection{Theoretical results}
\label{sec2.3}

In this section, we present detailed theoretical results for the bootstrap adaptive Huber estimator, including the deviation inequality, non-asymptotic Bahadur representation (Theorem~\ref{boot.concentration.thm}), and Wilks' expansions (Theorem~\ref{boot.wilks.thm}). Moreover, Theorems~\ref{Boot.consistency} and \ref{Boot.validity} establish the validity of the multiplier bootstrap for estimating quantiles of $\cL_\tau(\btheta^*)-\cL_\tau(\hat{\btheta}_\tau)$ when the variance $\sigma^2$ is unknown. Proofs of the finite sample properties of the bootstrap estimator require new techniques and are more involved than those of Theorems~\ref{br.thm} and \ref{wilks.thm}. We leave them to the supplemental material.

\begin{assumption} \label{weight.cond}
{\rm
$U_1,\ldots, U_n$ are IID from a random variable $U$ satisfying $\EE(U) = 0$, $\var(U) = 1$ and $\PP(|U|\geq t) \leq 2 \exp(-t^2/A_U^2)$ for all $t\geq 0$.}
\end{assumption}

\begin{theorem} \label{boot.concentration.thm}
Assume Condition~\ref{moment.cond} with $\delta=2$ and Condition~\ref{weight.cond} hold. For any $t>0$ and $v \geq \upsilon_4^{1/4}$, the estimator $\hat{\btheta}^\B_\tau$ with $\tau=v(\frac{n}{d+t})^{1/4}$ and $R\asymp v$ satisfies:
\begin{enumerate}
\item with probability (over $\mathcal{D}_n$) at least $1- 5 e^{-t}$,
\begin{align}
	\PP^*\{   \| \bSigma^{1/2}( \hat{\btheta}_\tau^\B - \btheta^*)  \|_2  \geq  c_1 v (d+t)^{1/2} n^{-1/2}   \} \leq  3 e^{-t} , \label{boot.est.deviation}
\end{align}

\item with probability (over $\mathcal{D}_n$) at least $1-  6 e^{-t}$,
\begin{align}
	\PP^*\Bigg\{   \bigg\| \bSigma^{1/2}( \hat{\btheta}_\tau^\B -  \hat \btheta_\tau )  - \frac{1}{n} \sn \ell_\tau'(\varepsilon_i)  U_i \bZ_i \bigg\|_2  \geq & \,c_2v \frac{d+t}{n}  \Bigg\}  \leq 4 e^{-t}   \label{boot.br}
\end{align}
\end{enumerate}
as long as $n \geq  \max\{ c_3\kappa_{\bSigma}(d+t) , c_4 (d  +t   )^2 \}$, where $c_1$--$c_3$ are positive constants depending on $(A_0, A_U)$, $c_4 = c_4(A_0)>0$  and $\kappa_{\bSigma} =  \overline \lambda_{\bSigma} / \underline \lambda_{\bSigma}$ is the condition number of $\bSigma$.
\end{theorem}


The following theorem is a bootstrap version of Theorem~\ref{wilks.thm}. Define the random process
\begin{align}
	\bxi^\B(\btheta ) = \bSigma^{-1/2} \{ \nabla \cL_\tau^\B(\btheta) - \nabla \EE^* \cL_\tau^\B(\btheta) \} , \ \ \btheta \in \RR^d. \label{xiB.def}
\end{align}
From \eqref{Zi.def} and \eqref{weight.cond1} we see that
$$
	\bxi^\B(\btheta )  = \bSigma^{-1/2}   \{ \nabla \cL_\tau^\B(\btheta) - \nabla  \cL_\tau(\btheta)   \}  = - \sn \ell'_\tau(Y_i - \bX_i^\T \btheta) U_i \bZ_i , \ \  \btheta \in \RR^d.
$$
In particular, write $\bxi^\B =	\bxi^\B(\btheta^*) =  - \sn \ell'_\tau(\varepsilon_i) U_i \bZ_i$.

\begin{theorem} \label{boot.wilks.thm}
Assume Condition~\ref{moment.cond} with $\delta=2$ and Condition~\ref{weight.cond} hold. For any $t>0$ and $v \geq \upsilon_4^{1/4}$, the bootstrap estimator $\hat{\btheta}^\B_\tau$ with $\tau=v (\frac{n}{d+t})^{1/4}$ and $R \asymp v$ satisfies that, with probability (over $\mathcal{D}_n$) at least $1- 5 e^{-t}$,
\begin{align}
  \PP^*\Bigg[   \bigg|  \cL^\B_\tau( \hat{\btheta}_\tau )  - \cL^\B_\tau( \hat \btheta_\tau^\B ) -  \frac{\|  \bxi^\B \|_2^2 }{2n}  \bigg| \geq   c_5   v^2   & \frac{(d+t)^{3/2}}{\sqrt{n}}     \Bigg] \leq  4 e^{-t}  \label{boot.wilks}
\end{align}
and
\begin{align}
 \PP^*\Bigg[ \bigg|  \sqrt{ 2 \{ \cL^\B_\tau( \hat{\btheta}_\tau )  - \cL^\B_\tau( \hat \btheta_\tau^\B )   \} } - \frac{\| \bxi^\B  \|_2}{\sqrt{n}}     \bigg|  \geq  c_6 v  & \frac{d+t}{\sqrt{n}}  \Bigg]   \leq  4 e^{-t} \label{boot.sqwilks}
\end{align}
as long as $n \geq  \max\{ c_3 \kappa_{\bSigma} (d+t) , c_4 (d  +t   )^2 \}$, where $c_5, c_6 >0$ are constants depending only on $(A_0, A_U)$.
\end{theorem}

The results \eqref{boot.wilks} and \eqref{boot.sqwilks} are non-asymptotic bootstrap versions of the Wilks' and square-root Wilks' phenomena. In particular, the latter indicates that the square-root excess
$
 \sqrt{ 2 \{ \cL^\B_\tau( \hat{\btheta}_\tau )  - \cL^\B_\tau( \hat \btheta_\tau^\B )   \} }
$
is close to $n^{-1/2} \|   \bxi^\B  \|_2$ with high probability as long as the dimension $d$ of the parameter
space satisfies the condition that $d^2/n$ is small.

\begin{remark}[{\sf Order of robustification parameter}] \label{rmk.boot.order}
{\rm
Similar to Remark~\ref{new.tau}, now with finite fourth moment $\upsilon_4$, the robustification parameter in Theorems~\ref{boot.concentration.thm} and \ref{boot.wilks.thm} can be chosen as
\begin{align}
	\tau = v  \{ n/(d+t) \}^{\eta} ~~\mbox{ for any } \eta \in [1/4 , 1/2) ~\mbox{ and }~ v\geq \upsilon_4^{1/4}, \label{new.boot.tau}
\end{align}
such that the same conclusions remain valid. Due to Lemma~\ref{cov.concentration} in the supplemental material, here we require $\eta$ to be strictly less than $1/2$. }
\end{remark}

The next result validates the approximation of the distribution of $\cL_\tau( \btheta^* )  - \cL_\tau( \hat \btheta_\tau  )$ by that of $\cL^\B_\tau( \hat{\btheta}_\tau )  - \cL^\B_\tau( \hat \btheta_\tau^\B )$ in the Kolmogorov distance. Recall that $\PP^*(\cdot) = \PP(\cdot \, | \mathcal{D}_n)$ denotes the conditional probability given the observed data $\mathcal{D}_n = \{ (Y_i, \bX_i) \}_{i=1}^n$.

\begin{theorem}
 \label{Boot.consistency}
Suppose Assumption~\ref{moment.cond} holds with $\delta=2$ and Condition~\ref{weight.cond} holds with $U \sim \mathcal{N}(0,1)$. For any $t>0$ and $v \geq \upsilon_4^{1/4}$, let $\tau=v (\frac{n}{d+t})^{\eta}$ for some $\eta \in [1/4, 1/2)$. Then, with probability (over $\mathcal{D}_n$) at least $1- 6 e^{-t}$, it holds for any $z \geq 0$ that
\begin{align}
 | \PP \{  \cL_\tau(\btheta^*) - \cL_\tau(\hat{\btheta}_\tau) \leq z \} -
	\PP^* \{  \cL_\tau^\B(\hat{\btheta}_\tau) - \cL_\tau^\B(\hat{\btheta}^\B_\tau) \leq z  \}  | \leq \Delta_1(n,d,t) , \label{Boot.unif.bound}
\end{align}
where
$$
	\Delta_1(n,d,t)= C  \{  d^{3/2} n^{-1/2}   +   d^{1/2} \{(d+t)/n\}^{1-2\eta}  + (d+t)^{3\eta} n^{1/2-3\eta} \}   + 7 e^{-t}
$$	
with $C = C(A_0, \sigma, \upsilon_4, v)>0$.
\end{theorem}

Theorem~\ref{Boot.consistency} is in parallel with and can be viewed as a partial extension of Theorem~2.1 in \cite{SZ2015} to the case of heavy-tailed errors. In particular, taking $\eta =1/4$ in Theorem~\ref{Boot.consistency}  we see that the error term scales as $(d^3/n)^{1/4}$, while in \cite{SZ2015} it is of order $(d^3/n)^{1/8}$. The difference is due to the fact that the latter allows misspecified models as discussed in Remark~A.2 therein. In some way, allowing asymmetric and heavy-tailed errors can be regarded as a particular form of misspecification, considering that the OLS is the maximum likelihood estimator at the normal model.

\begin{remark}[{\sf Asymptotic result}] \label{rmk1}
{\rm
To make asymptotic statements, we assume $n\to \infty$ with an understanding that $d = d(n)$ depends on $n$ and possibly $d \to \infty$ as $n\to \infty$. Theorem~\ref{Boot.consistency} can be used to show the bootstrap consistency, where the notion of consistency is the one that guarantees asymptotically valid inference. Specifically, it shows that when the dimension $d$, as a function of $n$, satisfies $d=o(n^{1/3})$, then with $\tau \asymp (\frac{n}{d+\log n})^{\eta}$ for some $\eta \in [1/4, 1/2)$, it holds
$$
	\sup_{z\geq 0}  | \PP\{  \cL_\tau(\btheta^*) - \cL_\tau(\hat{\btheta}_\tau) \leq z \} -
	\PP^*\{  \cL_\tau^\B(\hat{\btheta}_\tau) - \cL_\tau^\B(\hat{\btheta}^\B_\tau) \leq z   \} | = o_{\PP}(1)
$$
as $n\to \infty$.
}
\end{remark}

For any $\alpha \in (0,1)$, let
\begin{align}
	z_\alpha^\B := \inf \{ z \geq 0: \PP^*\{  \cL_\tau^\B(\hat{\btheta}_\tau) - \cL_\tau^\B(\hat{\btheta}^\B_\tau) >z \} \leq \alpha \}  \label{def:z*}
\end{align}
be the upper $\alpha$-quantile of $\cL_\tau^\B(\hat{\btheta}_\tau) - \cL_\tau^\B(\hat{\btheta}^\B_\tau) $ under $\PP^*$, which serves as an approximate to the target value $z_\alpha$ given in \eqref{def:zalpha}. As a direct consequence of Theorem~\ref{Boot.consistency}, the following result formally establishes the validity of the multiplier bootstrap for adaptive Huber regression with heavy-tailed error.

\begin{theorem}[{\sf Validity of multiplier bootstrap}] \label{Boot.validity}
Assume the conditions of Theorem~\ref{Boot.consistency} hold and take $\eta = 1/4$. Then, for any $\alpha \in (0,1)$,
\begin{align}
| \PP \{ \cL_\tau(\btheta^*) - \cL_\tau(\hat{\btheta}_\tau) > z^\B_\alpha\}  - \alpha  |  \leq \Delta_2(n,d, t) ,   \label{Boot.bound}
\end{align}
where $\Delta_2(n,d,t) = C \{ (d+t)^3/n \}^{1/4}+   16 e^{-t} $, where $C = C(A_0, \sigma, \upsilon_4, v)>0$. In particular, taking $\tau \asymp  (\frac{n}{d+\log n})^{1/4}$, it holds
\begin{align}
	\sup_{\alpha \in (0,1)} | \PP \{ \cL_\tau(\btheta^*) - \cL_\tau(\hat{\btheta}_\tau) > z^\B_\alpha \}  - \alpha | = o(1) \nn
\end{align}
provided that $d=d(n)$ satisfies $d= o(n^{1/3})$ as $n\to \infty$.
\end{theorem}

\section{Data-driven procedures for choosing $\tau$}
\label{sec:adap}

The theoretical results in Sections~\ref{sec2.1} and \ref{sec2.3} reveal that how the Huber-type estimator would perform under various idealized scenarios, as such providing guidance on the choice of the key tuning parameter, which is referred to as the robustification parameter that balances bias and robustness.
For estimation purpose, we take $\tau = v(\frac{n}{d+t})^{1/2}$ with $v\geq \sigma$; and for bootstrap inference, we choose $\tau=v(\frac{n}{d+t})^{1/4}$ with $v \geq \upsilon_4^{1/4}$. Since both $\sigma^2=\var(\varepsilon)$ and $\upsilon_4 \geq  \EE(\varepsilon^4)$ are typically unknown in practice, an intuitive approach is replace them by the empirical second and fourth moments of the residuals from the ordinary least squares (OLS) estimator, i.e. $\hat \sigma^2 := (n-d)^{-1}\sn(Y_i - \bX_i^\T \hat{\btheta}_{{\rm ols}})^2$ and $\hat{\upsilon}_4  := (n-d)^{-1}\sn(Y_i - \bX_i^\T \hat{\btheta}_{{\rm ols}})^4$. This simple approach performs reasonably well empirically (see Section \ref{sec.numerical}). However, when heavy tails may be a concern,  $\hat{\sigma}^2$ and $\hat \upsilon_4$ are not good estimates of $\sigma^2$ and $\upsilon_4$. In this section, we discuss two data-dependent methods for choosing the tuning parameter $\tau$: the first one  uses an adaptive technique based on Lepski's method  \citep{Lep1992}, and the second method is inspired by the censored equation approach in \cite{HKW1990} which was originally introduced in pursing a more robust weak convergence theory for self-normalized sums.


\subsection{Lepski-type method}
\label{sec:lepski_tuning}

Borrowing an idea from \cite{M2016}, we first consider a simple adaptive procedure based on Lepski's method.
Let $v_{\min}$ and $v_{\max}$ be some crude preliminary lower and upper bounds for the residual standard deviation, that is, $v_{\min} \leq \sigma \leq  v_{\max}$. For some prespecified $a>1$, let $v_j = v_{\min} a^j$ for $j=0,1,\ldots$ and define
$$
	\cJ = \{ j \in \mathbb Z : v_{\min} \leq v_j < a v_{\max} \}.
$$
It is easy to see that the set $\cJ$ has its cardinality bounded by $|\cJ | \leq 1+ \log_a(v_{\max}/v_{\min})$. Accordingly, we define a sequence of candidate parameters $\{ \tau_j = v_j (\frac{n}{d+t})^{1/2}, j\in \cJ\}$ and let $\hat \btheta^{(j)}$ be the Huber estimator with $\tau=\tau_j$. Set
\begin{align}
	\hat j_{{\rm L}} :=  \min\bigg\{   j \in \cJ : \| \hat \btheta^{(k)} - \hat \btheta^{(j)}  \|_2  \leq  c_0  v_k \sqrt{\frac{d+t}{n}} \mbox{ for all }  k \in \cJ \mbox{ and } k > j \bigg\}  \label{lepski.choice1}
\end{align}
for some constant $c_0>0$. The resulting adaptive estimator is then defined as $ \hat{\btheta}_{{\rm L}} = \hat \btheta^{(\hat j_{\rm L})}$.

\begin{theorem} \label{thm:lepski}
Assume that $c_0 \geq 2c_1 \underline{\lambda}_{\bSigma}^{-1/2}$ for $c_1>0$ as in Theorem~\ref{br.thm}. Then for any $t>0$,
\begin{align}
	 \|  \hat{\btheta}_{ {\rm L}}   - \btheta^* \|_2 \leq     \frac{3a}{2}  c_0  \sigma \sqrt{\frac{d+t}{n}} \nn
\end{align}
with probability at least $1-3\log_a (a v_{\max}/v_{\min}) e^{-t} $, provided $n\gtrsim d+t$.
\end{theorem}

Lepski's adaptation method serves a general technique to select the ``best" estimator from a collection of  certified candidates. The selected estimator adapts to the unknown noise level and satisfies near-optimal probabilistic bounds, while the associated parameter is not necessarily the theoretically optimal one. When applied with the bootstrap, Theorem~\ref{Boot.validity} suggests that the dependence on $d/n$ should be slightly adjusted.
Since the reuse of the sample brings a big challenge mathematically, we shall prove a theoretical result  for the data-driven multiplier bootstrap procedure with sample splitting. However, to avoid notational clutter, we state a two-step procedure without sample splitting, but with the assumption that the second step is carried out on an independent sample.

	\medskip
 \noindent
 {\sf A Two-Step Data-Driven Multiplier Bootstrap}.

\vspace{0.2cm}
\noindent
{\sc Step 1}.  Given independent observations $\{ (Y^{(1)}_i, \bX^{(1)}_i) \}_{i=1}^n$ from linear model \eqref{reg.model}, first we produce a robust pilot estimator using Lepski's method. Recall that Lepski's method requires initial crude upper and lower bounds for $\upsilon_4  \geq  \EE(\varepsilon^4)$. Let $\mu_Y = \EE(Y)$ and note that $\upsilon_{Y}  := \EE(Y- \mu_Y)^4 > \upsilon_4$. We shall use the median-of-means (MOM) estimator of $\upsilon_{Y }$ as a proxy, which is tuning-free in the sense that the construction does not depend on the noise level \citep{M2015}. Specifically, we divide the index set $\{ 1,\ldots, n\}$ into $m \geq 2$ disjoint, equal-length groups $G_1, \ldots, G_m$, assuming $n$ is divisible by $m$. For $j=1,\ldots, m$, compute the empirical 4th moment  evaluated over observations in group $j$: $\hat \upsilon_{Y,j}  =(1/|G_j|)\sum_{i\in G_j}  \{ Y^{(1)}_i - \bar Y^{(1)}_{G_j}\}^4$ with $\bar Y^{(1)}_{G_j} =(1/|G_j|)\sum_{i\in G_j}  Y^{(1)}_i$. The MOM estimator of $\upsilon_{Y }$ is  then defined by $\hat \upsilon_{Y , {\rm mom} }    = {\rm median}\{\hat \upsilon_{Y,1}  , \ldots, \hat \upsilon_{Y,m}  \}$.

Take $v_{\max} =  ( 2 \hat \upsilon _{Y , {\rm mom} } )^{1/4}$ and $v_{\min} = a^{-K} v_{\max} $ for some integer $K \geq 1$ and $a>1$. Denote $v_j =  a^jv_{\min}$ for $j=0, 1, \ldots$, so that $ \cJ = \{ j \in \mathbb Z : v_{\min} \leq v_j <a v_{\max} \} = \{ 0, 1, \ldots , K \}$. 
Slightly different from above, now we consider a sequence of parameters $\{ \tau_j = v_j (\frac{n}{d+\log n})^{1/4}\}_{ j\in \cJ}$ and let $\wt \btheta^{(j)}$ be the Huber estimator with $\tau=\tau_j$. Set
\begin{align}
	 \tilde  j  :=  \min\bigg\{   j \in \cJ : \| \wt \btheta^{(k)} - \wt \btheta^{(j)}  \|_2  \leq  c_0  v_k \sqrt{\frac{d+\log n}{n}} \mbox{ for all }  k \in \cJ \mbox{ and } k > j \bigg\}  \label{lepski.choice2}
\end{align}
for some constant $c_0>0$. Denote by $ \hat{\btheta}^{(1)} =\wt \btheta^{( \tilde j )}$ the corresponding estimator and put $\hat \tau = \tau_{  \tilde   j } $.

\vspace{0.2cm}
 \noindent
{\sc Step 2}. Taking $\hat{\btheta}^{(1)}$ and $\hat \tau$ from Step 1, next we apply the multiplier bootstrap procedure to a new sample $  (  Y_i^{(2)},  \bX^{(2)}_i) \}_{i=1}^n$ that is independent from the previous one. Similarly to \eqref{huber.est} and \eqref{bootHuber.est}, define
\begin{align}
	\hat{\btheta}  \in \argmin_{\btheta \in \RR^d}     \hat \cL (\btheta) ~~\mbox{ and  }~~ \hat{\btheta}^\B  \in  \argmin_{ \btheta \in \RR^d: \| \btheta - \hat {\btheta}^{(1)}  \|_2 \leq  \hat R }   \hat{ \mathcal{L}}^\B(\btheta) ,
\end{align}
where $ \hat \cL  (\btheta) =  \sn  \ell_{\hat \tau }(  Y^{(2)}_i -  \langle \bX_i^{(2)} , \btheta \rangle )$, $\hat {\mathcal{L}}^\B(\btheta) = \sn W_i \, \ell_{\hat \tau}(  Y_i^{(2)} - \langle  \bX_i^{(2)} ,  \btheta \rangle )$ and $\hat R = \hat \tau (\frac{d+\log n}{n})^{1/4}$. With the above preparations, we apply Algorithm~\ref{algo:huber_inference} to construct the confidence set $\hat{\mathcal{C}}_\alpha = \{\btheta \in \RR^d :   \hat \cL ( \btheta)  -\hat  \cL ( \hat \btheta  )   \leq \hat z^\B_{ \alpha}  \}$, where
  \begin{equation} 
      \hat     z^\B_{ \alpha} =\inf \{z \geq 0: \mathbb{P}\{  \hat \cL^\B( \hat{\btheta}  )  - \cL^\B ( \hat \btheta^\B )   > z | \bar{\mathcal{D}}_n    \} \leq \alpha \} \nn
 \end{equation}
 with $\bar{\mathcal{D}}_n = \{( Y_i^{(1)} , \bX_i^{(1)} ) , (Y_i^{(2) } , \bX_i^{(2)})\}_{i=1}^n$.

\begin{theorem} \label{thm:two-step}
Assume $\bar{\mathcal{D}}_n$ is an independent sample from $(Y,\bX)$ satisfying Condition~\ref{moment.cond} and moreover, $ \EE(|\varepsilon|^{4+\delta}) \leq \upsilon_{4+\delta}$ for some $\delta>0$. Let $W_1, \ldots, W_n$ be IID $\mathcal{N}(1,1)$ random variables that are independent of $\bar{\mathcal{D}}_n$.  Assume further that $d=d(n)$ satisfies $d=o(n^{1/3})$ as $n\to \infty$.
 Then, for any $\alpha \in (0,1)$, the confidence set $\hat{\mathcal{C}}_\alpha$ obtained by the two-step multiplier bootstrap procedure with  $m = \lfloor 8 \log n +1 \rfloor$ and $K \geq \lfloor \log_a(3\upsilon_Y /\upsilon_4)^{1/4} \rfloor +1$ satisfies $\mathbb P(\btheta^* \in \hat{\mathcal{C}}_\alpha) \to 1-\alpha$ as $n\to \infty$.
\end{theorem}

The proof of Theorem \ref{thm:two-step} will be provided in Section \ref{sec:proof-two-step} in the supplementary material.

\subsection{Huber-type method}
\label{sec:huber_tuning}
In Huber's original proposal, robust location estimation  with desirable  efficiency also depends on the scale parameter $\sigma$.
For example, in Huber's Proposal 2 \citep{H1964}, the location $\mu$ and scale $\sigma$ are estimated simultaneously by solving a system of "likelihood equations". 
Similarly in spirit, we propose a new data-driven tuning scheme to calibrate $\tau$ by solving a so-called censored equation \citep{HKW1990} instead of likelihood equation. We first consider mean estimation to illustrate the main idea, and then move forward to the regression problem. Due to space limitations, we leave some discussions and proofs of the theoretical  results to Appendix~\ref{app:D} in the supplemental material. 

\subsubsection{Motivation: truncated mean}
\label{sec:motiv}
Let $X_1,\ldots, X_n$ be IID random variables from $X$ with mean $\mu$ and variance $\sigma^2>0$. Without loss of generality, we first assume $\mu=0$.
\cite{C2012} proved that  the worst case deviations of the sample mean $\bar{X}_n$ are suboptimal with heavy-tailed data (see Appendix~\ref{sec:catoni}). To attenuate the erratic fluctuations in $\bar{X}_n$, it is natural to consider the truncated sample mean
\begin{align}
	\hat m_\tau = \frac{1}{n} \sn \psi_\tau(X_i) ~\mbox{ for some } \tau>0,
\end{align}
where
\begin{equation}\label{eq:psi}
\psi_\tau(u) := \ell'_\tau(u)=\sgn(u) \min(|u|, \tau), \quad u\in \RR,
\end{equation}
and  $\tau$ is a tuning parameter that balances between bias and robustness. To see this, let $\mu_\tau = \EE (\hat  m_\tau)$ be the truncated mean.  By Markov's inequality, the bias term can be controlled by
\begin{align}
	 | \mu_\tau |  & = | \EE \{ X - \sgn(X) \tau \}I(|X| >\tau)  | \nn \\
	& \leq   \EE (|X| - \tau )I(|X| >\tau)  \nn \\
	& \leq \frac{ \EE  ( X^2 - \tau^2 ) I(|X|>\tau) }{\tau} \leq  \frac{\sigma^2 - \EE \psi^2_\tau(X)}{\tau} . \label{bias.bound}
\end{align}
The robustness of $ \hat m_\tau$, on the other hand, can be characterized via the deviation
$$
 | \hat m_\tau - \mu_\tau |	 = \bigg| \frac{1}{n} \sn \psi_\tau(X_i) - \mu_\tau \bigg|.
$$
The following result shows that with a properly chosen $\tau$, the truncated sample mean achieves a sub-Gaussian performance under the finite variance condition. Moreover, uniformity of the rate over
a neighborhood of the optimal tuning scale requires an additional $\log(n)$-factor. For every $\tau>0$, define the truncated second moment
\begin{align}
	\sigma_\tau^2 = \EE \{ \psi_\tau^2(X) \} = \EE \{ \min(X^2, \tau^2) \} .
\end{align}

\begin{proposition} \label{prop2}
{\rm
For any $1\leq  t <  n \PP(|X|>0)$, let $\tau_t >0$ be the solution to
\begin{align}
	 \frac{ \EE \{  \psi_\tau^2(X) \} }{\tau^2} = \frac{t}{n} , \ \ \tau>0.   \label{population.tau}
\end{align}
\begin{enumerate}
\item[(i)] With probability at least $1- 2e^{-t}$, $\hat m_{\tau_t}$ satisfies
\begin{align}  \label{single.concentration}
	 | \hat m_{\tau_t}  - \mu_{\tau_t} | \leq  1.75 \sigma_{\tau_t}  \sqrt{\frac{t}{n}} ~\mbox{ and }~ | \hat m_{\tau_t}  | \leq  \bigg(  0.75 \sigma_{\tau_t}  + \frac{\sigma^2}{\sigma_{\tau_t}}  \bigg) \sqrt{\frac{t}{n}} .
\end{align}

\item[(ii)] With probability at least $1- 2e^{\log n-t}$,
\begin{align} \label{uniform.concentration}
	 \max_{\tau_t/2  \leq \tau \leq  3\tau_t /2 } |  \hat m_\tau  | \leq C_t \sqrt{\frac{t}{n}} + \frac{\sigma_{\tau_t}}{\sqrt{n}} ,
\end{align}
where $C_t := \sup_{\sigma_{\tau_t}/2 \leq c\leq 3\sigma_{\tau_t}/2}  \{  \sigma_{c(n/t)^{1/2}}   \sqrt{2} -  c^{-1}\sigma_{c(n/t)^{1/2}}^2 + c/3+ c^{-1} \sigma^2   \} \leq \sqrt{2} \sigma +   2\sigma^2/\sigma_{\tau_t}  +   \sigma_{\tau_t}/6$.
\end{enumerate}

%

}
\end{proposition}

The next result establishes existence and uniqueness of the solution to equation \eqref{population.tau}.

\begin{proposition}  \label{prop:exist}
{\rm
(i). Provided $0< t <  n \PP(|X|>0)$, equation~\eqref{population.tau} has a unique solution, denoted by $\tau_t$, which satisfies
$$
	  \{ \EE(X^2 \wedge q_{t/n}^2) \}^{1/2} \sqrt{\frac{n}{t}} \leq  \tau_t  \leq  \sigma \sqrt{\frac{n}{t}} ,
$$
where $q_\alpha := \inf\{ z : \PP(|X|>z)\leq \alpha \}$ is the upper $\alpha$-quantile of $|X|$. (ii). Let $t= t_n >0$ satisfy $t_n \to \infty$ and $t=o(n)$. Then $\tau_t \to \infty$, $\sigma_{\tau_t} \to \sigma$ and $\tau_t \sim \sigma \sqrt{n/t}$ as $n\to \infty$.
}
\end{proposition}


According to Proposition~\ref{prop2}, an ideal $\tau$ is such that the sample mean of truncated data $\psi_\tau(X_1), \ldots, \psi_\tau(X_n)$ is tightly concentrated around the true mean.  At the same time, it is reasonable to expect that the empirical second moment of $\psi_\tau(X_i)$'s  provides an adequate estimate of $\sigma_\tau^2 = \EE \{\psi_\tau^2(X)\}$. Motivated by this observation, we propose to choose $\tau$ by solving the equation
$$
	\tau = \bigg\{\frac{1}{n} \sn \psi^2_\tau(X_i)  \bigg\}^{1/2} \sqrt{\frac{n}{t}} , \ \ \tau>0,
$$
or equivalently, solving
\begin{align}	
	\frac{1}{n} \sn \frac{\psi_\tau^2(X_i)}{\tau^2}  = \frac{t}{n} , \ \ 	\tau>0 .  \label{sample.tau}
\end{align}
Equation \eqref{sample.tau} is the sample version of \eqref{population.tau}.
Provided the solution exists and is unique,  denoted by $\hat{\tau}_t$, we obtain a data-driven estimator
\begin{align}	
	\hat  m_{\hat{\tau}_t}= \frac{1}{n} \sn \sgn(X_i)  \min( |X_i | ,  \hat{\tau}_t ) .	\label{hat.Mn}
\end{align}


As a direct consequence of Proposition~\ref{prop:exist}, the following result ensures existence and uniqueness of the solution to equation \eqref{sample.tau}.

\begin{proposition} \label{prop:exist2}
{\rm
Provided $0<t < \sn I(|X_i|>0)$, equation~\eqref{sample.tau} has a unique solution.}
\end{proposition}

Throughout, we use $\hat \tau_t$ to denote the solution to equation \eqref{sample.tau}, which is unique and positive when $t< \sn I(|X_i|>0)$. For completeness, we set $\hat \tau_t =0$ if $t \geq \sn I(|X_i|>0)$. If $\PP(X=0)=0$, then $\hat \tau_t >0$ with probability one as long as $0<z<n$. In the special case of $t=1$, since $\psi_{\hat \tau_1} (X_i) = X_i$ for all $i=1,\ldots, n$, equation \eqref{sample.tau} has a unique solution $\hat \tau_1 =  (\sn X_i^2)^{1/2}$.
With both $\tau_t$ and $\hat \tau_t$ well defined, next we investigate the statistical property of $\hat \tau_t$.

\begin{theorem} \label{thm1}
{\rm
Assume  that $ \var(X)<\infty$ and $\PP(X=0)=0$. For any $1\leq t<n$ and $0<r<1$, we have
\begin{align}
 \pr( | \hat{\tau}_t /\tau_t  -1 | \geq r )  \leq  e^{ - a_1^2 r^2 t^2 /( 2 t + 2 a_1 r t /3 ) }  + e^{- a_2^2 r^2 t /2} +  2e^{-(a_1 \wedge a_2)^2 t/8  }  , \label{tau.tail.prob}
\end{align}
where
\begin{align}  \label{def.c+-}
	a_1 = a_1(t,r) = \frac{P(\tau_t)}{2Q(\tau_t)} \frac{2+r}{(1+r)^2} ,  \ \  a_2 = a_2(t,r) =\frac{P(\tau_t -\tau_t r)}{2Q(\tau_t)} \frac{2-r}{1-r},
\end{align}
where  $P(z) = \EE \{ X^2 I(|X|\leq z) \}$ and $Q(z) = \EE   \{\psi^2_z(X)\}$ for $z \geq 0$.
}
\end{theorem}
More properties of functions $P(z)$ and $Q(z)$ can be found in Appendix~\ref{sec:pre} in the supplement.

\begin{remark} \label{rmk1}
{\rm
We discuss some direct implications of Theorem~\ref{thm1}.
\begin{enumerate}
\item[(i)] Let $t = t_n\geq 1$ satisfy $t \to \infty$ and $t =o(n)$ as $n\to \infty$. By Proposition~\ref{prop:exist}, $\tau_t \to \infty$, $\sigma_{\tau_t}\to \sigma$ and $\tau_t   \sim  \sigma\sqrt{ n /t }$, which further implies $P(\tau_t) \to  \sigma$ and $Q(\tau_t) \to  \sigma$ as $n\to \infty$.

\item[(ii)] With $r=1/2$ and $t= (\log n)^{1+\kappa }$ for some $\kappa >0$ in \eqref{tau.tail.prob}, the constants $a_1 =a_1(t , 1/2)$ and $a_2 = a_2(t, 1/2)$ satisfy $a_1 \to 5/9$ and $a_2  \to 3/2$ as $n\to \infty$. The resulting $\hat{\tau}_t$ satisfies that with probability approaching one, $\tau_t /2   \leq \hat{\tau}_t\leq 3 \tau_t/2$.
\end{enumerate}
}
\end{remark}

The following result, which is a direct consequence of \eqref{uniform.concentration}, Theorem~\ref{thm1} and Remark~\ref{rmk1}, shows that the data-driven estimator $\hat m_{\hat \tau_t}$ with $t= (\log n)^{1+\kappa}$ ($\kappa >0$) is tightly concentrated around the mean with high probability.

\begin{corollary}\label{cor:hat_m}
{\rm 
Assume the conditions of Theorem~\ref{thm1} hold. Then, the truncated mean $\hat m=\hat m_{\hat \tau_t}$ with $t= (\log n)^{1+\kappa}$ for some $\kappa>0$ satisfies $| \hat m | \leq c_1 \sqrt{(\log n)^{1+\kappa}/n}$ with probability greater than $1- c_2 n^{-1}$ as $n\to \infty$, where $c_1, c_2>0$ are constants independent of $n$.
}
\end{corollary}

\subsubsection{Huber's mean estimator}
\label{sec:hub.mean}

For the truncated sample mean, even with the theoretically optimal tuning parameter, the deviation of the estimator only scales with the second moment rather than the ideal scale $\sigma$.  Indeed, the truncation method described above primarily serves as a heuristic device and paves the way for developing data-driven Huber estimators.

Given IID samples  $X_1,\ldots, X_n$ with mean $\mu$ and variance $\sigma^2$, recall the Huber estimator $\hat \mu_\tau = \argmin_\theta \sn \ell_\tau(X_i - \theta)$, which is also the unique solution to
\begin{align}
 \frac{1}{n} \sn \psi_\tau(X_i - \theta ) = 0 , \ \ \theta \in \RR. \label{huber.foc}
\end{align}
The non-asymptotic property of $\hat \mu_\tau$ is characterized by a Bahadur-type representation result developed in \cite{ZBFL2017}: for any $t>0$, $\hat \mu_\tau$ with $\tau = \sigma \sqrt{n/t}$ satisfies the bound $|\hat \mu_\tau - \mu - (1/n) \sn \psi_\tau(\varepsilon_i) | \leq  c_1 \sigma  t/\sqrt{n}$ with probability at least $1-3e^{-t}$ provided $n\geq c_2t$, where $c_1,c_2>0$ are absolute constants and $\varepsilon_i = X_i - \mu$ are noise variables. In other words, a properly chosen $\tau$ is such that the truncated average $ (1/n) \sn \psi_\tau(\varepsilon_i)$ is resistant to outliers caused by a heavy-tailed `noise'. Similar to \eqref{sample.tau}, now we would like to choose the robustification parameter by solving
\begin{align}
	\frac{1}{n} \sn \frac{\psi_\tau^2(\varepsilon_i)}{\tau^2} = \frac{t}{n} ,  \label{var.eqn}
\end{align}
which is practically impossible as $\varepsilon_i$'s are unobserved realized noise. In light of \eqref{huber.foc} and \eqref{var.eqn}, and motivated by Huber's Proposal 2 [page 96 in \cite{H1964}] for the simultaneous estimation of location and scale, we propose to estimate $\mu$ and calibrate $\tau$ by solving the following system of equations
\begin{align*}
\begin{cases}
	     \sn \psi_\tau(X_i - \theta) =0 , \\
 \frac{1}{n} \sn   \frac{ \psi_\tau^2(X_i - \theta ) }{\tau^2}  - \frac{t}{n} = 0,
\end{cases}   \theta \in \RR , \, \tau >0 .
\end{align*}
This method of simultaneous estimation can be naturally extended to the regression setting, as discussed in the next section.

A different while comparable proposal is a two-step method, namely $M$-estimation of $\mu$ with auxiliary robustification parameter computed separately by solving
$$
	\frac{1}{{n\choose 2}} \sum_{1\leq i<j\leq n} \frac{\min\{ (X_i- X_j )^2/2, \tau^2 \}}{\tau^2} = \frac{t}{n}.
$$
It is, however, less clear that how this method can be generalized to the regression problem. Therefore, our focus will be on the previous approach.

\subsubsection{Data-driven Huber regression}
\label{sec:data-adp-huber}
Consider the linear model $Y_i = \bX_i^\T \btheta^* + \varepsilon_i$ as in \eqref{reg.model} and the Huber estimator $\hat \btheta_\tau = \argmin_{\btheta \in \RR^d} \cL_\tau(\btheta)$, where $\cL_\tau(\btheta) = \sn \ell_\tau(Y_i - \bX_i^\T \btheta )$.
From the deviation analysis in \eqref{sec2.1} we see that to achieve the sub-Gaussian performance bound, the theoretically desirable tuning parameter for $\hat \btheta_\tau$ is $\tau \sim \sigma \sqrt{n/(d+t)}$ with $\sigma^2 = \var(\varepsilon_i)$.
Further, by the Bahadur representation \eqref{Bahadur.representation},
\begin{align}
	\hat \btheta_\tau  - \btheta^* =  \frac{1}{n} \sn \psi_\tau(\varepsilon_i)  \,\bSigma^{-1}\bX_i  + \cR_\tau, \nn
\end{align}
where the remainder $ \cR_\tau$ is of the order $\sigma (d+t)/n$ with exponentially high probability. This result demonstrates that the robustness is essentially gained from truncating the errors. Motivated by this representation and our discussions in Section~\ref{sec:motiv}, a robust tuning scheme is to find $\tau$ such that
\begin{align}
	\tau = \bigg\{ \frac{1}{n} \sn \psi_\tau^2(\varepsilon_i) \bigg\}^{1/2} \sqrt{\frac{n}{d+t}} , \ \ \tau>0. \label{eqn.residual}
\end{align}
Unlike the mean estimation problem, the realized noises $\varepsilon_i$ are unobserved. It is therefore natural to calibrate $\tau$ using fitted residuals. On the other hand, for a given $\tau >0$, the Huber loss minimization is equivalent to the following least squares problem with variable weights:
\begin{equation}
	\min_{w_i \geq 0, \; \btheta} \;\;\; \sn  \bigg\{ \frac{ (Y_i - \bX_i^\T \btheta )^2 }{w_i + 1} + \tau^2  w_i  \bigg\}, \label{wls}
\end{equation}
where the minimization is over $w_i \geq 0$ and $\btheta \in \RR^d$. This equivalence can be derived by writing down the KKT conditions of \eqref{wls}. Details will be provided in Remark \ref{rmk.equiv} below. By \eqref{wls}, this problem can be solved via the iteratively reweighted least squares method.

To summarize, we propose an iteratively reweighted least squares algorithm, which starts at iteration 0 with an initial estimate $\btheta^{(0)} = \hat \btheta_{{\rm ols}}$ (the least squares estimator) and involves three steps at each iteration.

\noindent
{\it Calibration:} Using the current estimate $\btheta^{(k)}$, we compute the vector of residuals $\bm{R}^{(k)} = (R_1^{(k)}, \ldots, R_n^{(k)} )^\T$, where $R_i^{(k)}  = Y_i -  \bX_i^\T\btheta^{(k)}$. Then we take $\tau^{(k)}$ as the solution to
\begin{align}
  \frac{1}{n} \sn \frac{\min\{ R_i^{(k)2}, \tau^2 \}}{\tau^2} = \frac{d+t}{n} , \ \ \tau>0.
  \label{eq:huber_fixed}
\end{align}
By Proposition~\ref{prop:exist2}, this equation has a unique positive solution provided $d+t < \sn I( |R_i^{(k)}|>0 )$.

\noindent
{\it Weighting:}  Compute the vector of weights ${\bm w}^{(k)} = (w_1^{(k)}, \ldots, w_n^{(k)})^\T$, where $w_i^{(k)} = |R_i^{(k)}|/\tau^{(k)} - 1 $ if $|R_i^{(k)} | > \tau^{(k)}$ and $w_i^{(k)} = 0$ if $|R_i^{(k)}|\leq \tau^{(k)} $. Then define the diagonal matrix $\textbf{W}^{(k)}={\rm diag}( (1+w_1^{(k)})^{-1}, \ldots, ( 1+ w_n^{(k)}  )^{-1} )$.

\noindent
{\it Weighted least squares:} Solve the weighted least squares problem \eqref{wls} with $w_i = w_i^{(k)}$ and $\tau = \tau^{(k)}$ to obtain
$$
	\btheta^{(k+1)}  =  ( \textbf{X}^\T \textbf{W}^{(k)} \textbf{X} )^{-1} \textbf{X}^\T {\bm Y} ,
$$
where $\textbf{X} = (\bX_1, \ldots, \bX_n)^\T \in \RR^{n \times d}$ and ${\bm Y} = (Y_1,\ldots, Y_n)^\T$.

Repeat the above three steps until convergence or until the maximum number of iterations is reached.

In addition, from Theorems~\ref{boot.concentration.thm}--\ref{Boot.consistency} we see that the validity of the multiplier bootstrap procedure requires a finite fourth moment condition, under which the ideal choice of $\tau$ is $\{  \upsilon_4   n /(d+t) \}^{1/4}$. To construct data-dependent robust bootstrap confidence set, we adjust equation \eqref{eq:huber_fixed} by replacing $R_i^{(k)2}$  and $\tau^2$ therein with $R_i^{(k) 4}$ and $\tau^4$, and solve instead
\begin{align}
	\frac{1}{n} \sn \frac{\min\{ R_i^{(k)4}, \tau^4 \}}{\tau^4} = \frac{d+t}{n} , \ \ \tau>0.     \label{tuning-free.2}
\end{align}
Keep the other two steps and repeat until convergence or the maximum number of iterations is reached.  Let $\hat{\btheta}_{\hat{\tau}}$ and $\hat{\tau}$ be the obtained solutions. Then, we apply Algorithm~\ref{algo:huber_inference} with $\tau = \hat{\tau}$ to construct confidence sets.

Finally we discuss the choice of $t$. 
Since $t$ appears in both the deviation bound and confidence level, we let $t = t_n$ slowly grow with the sample size to gain robustness without compromising unbiasedness. We take $t = \log n$, a typical slowly growing function of $n$, in all the numerical experiments carried out in this paper.

\begin{remark}[{\sf Equivalence between \eqref{wls} and Huber regression}] \label{rmk.equiv}
{\rm	For a given $\btheta$ in \eqref{wls}, define $R_i=Y_i - \bX_i^\T \btheta$, $i=1,\ldots,n$. The KKT condition of the program \eqref{wls} with respect to each $w_i$ under the constraint $w_i \geq 0$ now reads:
	\begin{align*}
	-\frac{R_i^2}{(w_i+1)^2} +\tau^2 -\lambda_i =0; \;\; w_i \geq 0, \;\;  \lambda_i \geq 0;  \;\;  \lambda_i w_i =0,
	\end{align*}
	where $\lambda_i$ is the Lagrangian multiplier. The solution to the KKT condition takes the form:
	\begin{align*}
	&w_i =\frac{|R_i|}{\tau}-1, \;\; \lambda_i=0  \; &  \mathrm{if} \quad |R_i|\geq \tau, \\
	& w_i=0,        \;\; \lambda_i=\tau^2-R_i^2 \; &   \mathrm{if}\quad |R_i|< \tau.
	\end{align*} 
This gives the optimal solution of $w_i$. By plugging the optimal solution of $w_i$ back into \eqref{wls}, we obtain the following optimization with respect to $\btheta$:
	\begin{equation*}
		\min_{\btheta}  
			\sum_{i=1}^n \left(2\tau |Y_i - \bX_i^\T \btheta|-\tau^2\right)  I(|Y_i - \bX_i^\T \btheta| \geq \tau)+   |Y_i - \bX_i^\T \btheta|^2 I(|Y_i - \bX_i^\T \btheta| < \tau),
	\end{equation*}
	which is equivalent to Huber regression.}
\end{remark}

\section{Multiple inference with multiplier bootstrap calibration}
\label{sec3}

In this section, we apply the adaptive Huber regression with multiplier bootstrap to simultaneously test the hypotheses in \eqref{hypotheses}. Given a random sample $(\by_1,\bx_1),\ldots, (\by_n, \bx_n)$ from the multiple response regression model \eqref{panel.data}, we define robust estimators
\begin{align} \label{robust.mle}
	(\hat{\mu}_k , \hat{\bbeta}_k) \in  \argmin_{\mu \in \RR , \, \bbeta \in \RR^s} \sn \ell_{\tau_k}(y_{ik} - \mu - \bx_i^\T \bbeta ), \ \ k=1,\ldots, m,
\end{align}
where $\tau_k$'s are robustification parameters.

To conduct simultaneous inference for $\mu_k$'s, we use the multiplier bootstrap to approximate the distribution of $\hat{\mu}_k-\mu_k$.  Let $W$ be a random variable with unit mean and variance. Independent of $\{ (\by_i, \bx_i) \}_{i=1}^n$, let $\{ W_{ik}, 1\leq i\leq n, 1\leq k\leq m\}$ be IID from $W$. Define the multiplier bootstrap estimators
\begin{align}
	(\hat{\mu}_k^\B,  \hat{\bbeta}^\B_k) \in  \argmin_{ {\btheta} = (\mu,  \, \bbeta^\T)^\T : \atop \| \btheta - \hat{\btheta}_k \|_2 \leq R_k } \sn W_{ik} \ell_{\tau_k}(y_{ik} - \mu - \bx_i^\T \bbeta ), \ \ k=1,\ldots, m, \label{boot.est}
\end{align}
where $\hat{\btheta}_k = (\hat{\mu}_k , \hat{\bbeta}_k^\T)^\T$ and $R_k$'s are radius parameters. We will show that the unknown distribution of $\sqrt{n} \,( \hat{\mu}_k - \mu_k )$ can be approximated by the conditional distribution of $\sqrt{n} \, ( \hat{\mu}_k^\B - \hat{\mu}_k )$ given $\mathcal{D}_{kn} := \{ (y_{ik}, \bx_i) \}_{i=1}^n$.

The main result of this section establishes validity of the multiplier bootstrap on controlling the FDP in multiple testing. For $k=1,\ldots, m$, define test statistics $\hat{T}_k = \sqrt{n} \, \hat{\mu}_k$ and the corresponding bootstrap $p$-values $p^\B_k = G^\B_k(|\hat{T}_k|)$, where $G^\B_k(z) := \PP (   \sqrt{n} \, | \hat{\mu}_k^\B - \hat{\mu}_k | \geq z   | \mathcal{D}_{kn}  )$, $z\geq 0$. For any given threshold $t\in (0,1)$, the false discovery proportion is defined as
\begin{align}
	\FDP(t) = V(t) / \max\{R(t) , 1\} , \label{def:FDP}
\end{align}
where $V(t) = \sum_{k\in \mathcal{H}_0} I(p^\B_k \leq t)$ is the number of false discoveries, $R(t) = \sum_{k=1}^m I(p^\B_k \leq t)$ is the number of total discoveries and $\mathcal{H}_0 := \{ k : 1\leq k\leq m, \mu_k = 0 \}$ is the set of true null hypotheses. For any prespecified $\alpha \in (0,1)$, applying the the Benjamini and Hochberg (BH) method \citep{BH1995} to the bootstrap $p$-values $p^\B_1,\ldots, p^\B_m$ induces a data-dependent threshold
\begin{align} \label{BH.threshold}
	t^\B_{{\rm BH}} = p^\B_{(k^\B)}   ~\mbox{ with }~ k^\B = \max \{ k: 1\leq k\leq m, p^\B_{(k)} \leq \alpha k/m \} .
\end{align}
We reject the null hypotheses for which $p^\B_k \leq t^\B_{{\rm BH}}$.

\begin{assumption}
\label{moment.cond2}
{\rm
$(\by_1, \bx_1), \ldots , (\by_n, \bx_n)$ are IID observations from $(\by, \bx)$ that satisfies $\by =  \bmu +  \bGamma  \bx + \bepsilon$, where $\by=(y_1, \ldots, y_m)^\T$, $\bmu =(\mu_1,\ldots, \mu_m)^\T$, $\bGamma = (\bbeta_1 , \ldots, \bbeta_m)^\T \in \RR^{m\times s}$ and $\bepsilon = (\varepsilon_1, \ldots, \varepsilon_m)^\T$. The random vector $\bx \in \RR^s$ satisfies $\EE(\bx) = \textbf{0}$, $\EE(\bx \bx^\T) = \bSigma$ and $\PP(|\langle \bu, \bSigma^{-1/2} \bx \rangle |\geq t) \leq 2\exp(-t^2 \| \bu \|_2^2/A_0^2)$ for all $\bu \in \RR^s$, $t\in \RR$ and some constant $A_0>0$. Independent of $\bx$, the noise vector $\bepsilon$ has independent elements and satisfies $\EE(\bepsilon) = \textbf{0}$ and
$
	c_l \leq \min_{1\leq k\leq m} \sigma_k \leq \max_{1\leq k\leq m} \upsilon_{k,4}^{1/4} \leq c_u
$
for some constants $c_l , c_u >0$, where $\sigma_k^2 = \EE(\varepsilon_k^2)$ and $\upsilon_{k,4}= \EE(\varepsilon_k^4)$.
}
\end{assumption}

\begin{theorem} \label{FDP.thm}
Assume Condition~\ref{moment.cond2} holds and $m=m(n)$ satisfies $m\to \infty$ and $\log m=o(n^{1/3})$. Moreover, as $(n,m)\to\infty$,
\begin{align} \label{cond.SNR}
	{\rm card}\Big\{  k: 1\leq k\leq m , |\mu_k | / \sigma_k \geq \lambda_0 \sqrt{(2 \log m) / n} \, \Big\} \to \infty
\end{align}
for some $\lambda_0>2$. Then, with
$$
	\tau_k  = v_k  \bigg\{ \frac{n}{s+  2\log(nm) } \bigg\}^{1/3} ~\mbox{ and }~ R_k = v_k \geq \upsilon_{k,4}^{1/4}  , \ \ k=1,\ldots, m ,
$$
in \eqref{robust.mle} and \eqref{boot.est}, it holds
\begin{align} \label{FDP.consistency}
	\frac{\FDP(t_{\BH}^\B)}{(m_0/m)}  \to \alpha ~\mbox{ in probability as } (n,m) \to \infty ,
\end{align}
where $m_0 = {\rm card}(\mathcal{H}_0)$.
\end{theorem}

In practice, conditional quantiles of $\sqrt{n} \, ( \hat{\mu}_k^\B - \hat{\mu}_k )$ can be computed with arbitrary precision by using the Monte Carlo simulations: Independent of the observed data, generate IID random weights $\{ W_{ik,b}, 1\leq i\leq n, 1\leq k\leq m , 1\leq b\leq B \}$ from $W$, where $B$ is the number of bootstrap replications. For each $k$, the bootstrap samples of $(\hat{\mu}_k^\B  , \hat{\bbeta}^\B_k)$ are given by
\begin{equation}\label{eq:boot.est}
	 (\hat{\mu}^\B_{k,b}, \hat{\bbeta}^\B_{k,b}  )   \in \argmin_{   \mu \in \RR , \, \bbeta  \in  \RR^{s} }   \sn W_{ik,b} \,  \ell_{\tau_k}( y_{ik} -  \mu  -  \bx_i^\T \bbeta ) , \ \ b=1,\ldots, B.
\end{equation}
For $k =1,\ldots,  m$, define empirical tail distributions
$$
	G^\B_{k , B}(z) =  \frac{1}{ B+1}  \sum_{b=1}^B I( \sqrt{n} \, |\hat{\mu}^\B_{k,b} - \hat{\mu}_k |  \geq z  ) , \ \ z \geq 0.
$$
The bootstrap $p$-values are thus given by $\{ p^\B_{k, B} =  G^\B_{k , B}(\sqrt{n}\,|\hat{\mu}_k| ) \}_{k=1}^m$, to which either the BH procedure or Storey's procedure can be applied. For the former, we reject $H_{0k}$ if and only if $p^\B_{k, B}  \leq p^\B_{(k_B^\B) , B}$, where $k_B^\B = \max\{ k : 1\leq k \leq m,  p^\B_{(k),B} \leq k \alpha/m\}$ for a predetermined $0<\alpha<1$ and $p^\B_{(1),B} \leq \cdots \leq p^\B_{(m),B}$ are the ordered bootstrap $p$-values. See Algorithm~\ref{algo:huber_multi_infer} for detailed implementations.

\begin{algorithm}[!t]
    \caption{{\small {\sf Huber Robust Multiple Testing}}}
    \label{algo:huber_multi_infer}
    {\textbf{Input:} Data $\{ (\by_i, \bx_i) \}_{i=1}^n$, number of bootstrap replications $B$, thresholding parameters $\{\tau_k\}_{k=1}^m$, nominal level $\alpha \in (0,1)$.}
    \begin{algorithmic}[1]
      \STATE  Solve $m$ Huber regressions in \eqref{robust.mle} and obtain $\{(\hat{\mu}_k , \hat{\bbeta}_k)\}_{k=1}^m$.
      \FOR{$b=1,2 \ldots, B$}
          \STATE Generate IID random weights $\{W_{ik, b}\}_{i \in [n], k \in [m]}$ satisfying $\EE(W_{ik,b})=1$ and $\var(W_{ik , b})=1$.
          \STATE Solve $m$ weighted Huber regressions in \eqref{eq:boot.est} and obtain $\{(\hat{\mu}^\B_{k,b}, \hat{\bbeta}^\B_{k,b})\}_{k=1}^m$.
      \ENDFOR
      \FOR{$k=1, \ldots, m$}
          \STATE Compute the bootstrap $p$-value:
          $$
            p^\B_{k,B}=\frac{1}{ B+1}  \sum_{b=1}^B I( |\hat{\mu}^\B_{k,b} - \hat{\mu}_k |  \geq |\hat{\mu}_k|   ).
          $$
      \ENDFOR

      \STATE Sort the bootstrap $p$-values: $p^\B_{(1),B} \leq \cdots \leq p^\B_{(m),B}.$
      \STATE Compute the BH threshold:
       $
            k_m^\B = \max\{1\leq k \leq m:  p^\B_{(k),B} \leq k \alpha/m\}.
       $
    \end{algorithmic}

 \textbf{Output:} Set $\{ 1\leq k\leq m: p^\B_{k ,B} \leq p^\B_{(k_m^\B ),B}\}$ of rejections, i.e. reject $H_{0k}$ if $p^\B_{k,B} \leq p^\B_{(k_m^\B ),B}$.
\end{algorithm}

\section{Numerical studies}
\label{sec.numerical}


\subsection{Confidence sets}\label{sec:exp_conf_sets}

We first provide simulation studies to illustrate the performance of the robust bootstrap procedure for constructing confidence sets with various heavy-tailed errors. Recall the linear model
$Y_i=\bX_i^\T \btheta^*+\varepsilon_i$ in \eqref{reg.model}. We simulate $\{\bX_i\}_{i=1}^n$ from $\mathcal{N}(0, \bI_d)$. The true regression coefficient $\btheta^*$ is a vector equally spaced between $[0,1]$. The errors $\varepsilon_i$ are IID from one of the following distributions, standardized to have mean 0 and variance 1.
\begin{enumerate}
\item Standard Gaussian distribution $\mathcal{N}(0,1)$;
\item $t_{\nu}$-distribution with degrees of freedom $\nu=3.5$;
\item Gamma distribution with shape parameter $3$ and scale parameter $1$;
\item $t$-Weibull mixture (Wbl mix) model: $\varepsilon = 0.5u_t+0.5u_\mathrm{W}$, where $u_t$ follows a standardized $t_4$-distribution and $u_{\rm W}$ follows a standardized Weibull distribution with shape parameter $0.75$ and scale parameter $0.75$;
\item Pareto-Gaussian mixture (Par mix) model: $\varepsilon=0.5u_\mathrm{P}+0.5u_\mathrm{G}$, where $u_{\mathrm{P}}$ follows a Pareto distribution with shape parameter $4$ and scale parameter $1$  and $u_{{\rm G}}\sim \mathcal{N}(0,1)$;
\item Lognormal-Gaussian mixtrue (Logn mix) model: $\varepsilon=0.5u_\mathrm{LN}+0.5u_\mathrm{G}$, where $u_\mathrm{LN}=\exp(1.25Z)$ with $Z\sim \mathcal{N}(0,1)$ and $u_{{\rm G}}\sim \mathcal{N}(0,1)$.
\end{enumerate}
Moreover, we consider three types of random weights as follows:
\begin{itemize}
\item Gaussian weights: $W_i\sim \mathcal{N}(0,1)+1$;
\item Bernoulli weights (with mean 0.5): $W_i\sim 2\mathrm{Ber}(0.5)$;
\item A mixture of Bernoulli and Gaussian weights considered by \citet{Z2016}:  $W_i=z_i+u_i+1$, with $u_i\sim(\mathrm{Ber}(b)-b)\sigma_u$, $b=0.276$, $\sigma_u = 0.235$, and $z_i\sim \mathcal{N}(0,\sigma_z^2)$, $\sigma_z^2= 0.038$.
\end{itemize}
All three weights considered are such that $\EE( W_i)=\var(W_i)=1$. Using non-negative random weights has the advantage that the weighted objective function is
guaranteed to be convex. Numerical results reveal that Gaussian and Bernoulli weights demonstrate almost the same coverage performance.

\begin{table}[!t]
\centering
\caption{Average coverage probabilities with $n=100$ and $d=5$ for different nominal coverage levels $1-\alpha=[0.95,0.9,0.85,0.8,0.75]$. The weights $W_i$ are generated from $\mathcal{N}(1,1)$.}
\begin{tabular}{lllllllll}
\hline
Noise ~~~~~& Approach  &  $0.95$ & $0.9$ & $0.85$&$0.8$&$0.75$ \\
\hline
\multicolumn{4}{l}{Gaussian}\\
&boot-Huber&0.954&0.908&0.842&0.783&0.734\\
&boot-OLS&0.952&0.908&0.837&0.785&0.735\\
\multicolumn{4}{l}{$t_\nu$}\\
&boot-Huber&0.966&0.904&0.848&0.801&0.748\\
&boot-OLS&0.954&0.887&0.798&0.710&0.630\\
\multicolumn{4}{l}{Gamma}\\
&boot-Huber&0.962&0.918&0.860&0.798&0.747\\
&boot-OLS&0.955&0.910&0.843&0.775&0.700\\
\multicolumn{4}{l}{Wbl mix}\\
&boot-Huber&0.962&0.907&0.851&0.797&0.758\\
&boot-OLS&0.944&0.899&0.808&0.775&0.680\\
\multicolumn{4}{l}{Par mix}\\
&boot-Huber&0.955&0.907&0.856&0.802&0.761\\
&boot-OLS&0.948&0.900&0.843&0.785&0.738\\
\multicolumn{4}{l}{Logn mix}\\
&boot-Huber&0.958&0.912&0.860&0.782&0.744\\
&boot-OLS&0.954&0.912&0.796&0.682&0.616\\
\hline
\end{tabular}
\label{table:noise}
\end{table}

\begin{table}
\centering
\caption{Average coverage probabilities with $n=100$ and $d=5$ for different nominal coverage levels $1-\alpha=[0.95,0.9,0.85,0.8,0.75]$.}
Bernoulli weights $W_i \sim 2\mathrm{Ber}(0.5)$.\\
~\\
\begin{tabular}{llllllll}
\hline
Noise & Approach  & $0.95$ & $0.9$ & $0.85$&$0.8$&$0.75$ \\
\hline
Gaussian\\
&boot-Huber  & 0.947 & 0.895 & 0.847 & 0.792 & 0.740\\
&boot-OLS & 0.926 & 0.865 & 0.824 & 0.783& 0.726\\
$t_\nu$\\
&boot-Huber  & 0.930 & 0.897 & 0.841 & 0.786 & 0.755\\
&boot-OLS  & 0.884 & 0.815 & 0.757 & 0.703 & 0.646\\
Gamma\\
&boot-Huber  & 0.943 & 0.900 & 0.861 & 0.805 & 0.756\\
&boot-OLS  & 0.935 & 0.882 & 0.831 & 0.774 & 0.720\\
Wbl mix\\
&boot-Huber  & 0.948 & 0.894 & 0.842 & 0.793 & 0.739\\
&boot-OLS  & 0.931 & 0.859 & 0.779 & 0.721 & 0.664\\
Par mix\\
&boot-Huber  & 0.932&0.875&0.832&0.780&0.741\\
&boot-OLS  & 0.927&0.871&0.817&0.765&0.716\\
Logn mix\\
&boot-Huber  & 0.944 & 0.892 & 0.846 & 0.807 & 0.758\\
&boot-OLS & 0.915 & 0.838 & 0.792 & 0.720 & 0.674\\
\hline
\end{tabular}
~\\
~\\
~\\
Mixture weights $W_i=z_i+u_i+1$, with $u_i\sim(\mathrm{Ber}(b)-b)\sigma_u$, $b=0.276$, $\sigma_u=0.235$, and $z_i\sim \mathcal{N}(0,\sigma_z^2)$, $\sigma_z^2= 0.038$.\\
~\\
\begin{tabular}{llllllll}
\hline
Noise & Approach   & $0.95$ & $0.9$ & $0.85$&$0.8$&$0.75$ \\
\hline
Gaussian\\
&boot-Huber  &  0.930&0.864&0.812&0.763&0.698\\
&boot-OLS & 0.920&0.842&0.772&0.695&0.640\\
$t_\nu$\\
&boot-Huber  & 0.942& 0.893 & 0.844 & 0.788 & 0.733\\
&boot-OLS  & 0.894& 0.792 & 0.695 & 0.605 & 0.528\\
Gamma\\
&boot-Huber  &0.930 &0.873&0.831&0.782&0.731\\
&boot-OLS  & 0.911 & 0.832 & 0.754 & 0.686 & 0.625\\
Wbl mix\\
&boot-Huber  & 0.963 & 0.919 & 0.874 & 0.812 & 0.754\\
&boot-OLS  & 0.924 & 0.759 & 0.656 & 0.545 & 0.459\\
Par mix\\
&boot-Huber  &0.942&0.884&0.818&0.750&0.702 \\
&boot-OLS  &0.924&0.830&0.736&0.664&0.614 \\
Logn mix\\
&boot-Huber  & 0.940 & 0.876 & 0.829 & 0.796 & 0.751\\
&boot-OLS & 0.903 & 0.812 & 0.736 & 0.637 & 0.579\\
\hline
\end{tabular}
\label{table:weight}
\end{table}

\begin{table}[!t]
\centering
\caption{Average coverage probabilities with $n=200$, $d=5$ for different nominal coverage levels $1-\alpha=[0.95,0.9,0.85,0.8,0.75]$. The  weights $W_i$ are generated from $\mathcal{N}(1,1)$.}
\begin{tabular}{llllllll}
\hline
Noise ~~~~~& Approach  &  $0.95$ & $0.9$ & $0.85$&$0.8$&$0.75$ \\
\hline
\multicolumn{4}{l}{Gaussian}\\
&boot-Huber &0.957 &0.910 &0.850 &0.790 &0.736 \\
&boot-OLS &0.955 &0.907 &0.850 &0.789 &0.736 \\
\multicolumn{4}{l}{$t_\nu$}\\
&boot-Huber &0.958 &0.906 &0.848 &0.798 &0.749 \\
&boot-OLS &0.940 &0.863 &0.772 &0.684 &0.599 \\
\multicolumn{4}{l}{Gamma}\\
&boot-Huber &0.948 &0.899 &0.845 &0.780 & 0.726 \\
&boot-OLS &0.944 &0.889 &0.822 &0.751 &0.685 \\
\multicolumn{4}{l}{Wbl mix}\\
&boot-Huber &0.954 &0.889 &0.837 &0.775 &0.713 \\
&boot-OLS &0.939 &0.865 &0.784 &0.695 &0.621 \\
\multicolumn{4}{l}{Par mix}\\
&boot-Huber &0.945 &0.898 &0.847 &0.789 &0.738 \\
&boot-OLS &0.941 &0.886 &0.820 &0.757 &0.700 \\
\multicolumn{4}{l}{Logn mix}\\
&boot-Huber&0.958&0.916&0.864&0.812&0.748\\
&boot-OLS&0.938&0.886&0.812&0.718&0.590\\
\hline
\end{tabular}
\label{table:noise2}
\end{table}

\begin{table}[!t]
\centering
\caption{Average coverage probabilities for the Wbl mix error and for different nominal coverage levels $1-\alpha=[0.95,0.9,0.85,0.8,0.75]$. The weights $W_i$ are generated from $\mathcal{N}(1,1)$.}
~\\
\begin{tabular}{lllllllll}
\hline
Approach&$d$ & $n$\quad~\quad~\quad~\quad  &  $0.95$ & $0.9$ & $0.85$&$0.8$&$0.75$ \\
\hline
 boot-Huber \\
&&50&0.951&0.904&0.848&0.789&0.725\\
&2&100&0.959&0.914&0.866&0.827&0.771\\
&&200&0.954&0.917&0.856&0.814&0.756\\
&&50&0.982&0.945&0.876&0.826&0.752\\
&5&100&0.966&0.917&0.855&0.802&0.760\\
&&200&0.950&0.894&0.835&0.777&0.721\\
&&50&0.990&0.972&0.955&0.915&0.881\\
&10&100&0.980&0.949&0.897&0.850&0.799\\
&&200&0.970&0.922&0.864&0.826&0.777\\
\hline
boot-OLS \\
&&50&0.942&0.887&0.827&0.758&0.672\\
&2&100&0.956&0.901&0.849&0.785&0.714\\
&&200&0.947&0.898&0.822&0.763&0.685\\
&&50&0.976&0.911&0.836&0.754&0.688\\
&5&100&0.954&0.896&0.824&0.751&0.674\\
&&200&0.940&0.868&0.790&0.698&0.622\\
&&50&0.997&0.970&0.919&0.844&0.761\\
&10&100&0.975&0.921&0.850&0.784&0.719\\
&&200&0.954&0.879&0.816&0.731&0.650\\
\hline
\end{tabular}
\label{table:dn_boot_ols_4}
\end{table}

The number of bootstrap replications is set to be $B = 2000$. Nominal coverage probabilities $1-\alpha$ are given in the columns, where we consider $1-\alpha \in \{0.95, 0.90, 0.85, 0.80, 0.75\}$.  We report the empirical coverage probabilities from $1000$ simulations. We first consider a simple approach for choosing $\tau$, which is set to be $1.2 \{ \hat{\nu}_4 n / (d + \log n)\}^{1/4}$. Here, $\hat{\nu}_4$ is the empirical fourth moment of the residuals from the OLS and the constant 1.2 (which is slightly larger than 1) is chosen in accordance with Theorem \ref{Boot.consistency} which requires $v \geq \upsilon_4^{1/4}$. This simple ad hoc approach leads to adequate results in most cases. In Section \ref{sec:exp_adap}, we further investigate the empirical performance of the fully data-dependent procedure proposed in Section \ref{sec:adap}.


We compare our method with an OLS-based bootstrap procedure studied in \cite{SZ2015}, namely, replacing the weighted Huber loss in \eqref{bootHuber.est} by the weighted quadratic loss $\mathcal{L}_{\mathrm{ols}}^\B(\btheta) = \sn W_i (Y_i - \bX_i^\T \btheta)^2$.

Consider the sample size $n=100$ and dimension $d=5$. Table \ref{table:noise} and Table \ref{table:weight} display the coverage probabilities of the proposed bootstrap Huber method (boot-Huber) and the bootstrap OLS method (boot-OLS). 
At the normal model, our approach achieves a similar performance as the boot-OLS, which demonstrates the efficiency of adaptive Huber regression.
For heavy-tailed errors, our method significantly outperforms the boot-OLS using all three types of random weights. 
Also, we observe that the Gaussian and Bernoulli weights demonstrate nearly the same desirable performance.
For simplicity, we focus on the Gaussian weights throughout the remaining simulation studies.

In Table \ref{table:noise2}, we increase the sample size to $n=200$ and retain all the other settings.
For most cases of heavy-tailed errors, the coverage probability of the boot-OLS method is lower than the nominal level, sometimes to a large extent.
In Table \ref{table:dn_boot_ols_4}, we generate errors from a $t$-Weilbull mixture distribution and consider different combinations of  $n$ ($n \in \{50, 100, 200\}$) and $d$ ($d \in \{2,5,10\}$). The robust procedure outperforms the least squares method across most of the settings.
Similar phenomena are also observed in other cases of heavy-tailed errors.

We also report the standard deviations of the estimated quantiles of boot-Huber and boot-OLS; see Appendix~\ref{sec:std} in the supplement. The experimental results show that the boot-Huber leads to uniformly smaller standard deviations.
Furthermore, we consider more challenging settings with correlated or non-Gaussian designs and non-equally spaced $\btheta^{*}$. The average coverage probabilities of the boot-Huber method are in general close to nominal level, while the boot-OLS leads to severe under-coverage in many heavy-tailed noise settings. More details are presented in Appendix~\ref{sec:corr} in the supplementary material.

\subsection{Performance of the data-driven tuning approach}
\label{sec:exp_adap}

\begin{figure}[!t]
\centering
  \subfigure[$\mathcal{N}(0,1)$]{\includegraphics[width = 0.495\textwidth]{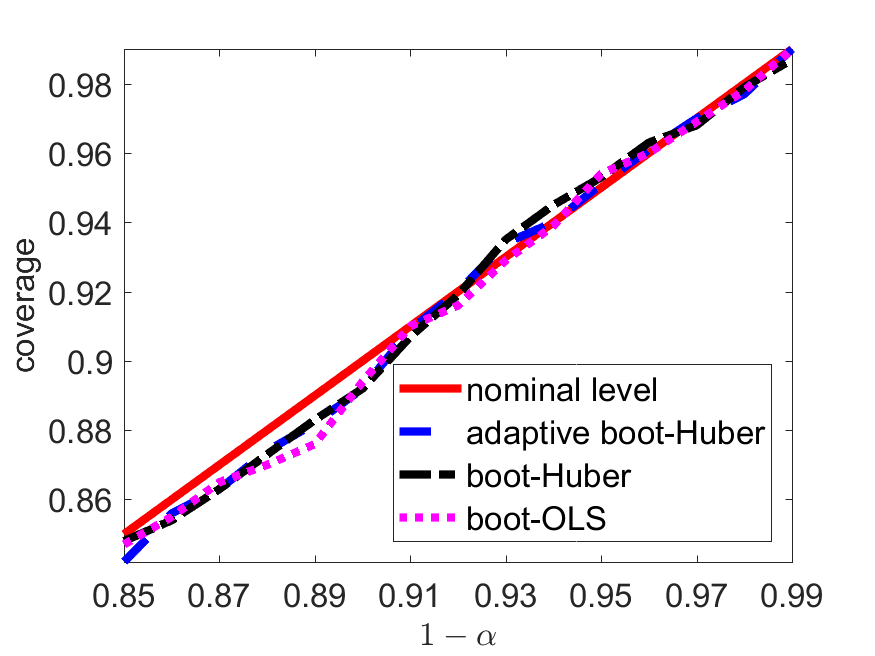}}
  \subfigure[$Logn(0,1)$]{\includegraphics[width = 0.495\textwidth]{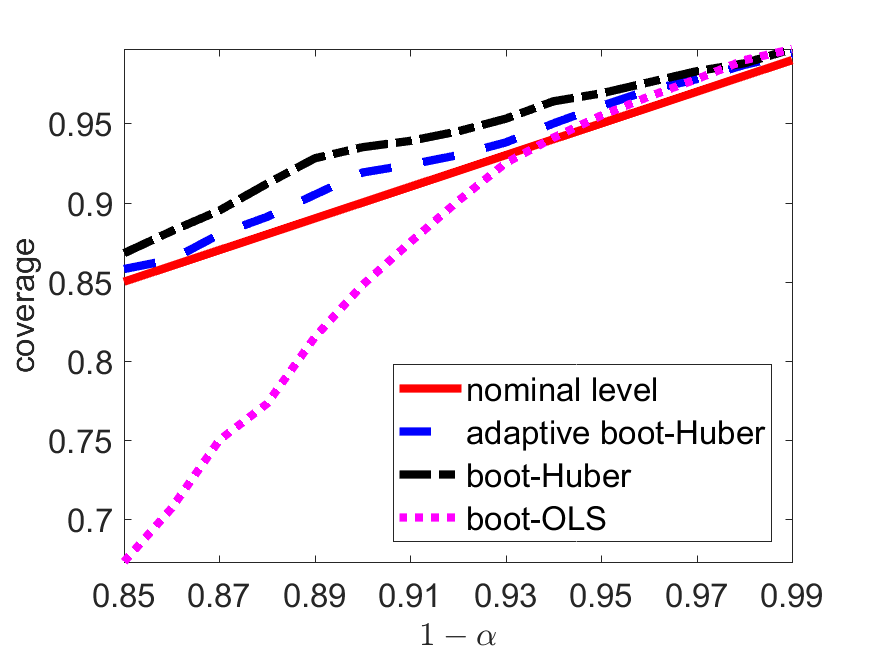}}
  \subfigure[$Logn(0,1.5)$]{\includegraphics[width = 0.495\textwidth]{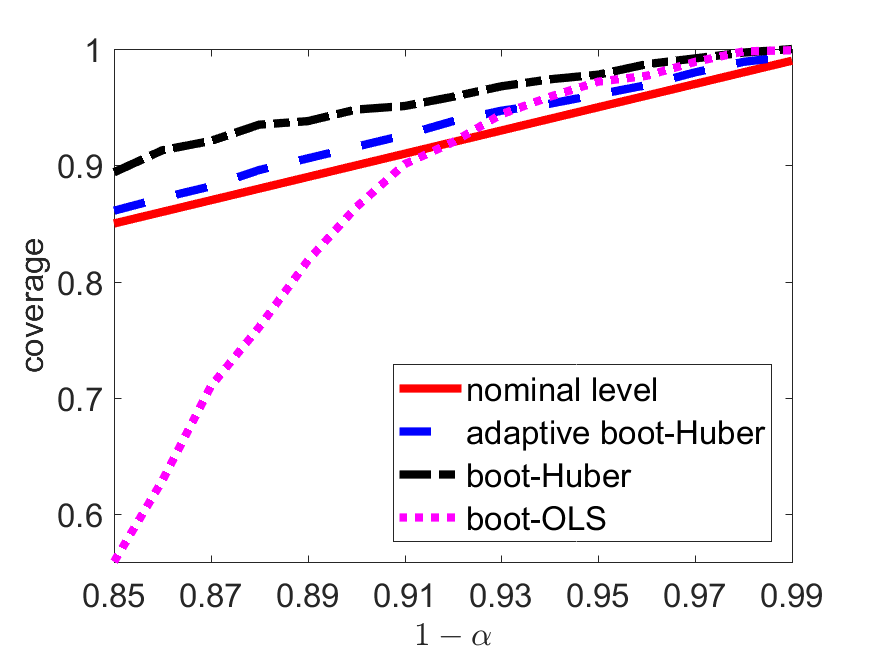}}
  \subfigure[$Logn(0,2)$]{\includegraphics[width = 0.495\textwidth]{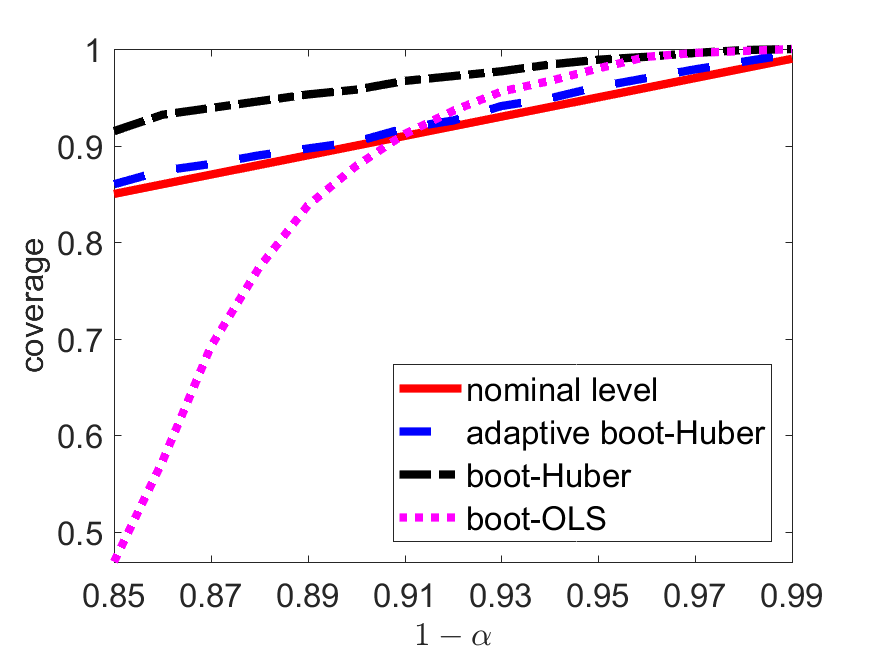}}
\caption{Comparison of coverage probabilities for different error distributions when the nominal coverage level $1-\alpha$ ranges from $0.85$ to $0.99$. Here, $x$-axis represents $1-\alpha$ and $y$-axis represents the average coverage rates over 1000 simulations. The red line represents the nominal coverage probability.
}
\label{fig:adap}
\end{figure}

\begin{table}[!t]
\centering
\caption{Average coverage probabilities for different nominal coverage levels $1-\alpha \in \{0.99,0.97,0.95,0.9,0.87\}$. The weights $W_i$ are generated from $\mathcal{N}(1,1)$.}
\label{tab:adap}
\begin{tabular}{lllllll}
\hline
Noise         & Approach            & 0.99 & 0.97 & 0.95 & 0.90 & 0.87 \\
\hline
$\cN(0,1)$          &                     &       &       &       &       &       \\
              & adaptive boot-Huber & 0.993 & 0.970 & 0.942 & 0.896 & 0.868 \\
              & boot-Huber          & 0.991 & 0.971 & 0.946 & 0.899 & 0.868 \\
              & boot-OLS            & 0.993 & 0.970 & 0.948 & 0.895 & 0.868 \\
$Logn(0, 1)$       &                     &       &       &       &       &       \\
              & adaptive boot-Huber & 0.994 & 0.978 & 0.961 & 0.919 & 0.880 \\
              & boot-Huber          & 0.997 & 0.983 & 0.969 & 0.935 & 0.895 \\
              & boot-OLS            & 0.997 & 0.978 & 0.955 & 0.848 & 0.750 \\
$Logn(0, 1.5)$ &                     &       &       &       &       &       \\
              & adaptive boot-Huber & 0.994 & 0.980 & 0.961 & 0.916 & 0.882 \\
              & boot-Huber          & 1.000 & 0.992 & 0.978 & 0.948 & 0.921 \\
              & boot-OLS            & 0.999 & 0.989 & 0.972 & 0.864 & 0.710 \\
$Logn(0, 2)$       &                     &       &       &       &       &       \\
              & adaptive boot-Huber & 0.995 & 0.979 & 0.961 & 0.904 & 0.881 \\
              & boot-Huber          & 1.000 & 0.996 & 0.989 & 0.958 & 0.939 \\
              & boot-OLS            & 1.000 & 0.996 & 0.980 & 0.879 & 0.692 \\
\hline
\end{tabular}
\end{table}

We further investigate the empirical performance of the data-driven procedure proposed in Section \ref{sec:adap}. We consider lognormal distributions $Logn(\mu, \sigma)$ with location parameter $\mu=0$ and varying shape parameters $\sigma$. The larger the value of $\sigma$ is, the heavier the tail is.
Moreover, we take $n=200$, $d=5$ and $1-\alpha \in [0.85, 0.99]$ and compare three methods: (1) Huber-based bootstrap procedure with $\tau$ calibrated by solving \eqref{tuning-free.2} (adaptive boot-Huber), (2) Huber-based bootstrap procedure with $\tau=1.2 \{ \hat{\nu}_4 n / (d + \log n)\}^{1/4}$ (boot-Huber), and (3) OLS-based bootstrap method (boot-OLS).

From Figure~\ref{fig:adap} and Table~\ref{tab:adap} we see that, under lognormal models, the coverage probabilities of the adaptive boot-Huber method are closest to nominal levels, while the boot-OLS suffers from distorted empirical coverage: it tends to overestimate the real quantiles at high levels and severely underestimate the real quantiles at relatively lower levels.
In addition, Figure~\ref{fig:adap}--(a) shows that the proposed Huber-based procedure almost loses no efficiency under a normal model.

\section{Discussion}
\label{sec.discuss}

In this paper, we have proposed and analyzed robust inference methods for linear models with heavy-tailed errors. Specifically, we use a multiplier bootstrap procedure for constructing sharp confidence sets for adaptive Huber estimators and conducting large-scale simultaneous inference with heavy-tailed panel data. Our theoretical results provide explicit bounds for the bootstrap approximation errors and justify the bootstrap validity; the error of coverage probability is small as long as $d^3/n$ is small. For multiple testing, we show that when the error distributions have finite 4th moments and the dimension $m$ and sample size $n$ satisfy $\log m = o(n^{1/3})$, the bootstrap Huber procedure asymptotically controls the overall false discovery proportion at the nominal level.

Furthermore, the proposed robust inference method can be potentially applied to a broad range of statistical problems, including high dimensional sparse regression, reduced rank regression, covariance matrix estimation and low-rank matrix recovery. We leave such an extension for further research.

\newpage
\section*{Supplementary Material} 
 \appendix

This supplemental material contains (1) the proofs of Theorems~2.1--2.6 and Theorem 3.1 in the main text, (2) implementation of the proposed methods, and (3) additional simulation studies.

\section{Notations and Preliminaries}

\subsection{Notations}
\label{sec:notations}

Recall that the error variable $\varepsilon$ has mean zero and variance $\sigma^2 = \EE(\varepsilon^2) >0$. For every $\tau>0$, we define the truncated mean and second moment of $\varepsilon$ to be
\begin{align}
 m_\tau = \EE\{ \psi_\tau(\varepsilon) \} ~~\mbox{ and }~~ \sigma^2_\tau = \EE \{ \psi^2_\tau(\varepsilon)\} , \label{trun.mean.def}
\end{align}
where $\psi_\tau(u) := \ell'_\tau(u)=\sgn(u) \min(|u|, \tau)$, $u\in \RR$. For IID random variables $\varepsilon_1, \ldots, \varepsilon_n$ from $\varepsilon$, we define truncated variables
\begin{align}
	\xi_i = \psi_\tau(\varepsilon_i ) , \ \ i=1,\ldots, n. \label{def:xii}
\end{align}
The dependence of $\xi_i$ on $\tau$ will be assumed without displaying.

Moreover, define the $d\times d$ random matrix
\begin{align}
	\bS_n =  \bS_{n,\tau} = \frac{1}{n} \sn \xi_i ^2  \bZ_i \bZ_i^\T \label{def:Sn}
\end{align}
and the random variable
\begin{align}
	  M_{n,4} = \sup_{\bu \in \mathbb{S}^{d-1}} \frac{1}{n} \sn ( \bu^\T \bZ_i )^4 .   \label{def:Mn}
\end{align}

Throughout, we use $\PP^{\dagger}$-probability to denote the probability measure over $\{(Y_i, \bX_i)\}_{i=1}^n$ and use $\PP^*$-probability to denote the probability measure over $\{ U_i \}_{i=1}^n$ conditioning on $\{(Y_i, \bX_i)\}_{i=1}^n$. In general, $\PP$ denotes the probability measure over all the random variables involved.

\subsection{Technical lemmas}

In this section, we provide several technical lemmas that will be used repeatedly to prove the main theorems. Recall the isotropic random vectors $\bZ_i$ given in \eqref{Zi.def}. The first two lemmas provide concentration properties for $M_{n,4}$ and $\bS_{n}$, respectively.

\begin{lemma} \label{lem.design}
{\rm
Assume Condition~\ref{moment.cond} holds. Then for any $x>0$,
\begin{align}
	M_{n,4} \leq 2 \sup_{\bu \in \mS^{d-1}} \EE (\bu^\T \bZ)^4 + C \Bigg\{ \sqrt{\frac{3d + x }{n}} + \frac{(3d + x)^2}{n} \Bigg\}    \label{design.concentration}
\end{align}
with probability at least $1-2e^{-x}$, where $C>0$ depends only on $A_0$.}
\end{lemma}

\begin{proof}
The proof is based on the covering argument. For any $\epsilon\in(0,1)$, we can find an $\epsilon$-net $\cN_\epsilon$ of the unit sphere $\mathbb{S}^{d-1}$ satisfying ${\rm card}(\cN_\epsilon) \leq (1+2/\epsilon)^d$. For every $\bu \in \mS^{d-1}$, there exists some $\bv\in \cN_\epsilon$ such that $\| \bu - \bv \|_2 \leq \epsilon$. Define the map $f: \RR^d \to \RR^n$ as
$$
	f(\bu )  = n^{-1/4} (  \bu^\T \bZ_1, \ldots, \bu^\T \bZ_n )^\T.
$$
By the triangle inequality,
\begin{align}
	& \| f(\bu ) \|_4 \leq  \| f(\bv ) \|_4 +  \| f(\bu ) - f(\bv ) \|_4 \nn \\
	& =    \| f(\bv ) \|_4 +\bigg( \frac{1}{n}\sn \langle \bu - \bv , \bZ_i \rangle^4 \bigg)^{1/4}  \leq    \| f(\bv ) \|_4 + \epsilon M_{n,4}^{1/4} . \nn
\end{align}
Taking the maximum over $\bv \in \cN_\epsilon$ and then taking the supremum over $\bu \in \mS^{d-1}$, we arrive at
$$
	M_{n,4}^{1/4} \leq  N_{n, \epsilon}^{1/4} +  \epsilon M_{n,4}^{1/4},
$$	
where $N_{n, \epsilon} = \max_{\bv \in \cN_\epsilon} (1/n) \sn (\bv^\T \bZ_i)^4$. Solving this inequality yields
\begin{align}
	M_{n,4} \leq   (1-\epsilon)^{-4} N_{n,\epsilon} . \label{discretization}
\end{align}

For every $\bv \in \cN_\epsilon$, note that $\PP\{ (\bv^\T \bZ_i)^4 \geq y\} \leq 2e^{-\sqrt{y}/A_0^2}$ for any $y>0$. Hence, by inequality (3.6) in \cite{Adam2011} with $s=1/2$, we obtain that for any $z>0$,
$$
	\PP\Bigg\{   \bigg| \sn (\bv^\T \bZ_i)^4 - \EE (\bv^\T \bZ_i)^4 \bigg| \geq z \Bigg\} \leq 2 \exp\Bigg\{-  c \min\bigg( \frac{z^2}{n C_1^2} , \sqrt{\frac{z}{C_1}} \, \bigg) \Bigg\} ,
$$
where $c>0$ is a universal constant and $C_1>0$ depends only on $A_0$. Taking the union bound over all vectors $\bv$ in $\cN_\epsilon$ gives
\begin{align}
	& \PP\Bigg\{ \max_{\bv \in \cN_\epsilon}  \bigg| \sn (\bv^\T \bZ_i)^4 - \EE (\bv^\T \bZ_i)^4 \bigg| \geq z \Bigg\} \nn \\
	&  \leq 2\exp\Bigg\{  d \log(1+2/\epsilon) -  c \min\bigg( \frac{z^2}{n C_1^2} , \sqrt{\frac{z}{C_1}} \, \bigg) \Bigg\} . \nn
\end{align}
It follows that
\begin{align}
	N_{n,\epsilon} & \leq \sup_{\bu\in\mS^{d-1}} \EE(\bu^\T \bZ)^4  \nn \\
	& \quad~ + C_2  \Bigg[ \sqrt{\frac{d\log (1+2/\epsilon) + x }{n}} + \frac{\{ d \log  (1+2/\epsilon)  + x\}^2}{n} \Bigg]  \label{discrete.concentration}
\end{align}
with probability at least $1-2e^{-x}$, where $C_2>0$ depends only on $A_0$.

Finally, taking $\epsilon=1/8$ in \eqref{discretization} and \eqref{discrete.concentration} implies \eqref{design.concentration}.
\end{proof}

\begin{lemma} \label{cov.concentration}
{\rm
Assume Condition~\ref{moment.cond} holds with $\delta=2$. For any $x>0$,
\begin{align}
 \|  \bS_{n} - \sigma^2_\tau \,\bI_d  \|_2 \leq 4 \sqrt{2} A_0^2 \,  \upsilon_4^{1/2}\sqrt{\frac{ \log 9^d + x }{n}} + 2 A_0^2 \, \tau^2 \frac{ \log 9^d+ x}{n}  \label{barSn.concentration}
\end{align}
with probability at least $1- 2 e^{-x}$.}
\end{lemma}

\begin{proof}
Define random variables $w_{i} = \xi_i  / \sigma_\tau$ so that $\EE( w_{i}^2 ) =1$. We will bound $\| \bDelta \|_2$ via a standard covering argument, where
$$
	\bDelta =\bigg\| \frac{1}{n} \sn w_{i}^2 \bZ_i \bZ_i^\T - \bI_d  \bigg\|_2 =  \frac{1}{\sigma_\tau^2} \|  \bS_{n} - \sigma^2_\tau \,\bI_d  \|_2  .
$$

Proceed similarly to the proof of Lemma~4.4.1 in \cite{V2018}, it can be shown that there exists a $1/4$-net $\mathcal{N}_{1/4}$ of the unit sphere $\mathbb{S}^{d-1}$ satisfying $|\mathcal{N}_{1/4} | \leq 9^d$ such that
$$
 \| \bDelta \|_2 \leq 2 \max_{\bu \in \mathcal{N}_{1/4} } | \bu^\T \bDelta \bu |  = 2 \max_{\bu \in \mathcal{N}_{1/4} }  \bigg|  \frac{1}{n} \sn   w_{i}^2 (\bu^\T \bZ_i)^2  - 1 \bigg| .
$$
For any $\bu \in \mathbb{S}^{d-1}$, by \eqref{sG.equiv} we have $\EE(\bu^\T \bZ)^{2k} \leq  2 A_0^{2k} \, k!$ for all $k\geq 1$. This implies
\begin{align}
	\EE\{ w_{i}^4 (\bu^\T \bZ_i)^4 \} = \EE \{ (\bu^\T \bZ_i)^4 \EE(w_i^4 | \bX_i) \} \leq  4A_0^4 \,\sigma_\tau^{-4} \upsilon_4 , \nn \\
	\mbox{ and }~  \EE \{ w_{i}^2 (\bu^\T \bZ_i)^2\}^k  \leq   \frac{k!}{2} 4 A_0^4 \,\sigma_\tau^{-4} \upsilon_4  (  A_0^2 \, \sigma_\tau^{-2} \tau^2  )^{k-2} ~\mbox{ for all } k\geq 3 .   \nn
\end{align}
It then follows from Bernstein's inequality that for any $x\geq 0$,
\begin{align}
	\PP \bigg\{  | \bu^\T \bDelta \bu | \geq 2\sqrt{2} A_0^2 \, \sigma_\tau^{-2} \upsilon_4^{1/2}  \sqrt{\frac{x}{n}} + A_0^2 \, \sigma_\tau^{-2} \tau^2\frac{  x}{n} \bigg\} \leq 2e^{-x} . \nn
\end{align}
Taking the union bound over all $\bu \in \mathcal{N}_{1/4}$ yields
\begin{align}
 \| \bDelta \|_2 \leq 4\sqrt{2} A_0^2 \, \sigma_\tau^{-2}  \upsilon_4^{1/2}  \sqrt{\frac{x}{n}} + 2 A_0^2 \, \sigma_\tau^{-2} \tau^2 \frac{  x}{n} \nn
\end{align}
with probability at least $1- 9^d \cdot 2 e^{-x}$. Reinterpret this we reach \eqref{barSn.concentration}.
\end{proof}

The next lemma gives a deviation bound for the $\ell_2$-norm of the $d$-variate random vector $\bxi^\B =  - \sn \xi_i  U_i \bZ_i$, where $\xi_i$ are given in \eqref{def:xii}. Recall that $\PP^*$ is the conditional probability measure over the random multipliers given $\mathcal{D}_n=\{ (Y_i, \bX_i) \}_{i=1}^n$.

\begin{lemma} \label{lem.l2bootscore}
{\rm
Assume Condition~\ref{weight.cond} is fulfilled. For every $x>0$, it holds with $\PP^*$-probability at least $1- 2 e^{-x}$ that
\begin{align}
 \frac{1}{\sqrt{n}} \| \bxi ^\B \|_2  \leq  B_U  \|   \bS_{n} \|_2^{1/2} (\log 5^d + x)^{1/2} , \label{cond.xiB.concentration}
\end{align}
where $B_U>0$ is a constant depending only on $A_U$.}
\end{lemma}

\begin{proof}
The proof is based on an argument similar to that leads to \eqref{xi.concentration}. There exists a $1/2$-net $\mathcal{N}_{1/2}$ of $\mathbb{S}^{d-1}$ with cardinality $|\mathcal{N}_{1/2}| \leq 5^d$ such that
$$
	\| \bxi^\B \|_2 \leq 2\max_{\bu \in \mathcal{N}_{1/2}} | \bu^\T \bxi^\B |  =   2\max_{\bu \in \mathcal{N}_{1/2}}  \bigg| \sn \xi_i \bu^\T \bZ_i   U_i \bigg|.
$$
For each fixed $\bu\in \mathbb{S}^{d-1}$, applying Theorem~2.6.3 in \cite{V2018} gives
\begin{align}
	\PP^*\Bigg[   \bigg| \sn \xi_i \bu^\T \bZ_i U_i \bigg| \geq  C \bigg\{ \sn \xi_i^2 (\bu^\T \bZ_i)^2 \bigg\}^{1/2}  \sqrt{x} \,\Bigg]  \leq 2 e^{- x} \mbox{ for every } x \geq 0 ,\nn
\end{align}
where $C =C(A_U)>0$ and $\xi_i = \psi_\tau(\varepsilon_i)$ are as in \eqref{def:xii}. Taking the union bound over all vectors $\bu \in \mathcal{N}_{1/2}$, we obtain that with $\PP^*$-probability greater than $1- 5^d \cdot 2e^{-x}$,
\begin{align}
	\| \bxi^\B \|_2 \leq 2  C \bigg\|  \sn \xi_i^2 \bZ_i \bZ_i^\T \bigg\|_2^{1/2} \sqrt{x} . \nn
\end{align}
Reinterpret this inequality to obtain the stated result \eqref{cond.xiB.concentration}.
\end{proof}

Recall the random process $\bxi^\B(\btheta) = -\sn \ell'_\tau(Y_i - \bX_i^\T \btheta) U_i \bZ_i$, $\btheta \in \RR^d$ defined in \eqref{xiB.def}. The following lemma gives an upper bound on the local fluctuation $\sup_{\btheta \in \Theta_0(r)} \| \bxi^\B(\btheta) -	\bxi^\B(\btheta^*) \|_2$ for $r>0$.

\begin{lemma} \label{lem.bootfluctuation}
{\rm
Assume Condition~\ref{weight.cond} holds. For any $x>0$, it holds with $\PP^*$-probability at least $1- e^{-x}$ that
\begin{align} \label{score.fluctuation}
  \sup_{\btheta \in \Theta_0(r)}  \frac{1}{\sqrt{n}}\| \bxi^\B( {\btheta} ) -\bxi^\B( {\btheta}^*)  \|_2 \leq C  M_{n,4}^{1/2}  (8 d + 2 x )^{1/2}  r   ,
\end{align}
where $C >0$ depends only on $A_U$ and $M_{n,4}$ is given in \eqref{def:Mn}.}
\end{lemma}

\begin{proof}
To begin with, note that
$$
	\sup_{\btheta \in \Theta_0(r)}  \| \bxi^\B( {\btheta} ) -\bxi^\B( {\btheta}^*)  \|_2 = \sup_{\btheta \in \Theta_0(r)}  \sup_{  \bomega \in \mathbb{B}^d(r)}   \bomega^\T \{\bxi^\B( {\btheta} ) -\bxi^\B( {\btheta}^*) \} /r .
$$
Define a new process $\bxi^\B(\btheta,  \bomega) = \bomega^\T \{\bxi^\B( {\btheta} ) -\bxi^\B( {\btheta}^*) \} /(2r \sqrt{n} ) $ for $\btheta \in \Theta_0(r)$ and $ \bomega \in  \mathbb{B}^d(r)$, so that
\begin{align}
\sup_{\btheta \in \Theta_0(r)} \frac{1}{\sqrt{n}}  \| \bxi^\B( {\btheta} ) -\bxi^\B( {\btheta}^*)  \|_2 =  2  \sup_{\btheta \in \Theta_0(r)}  \sup_{ \bomega \in \mathbb{B}^d(r)} \bxi^\B(\btheta,  \bomega). \label{equiv}
\end{align}
It is easy to see that $\EE^* \{  \bxi^\B(\btheta,  \bomega) \} = 0$ and
$$
	\nabla_{\btheta } \,\bxi^\B(\btheta,  \bomega)  = \frac{  \bomega^\T \nabla \bxi^\B(\btheta) }{2 r \sqrt{n}}   , \quad  \nabla_{ \bomega}\,\bxi^\B(\btheta,  \bomega)  = \frac{    \bxi^\B( {\btheta} ) -\bxi^\B( {\btheta}^*)  }{2r \sqrt{n} } .
$$
For any $\bu, \bv \in \RR^d$ and $\lambda \in \RR$, by H\"older's inequality,
\begin{align}
 & \log \EE^* \exp\bigg\{  \lambda   \frac{  (\bu^\T , \bv^\T) \nabla \bxi^\B(\btheta,  \bomega) }{ \| ( ( \bSigma^{1/2} \bu )^\T , \bv ) \|_2  } \bigg\}  \nn \\
 & \leq \frac{1}{2} \log   \EE^* \exp\bigg\{ \frac{\lambda}{r\sqrt{n}}   \frac{   \bomega^\T \nabla \bxi^\B(\btheta) \bu   }{\| ( ( \bSigma^{1/2} \bu )^\T , \bv ) \|_2} \bigg\}  \nn \\
 & \quad ~+ \frac{1}{2} \log \EE^* \exp\bigg[ \frac{\lambda}{ r \sqrt{n}}  \frac{ \bv^\T \{\bxi^\B( {\btheta} ) -\bxi^\B( {\btheta}^*) \} }{\| ( ( \bSigma^{1/2} \bu )^\T , \bv ) \|_2} \bigg]. \label{log.exp.bound}
\end{align}
Write $\overline{\bu} = \bSigma^{1/2} \bu /\| ( ( \bSigma^{1/2} \bu )^\T , \bv ) \|_2 $ and $\overline{ \bv} = \bv / \| ( ( \bSigma^{1/2} \bu )^\T , \bv ) \|_2 $. For the first term on the right-hand side of \eqref{log.exp.bound}, it follows from \eqref{xiB.def} and Condition~\ref{weight.cond} that
\begin{align}
 & \EE^* \exp\bigg\{   \frac{\lambda}{r \sqrt{n}}  \frac{    \bomega^\T \nabla \bxi^\B(\btheta) \bu }{\| ( ( \bSigma^{1/2} \bu )^\T , \bv ) \|_2}  \bigg\}  \nn \\
 & =  \EE^*  \exp\bigg\{ \frac{\lambda}{r \sqrt{n}}  \sn \ell''_\tau(Y_i - \bX_i^\T \btheta)    \bomega^\T \bZ_i  \bZ_i^\T  \overline \bu  \, U_i \bigg\} \nn \\
 & = \prod_{i=1}^n  \EE^* \exp\bigg\{   \frac{\lambda}{r \sqrt{n}}  \ell''_\tau(Y_i - \bX_i^\T \btheta)    \bomega^\T \bZ_i  \bZ_i^\T  \overline \bu \, U_i \bigg\} \nn \\
 & \leq \prod_{i=1}^n    \exp\bigg\{   \frac{\lambda^2 B_U^2 }{2  r^2 n  } ( \bomega^\T \bZ_i)^2 (\overline{\bu}^\T \bZ_i )^2 \bigg\} \nn \\
 & \leq \exp\bigg(   \frac{\lambda^2 }{2} B_U^2  M_{n,4}  \bigg)   \label{log.exp.bound1}
\end{align}
almost surely. For the second term, by the mean value theorem and taking $\bdelta_r = \bSigma^{1/2}(\btheta - \btheta^*)/r$, we get
\begin{align}
	& \EE^* \exp\bigg[ \frac{\lambda}{ r \sqrt{n}}   \overline \bv^\T \{\bxi^\B( {\btheta} ) -\bxi^\B( {\btheta}^*) \}  \bigg] \nn  \\
 & = \prod_{i=1}^n \EE^*    \exp\bigg\{  \frac{\lambda}{ \sqrt{n} }  I (|Y_i - \bX_i^\T \wt \btheta | \leq \tau ) \overline{\bv}^\T \bZ_i \bZ_i^\T  \bdelta_r \, U_i \bigg\}   \nn \\
 & \leq   \exp\bigg(   \frac{\lambda^2  }{2} B_U^2 M_{n,4}  \bigg) \label{log.exp.bound2}
\end{align}
almost surely, where $\wt \btheta$ is a convex combination of $\btheta$ and $\btheta^*$ and $M_{n,4}$ is given in \eqref{def:Mn}. Putting \eqref{log.exp.bound}, \eqref{log.exp.bound1} and \eqref{log.exp.bound2} together yields
\begin{align}
	&  \log \EE^* \exp\bigg\{ \lambda   \frac{  (\bu^\T , \bv^\T) \nabla \bxi^\B(\btheta,  \bomega) }{ \| ( ( \bSigma^{1/2} \bu )^\T , \bv ) \|_2  } \bigg\}   \leq \frac{\lambda^2 }{2}  B_U^2 M_{n,4} . \nn
\end{align}
Applying a conditional version of Theorem~A.1 in \cite{S2013} with $p = 2d$,
\begin{align*}
H_0 =  {\small\left(
\begin{array}{cc}
  \bSigma^{1/2} & \textbf{0}   \\
 \textbf{0}  & \bI_d
\end{array}
\right) } , \quad {\rm g} = \infty  ~\mbox{ and }~  \nu_0^2=  B_U^2 M_{n,4}
\end{align*}
to the process $\{ \bxi^\B(\btheta,  \bomega) : \btheta \in \Theta_0(r),  \bomega \in \mathbb{B}^d(r)\}$ in \eqref{equiv}, we arrive at
\begin{align}
	\PP^* \Bigg\{ \sup_{\btheta \in \Theta_0(r)}  \frac{1}{\sqrt{n}} \| \bxi^\B( {\btheta} ) \!- \! \bxi^\B( {\btheta}^*)  \|_2  \geq  6B_U M_{n,4}^{1/2}   (8 d +   2 x )^{1/2}    r   \Bigg\} \leq e^{-x}  \nn
\end{align}
almost surely. This is the bound stated in \eqref{score.fluctuation}.
\end{proof}

The following lemma provides moderate deviation results for the robust estimators $\hat{\mu}_k$ given in \eqref{robust.mle}.

\begin{lemma} \label{robust.md}
{\rm
Assume Condition~\ref{moment.cond2} holds. Let $\{a_n \}_{n\geq 1}$ be a sequence of positive numbers satisfying $a_n \to \infty$ and $a_n = o(n^{1/2})$ as $n\to \infty$. For each $1\leq k\leq m$, the robust estimator $\hat{\mu}_k$ with $\tau_k = v_k \{ n/ (s + a_n) \}^{1/3}$ for some $v_k \geq \upsilon_{k,4}^{1/4}$ satisfies
\begin{align}
	 \frac{\PP( \sqrt{n} \, |\hat{\mu}_k - \mu_k | \geq z ) }{2 \{ 1 - \Phi(z/\sigma_k) \}} \to 1  \label{md.result1}
\end{align}
uniformly for $0\leq z\leq o\{  \sigma_k \min( n^{1/6} ,  \sqrt{n} a_n^{-1} ) \} $ and $z \leq  \sigma_k \sqrt{a_n}$. }
\end{lemma}

\begin{proof}
Let $1\leq k\leq m$ be fixed and write $\tau=\tau_k$ for simplicity. Define truncated mean and variance of $\varepsilon_k$ by $m_{k,\tau} = \EE\{ \psi_\tau(\varepsilon_k)\}$ and $\sigma_{k,\tau}^2 = \EE \{ \psi^2_\tau(\varepsilon_k) \} $. Moreover, define
$$
	T_k =  \sn \psi_\tau(\varepsilon_{ik}) ~\mbox{ and }~ T_{0k} =  \sn \{ \psi_\tau(\varepsilon_{ik}) - m_{k,\tau} \} .
$$
Taking $\bX = (1, \bx^\T)^\T$, $\btheta^* = (\mu_k, \bbeta_k^\T)^\T$ and $\varepsilon=\varepsilon_k$ in Theorem~\ref{br.thm}, we obtain that with probability at least $1-4e^{-a_n}$,
\begin{align}
	| \sqrt{n} \, (\hat{\mu}_k - \mu_k) -  T_k  / \sqrt{n} | \leq C_1 v_k (s+ a_n ) n^{-1/2}  \label{BR}
\end{align}
as long as $n \geq C_2  (s+a_n )$.  We prove \eqref{md.result1} by considering the following two cases.

\medskip
\noindent
{\sc Case 1}: Assume $0\leq z/\sigma_k \leq 1$. Applying Theorem~2.2 in \cite{ZBFL2017} to $T_k$ gives
\begin{align}
	\sup_{x\in \RR}  | \PP (  T_k / \sqrt{n} \leq x ) - \Phi(x/\sigma_k)   | \leq C  \bigg( \frac{\upsilon_{k,3}}{\sigma_k^3 \sqrt{n}} + \frac{\upsilon_{k,4}}{\sigma_k^2 \tau^2 }  + \frac{\upsilon_{k,4 }\sqrt{n} }{\sigma_k \tau^3} \bigg) , \nn
\end{align}
where $C >0 $ is an absolute constant. The stated result \eqref{md.result1} then holds uniformly for $0\leq z\leq \sigma_k$.

\medskip
\noindent
{\sc Case 2}: Assume $1 \leq z/\sigma_k \leq \sqrt{a_n}$. It follows from Proposition A.2 with $\kappa=4$ in the supplement of \cite{ZBFL2017} that $| m_{k,\tau} | \leq \upsilon_{k,4} \tau^{- 3}$. Together with \eqref{BR}, this implies that with probability at least $1-4e^{-a_n}$,
\begin{align}
	& |   \sqrt{n} \, (\hat{\mu}_k - \mu_k) - T_{0k} / \sqrt{n} |  \nn \\
	& \leq \delta_1 := C_1 v_k (s+a_n ) n^{-1/2} + \upsilon_{k,4}  \tau^{-3} \sqrt{n} .
	\label{def:delta1}
\end{align}
It follows that
\begin{align}
	& \PP (    |T_{0k} |  / \sqrt{n} \geq z+ \delta_1  ) - 4e^{-a_n} \nn \\
	&  \leq \PP ( \sqrt{n}\, |\hat{\mu}_k - \mu_k| \geq z  ) \leq \PP ( |T_{0k} |  / \sqrt{n} \geq z-\delta_1 ) + 4e^{-a_n } . \label{GAR.1}
\end{align}

Next we focus on $T_{0k}$. Recall that $\tau = v_k \{ n/(s+a_n) \}^{1/3}$ with $v_k  \geq \upsilon_{k,4}^{1/4}$. To apply Lemma 3.1 in the supplement of \cite{LS2014}, we take
$$
	d=1, \ \ B_n = \sigma_k^2 n, \ \ c_n = \frac{\tau}{\sigma_k \sqrt{n}}  = \frac{v_k }{\sigma_k (s+a_n )^{1/3} n^{1/6}}
$$
and note that
\begin{align}
	  \bigg| \frac{ \cov(T_{0k}) }{B_n} -1  \bigg| \leq  \frac{|\sigma_{k,\tau}^2 - \sigma_k^2 | + m_{k,\tau}^2 }{\sigma_k^2 } \leq b_n := \frac{ \upsilon_{k,4}}{\sigma_k^2 \tau^2} + \frac{\sigma_k^2}{ \tau^2}  , \nn \\
  \beta_n := \frac{1}{B_n^{3/2}} \sn \EE |\psi_\tau(\varepsilon_{ik}) - m_{k,\tau} |^3  \leq C_2 \frac{ \upsilon_{k,3}}{\sigma_k^3 \sqrt{n}}  , \nn
\end{align}
where $C_2>0$ is an absolute constant. Consequently, taking $d_n = n^{-1/6}$, $x= \sqrt{n} z$ and $t_n = (C_{3,1}^{-1/2} \vee 4)\{ z/\sigma_k + (\log n)^{1/2} \} $ in Lemma 3.1 implies that for all sufficiently large $n$,
\begin{align}
	& | \PP (    |T_{0k} |  / \sqrt{n} \geq z  ) - \PP (  | Z | \geq  z /\sigma_k ) | \nn \\
	& \leq    C_3  \{ \beta_n t_n^3 + n^{-1/6} ( 1+  z/\sigma_k ) \} \PP (  | Z | \geq  z/ \sigma_k )  \nn \\
	& \quad \, +   7n^{-1} e^{ -( z/\sigma_k )^2 } + 9 e^{- c_1 n^{1/3}} \label{GAR.2}
\end{align}
uniformly over $0\leq z/\sigma_k \leq  c_2 \min( c_n^{-1}, \beta_n^{-1/3} , d_n^{-1} ) $, where $Z\sim \mathcal{N}(0,1)$, $c_1>0$ depends only on $(\sigma_k, \upsilon_{k,3})$ and $C_3, c_2>0$ are absolute constants. For normal distribution, it is known that for any $w > 0$,
$$
	 \frac{w}{1+w^2}  \frac{1}{\sqrt{2\pi}} e^{-w^2/2} \leq    1-\Phi(w) \leq  \min\bigg( \frac{1}{2} , \frac{1}{ w\sqrt{2\pi}} \bigg) e^{-w^2/2}.
$$
Combining this with \eqref{GAR.2},  we obtain that for $z > \sigma_k$ and $\delta_1$ in \eqref{def:delta1},
\begin{align}
	& |  \PP (\sigma_k |Z| \geq z - \delta_1 ) - \PP ( |Z| \geq z/\sigma_k ) | \nn \\
	& \leq \frac{2\delta_1}{\sqrt{2\pi}} e^{-(z-\delta_1)^2/(2\sigma_k^2)} \leq \delta_1 (1+ z/\sigma_k ) e^{\delta_1 z/\sigma_k^2} \, \PP (  |Z| \geq z /\sigma_k ) \label{GAR.3}
\end{align}
and
\begin{align}
	& |  \PP (\sigma_k |Z| \geq z + \delta_1 ) - \PP (  |Z| \geq z / \sigma_k ) |  \nn  \\
	& \leq \frac{2\delta_1}{\sqrt{2\pi}} e^{-z^2/(2\sigma_k^2)}  \leq  \delta_1  (1+ z/\sigma_k ) \, \PP ( |Z| \geq z /\sigma_k ) . \label{GAR.4}
\end{align}

Finally, observe that $e^{-a_n} \leq e^{-a_n/2- (z/\sigma_k)^2/2}$ for $z/\sigma_k \leq \sqrt{a_n}$. Then it follows from \eqref{def:delta1}--\eqref{GAR.4} that \eqref{md.result1} holds uniformly for $1\leq z / \sigma_k \leq o\{ \min ( n^{1/6} ,  \sqrt{n} a_n^{-1}  )  \}$, which completes the proof.
\end{proof}

\section{Proofs for Section~2}
\label{sec:supp_proof}

Without loss of generality, we assume $t\geq \log 2$, or equivalently $2e^{-t} \leq 1$ throughout the proof; otherwise if $2e^{-t}>1$, the conclusion is trivial. Let $\| \cdot \|_{\bSigma, 2}$ denote the rescaled $\ell_2$-norm on $\RR^d$, i.e. $\| \bu \|_{\bSigma,2} = \| \bSigma^{1/2} \bu \|_2$ for $\bu \in \RR^d$.

\subsection{Proof of Theorem~\ref{br.thm}}
\label{proof2.1}
\noindent
{\sc Proof of \eqref{concentration.MLE}}. To begin with, define the parameter set
\begin{align}
		\Theta_0(r)  := \{ \btheta \in \RR^{d} : \|   \btheta - \btheta^*  \|_{\bSigma, 2 } \leq r  \}  , \ \ r >0. \label{set.def}
\end{align}
For any prespecified $r>0$, we can find an intermediate estimator $\hat{\btheta}_{\tau, \eta} = \btheta^* + \eta(\hat{\btheta}_\tau - \btheta^*)$ for some $\eta\in [0,1]$, satisfying $w(\eta) := \|     \hat{\btheta}_{\tau, \eta} - \btheta^*  \|_{\bSigma, 2} \leq r$. In fact, if $\hat{\btheta}_\tau \in \Theta_0(r)$, we can simply take $\eta =1$; otherwise, since the function $w(\cdot): [0,1] \mapsto (0,\infty)$ is continuous with $w(0)=0$ and $w(1) >r$, there always exists some $\eta\in (0,1)$ such that $w(\eta) = r$. Applying Lemma~F.2 in \cite{FLSZ2015} to the loss function $\overline{\cL}_\tau := (1/n)\cL_\tau$, we obtain
\begin{align}
	 \langle \nabla  \overline \cL_\tau(\hat{\btheta}_{\tau, \eta} ) - \nabla  \overline \cL_\tau(\btheta^*)  , \hat{\btheta}_{\tau, \eta} - \btheta^* \rangle &  \leq \eta \langle   \nabla \overline \cL_\tau(\hat{\btheta}_\tau  ) - \nabla \overline \cL_\tau(\btheta^*)  , \hat{\btheta}_\tau - \btheta^*  \rangle \nn \\
	&  =  -\eta  \langle   \nabla \overline \cL_\tau(\btheta^*)  , \hat{\btheta}_{\tau} - \btheta^*  \rangle , \label{FOC.cond}
\end{align}
where the last step uses the first order condition $ \nabla \cL_\tau(\hat{\btheta}_\tau ) = \textbf{0}$.

In what follows, we bound the two sides of \eqref{FOC.cond} separately. Proposition~\ref{prop:RSC} below shows that $\overline \cL_\tau$ is strongly convex on $\Theta_0(r)$ with high probability.

\begin{proposition} \label{prop:RSC}
{\rm
Assume that kurtosises of the linear forms $\langle \bu, \bZ\rangle$ are uniformly bounded by $\kappa^4$ for some $\kappa>0$, i.e. $\EE \langle \bu,  \bZ \rangle^4 \leq \kappa^4 \| \bu \|_2^4$ for all $\bu \in \RR^d$.
Let $(\tau, r)$ and $(n,d,t)$ satisfy
\begin{align}
	 \tau \geq 2 \max\big\{ (4 \upsilon_{2+\delta})^{1/(2+\delta)} ,    4 \kappa^2 r \big\}    ~\mbox{ and }~  n \geq C (\tau/r)^2(d+t), \label{scaling.cond}
\end{align}
where $C >0$ is an absolute constant. Then with probability at least $1-e^{-t}$,
\begin{align}
    \langle \nabla  \overline \cL_\tau( \btheta ) - \nabla  \overline \cL_\tau(\btheta^*)  , \btheta - \btheta^* \rangle \geq \frac{1}{4} \| \btheta -\btheta^*  \|_{\bSigma, 2}^2 ~\mbox{ for all }~ \btheta \in \Theta_0(r). \label{RSC.lbd}
\end{align}
}
\end{proposition}

By construction, $\hat \btheta_{\tau, \eta } \in \Theta_0(r)$ and therefore under the scaling \eqref{scaling.cond},
\begin{align}
	\langle  \overline \cL_\tau( \hat \btheta_{\tau,\eta} ) - \nabla \overline  \cL_\tau(\btheta^*) , \hat \btheta_{\tau,\eta} - \btheta^* \rangle  \geq \frac{1}{4}  \|  \hat \btheta_{\tau,\eta} - \btheta^* \|_{\bSigma, 2}^2   \label{RSC.bound}
\end{align}
with probability at least $1-e^{-t}$.

Next we bound the quadratic form $ \|   \bSigma^{-1/2} \nabla \overline \cL_\tau(\btheta^*) \|_2$. Define the centered random vector $\bgamma = \bSigma^{-1/2}  \{ \nabla  \overline \cL_\tau(\btheta^*) - \EE \nabla  \overline \cL_\tau( \btheta^* ) \}$ so that
\begin{align}
	\|   \bSigma^{-1/2} \nabla   \overline \cL_\tau(\btheta^*) \|_2 \leq \| \bgamma \|_2 +  \|   \bSigma^{-1/2}  \EE \nabla  \overline \cL_\tau( \btheta^* ) \|_2. \label{L2.dec}
\end{align}
To bound $\| \bgamma  \|_2$, by a standard covering argument, there exits a $1/2$-net $\mathcal{N}_{1/2}$ of $\mathbb{S}^{d-1}$ with $|\mathcal{N}_{1/2} | \leq 5^d$ such that $\| \bgamma \|_2   \leq 2\max_{\bu \in \mathcal{N}_{1/2}} | \bu^\T \bgamma |$. For every $\bu\in \mathbb{S}^{d-1}$, note that $ | \bu^\T \bgamma | = |  (1/n) \sn  \{ \xi_i   \bu^\T \bZ_i  - \EE \xi_i   \bu^\T \bZ_i  \} |$, where $\xi_i = \ell'_\tau( \varepsilon_i )$ and $\bZ_i$ are IID from $\bZ$ given in \eqref{Zi.def}.
Since $\bu^\T \bZ$ is sub-Gaussian, it follows from the proof of Proposition~2.5.2 in \cite{V2018} that
\begin{align}
	\EE|\bu^\T \bZ|^{k} \leq  A_0^k  k \Gamma(k/2) ~\mbox{ for all } k \geq 2  , \label{sG.equiv}
\end{align}
If $k = 2\ell$ for some $\ell \geq 1$, $\EE|\bu^\T \bZ|^{k} \leq  2 A_0^k (k/2)!$; otherwise if $k=2\ell +1$ for some $\ell \geq 1$,
$$
\EE|\bu^\T \bZ|^{k} \leq     A_0^k  k \Gamma(\ell +1/2) =  k\sqrt{\pi} A_0^k \frac{(2\ell)!}{4^\ell \ell!} =  2\sqrt{\pi}  A_0^k \frac{k!}{2^k \ell!} .
$$
By the above calculations, we obtain
\begin{align}
	\EE(\xi_i \bu^\T \bZ_i)^2 = \EE\{ \bu^\T \bZ_i)^2 \EE(\xi_i^2 | \bX_i) \}	\leq \sigma^2   , \nn \\
	\mbox{ and }~ \EE | \xi_i \bu^\T \bZ_i |^k \leq \frac{ k!}{2} \,  2\sigma^2  A_0^2    (A_0 \tau /2 )^{k-2}  ~\mbox{ for all } k\geq 3. \nn
\end{align}
Applying Bernstein's inequality, we see that
\begin{align}
	\PP\bigg(   | \bu^\T \bgamma | \geq  2\sigma A_0 \sqrt{\frac{x}{n}} + \frac{A_0 }{2} \frac{\tau x}{n}  \bigg) \leq 2 e^{-x } ~\mbox{ for any } x>0. \nn
\end{align}
Taking the union bound over all vectors $\bu \in \mathcal{N}_{1/2}$, we obtain that with probability at least $5^d \cdot 2 e^{-x}$, $\| \bgamma \|_2 \leq 2\max_{\bu \in \mathcal{N}_{1/2}} | \bu^\T \bgamma | \leq 4 \sigma A_0 \sqrt{ x/n} +  A_0 \tau x/n$. Taking $x= 2(d+t)\geq \log (5^d) +2 t$, we reach
\begin{align}
	\PP\bigg\{  \| \bgamma \|_2 \geq    4\sqrt{2} \, \sigma A_0 \sqrt{\frac{d+t}{n}}   +  2A_0 \tau \frac{d+t}{n} \bigg\} \leq 2 e^{-2t} \leq e^{-t}.  \label{xi.concentration}
\end{align}
For the second term $\| \bSigma^{-1/2} \EE \nabla  \overline \cL_\tau(\btheta^*) \|_2$ in \eqref{L2.dec}, it holds
\begin{align}
	\|  \bSigma^{-1/2}  \e \nabla  \overline \cL_\tau(\btheta^*) \|_2  =  \sup_{\bu \in \mathbb{S}^{d-1}}   \frac{1}{n}\sn \e  ( \xi_i  \bu^\T \bZ_i )    \leq   \frac{\upsilon_{2+\delta}}{\tau^{1+\delta}} . \nn
\end{align}
Together, the last two displays imply with probability at least $1-  e^{-t}$,
\begin{align}
  & \|   \bSigma^{-1/2}    \nabla  \overline \cL_\tau(\btheta^*) \|_2 \nn \\
  &  \leq   r_0 := \frac{ \upsilon_{2+\delta}}{\tau^{1+\delta}}  +  4\sqrt{2} \, \sigma  A_0 \sqrt{\frac{ d +t}{n}} + 2A_0 \tau \frac{ d+t}{n}   .  \label{quadratic.bound}
\end{align}

Finally, in view of \eqref{scaling.cond}, \eqref{RSC.bound} and \eqref{quadratic.bound}, we choose $r = \tau/(4\kappa^2)$.
Under Condition~\ref{moment.cond}, $\kappa$ scales as $A_0$.
Then with probability at least $1-2e^{-t}$, $ \hat{\btheta}_{\tau, \eta} \in \Theta_0(4r_0)$ under the scaling \eqref{scaling.cond}. Provided $n\gtrsim  A_0^4(   d+t )$, we have $r>4r_0$ so that $\hat \btheta_{\tau, \eta}$ lies in the interior of $\Theta_0(r)$, which enforces $\eta=1$ and $\hat \btheta_{\tau,\eta} = \hat \btheta_\tau$ (otherwise $\hat \btheta_{\tau,\eta}$ will lie on the boundary).
Putting together the pieces, we arrive at \eqref{concentration.MLE}.

\medskip
\noindent
{\sc Proof of \eqref{Bahadur.representation}}. Next we prove \eqref{Bahadur.representation}. Define random processes
\begin{align}
	\bB(\btheta ) =  \bSigma^{-1/2}   \{ \nabla  \overline \cL_\tau(\btheta) - \nabla  \overline \cL_\tau(\btheta^*) \}   -   \bSigma^{1/2} (\btheta  - \btheta^*) \label{process.B}
\end{align}
and $\bzeta(\btheta) =   \overline \cL_\tau(\btheta ) - \EE  \overline \cL_\tau(\btheta)$. In this notation, we have
\begin{align}
	& \bB(\btheta ) \nn \\
	& =    \bSigma^{-1/2}    \{ \nabla \bzeta(\btheta) - \nabla \bzeta(\btheta^*) + \nabla \EE  \overline \cL_\tau(\btheta) - \nabla \EE  \overline \cL_\tau(\btheta^*)    -    \bSigma (\btheta  - \btheta^*) \}  \nn \\
	&\mbox{ and }~  \EE \{ \bB(\btheta ) \}  = \bSigma^{-1/2}  \{  \nabla \EE  \overline \cL_\tau(\btheta) - \nabla\EE  \overline \cL_\tau(\btheta^*)  \}  -    \bSigma^{1/2} (\btheta  - \btheta^*) . \nn
\end{align}
In the following, we will deal with $ \bB(\btheta )  - \EE \{ 	\bB(\btheta)\}$ and $\EE \{ \bB(\btheta ) \}$ separately, starting with the latter. By the mean value theorem for vector-valued functions (see, e.g. Theorem~12 in Section~2 of \cite{P2015}),
\begin{align}
  \EE \{ \bB(\btheta ) \} & =  \bSigma^{-1/2}    \EE \int_0^1 \nabla^2  \overline \cL_\tau(\btheta^*_t ) \, dt  (\btheta - \btheta^*) -   \bSigma^{1/2} (\btheta  - \btheta^*) \nn \\
 &  =  \bigg\{ \bSigma^{-1/2}   \int_0^1  \nabla^2 \EE \overline \cL_\tau( \btheta_t^* )\, dt \,\bSigma^{-1/2} -   \bI_d  \bigg\}\bSigma^{1/2} (\btheta  - \btheta^* ), \nn
\end{align}
where $\btheta^*_t = (1-t)\btheta^* + t \btheta$. Note that
$$
	 \nabla^2 \EE  \overline \cL_\tau(\btheta_t^*) = \bSigma - \bSigma^{1/2}  \frac{1}{n}\sn \e  \{ I (| Y_i - \bX_i^\T \btheta_t^* | > \tau  ) \bZ_i \bZ_i^\T  \} \bSigma^{1/2}.
$$
For every $t\in [0,1]$, since $\btheta\in \Theta_0(r)$ and $\bu \in \mathbb{S}^{d-1}$, we have $\btheta^*_t \in \Theta_0(r)$ so that $\bdelta_t := \bSigma^{1/2}( \btheta_t^* - \btheta^*)$ satisfies $\| \bdelta_t \|_2\leq r$. Consequently, by Markov's inequality and \eqref{sG.equiv},
\begin{align}
	&  |   \bu^\T \{ \bSigma^{-1/2}  \nabla^2 \EE  \overline \cL_\tau( \btheta_t^* ) \bSigma^{-1/2}  -  \bI_d   \} \bu   |  \nn \\
	&   \leq   \frac{1}{n} \sn \e  \{ I (| Y_i - \bX_i^\T \btheta_t^* | > \tau ) ( \bu^\T \bZ_i )^2 \}   \nn \\
	& \leq   \upsilon_{2+\delta} (2/\tau)^{2+\delta}   +   2  ( \EE |\bdelta_t^\T \bZ|^{2 } )^{1/2} ( \EE |\bu^\T \bZ|^{4} )^{1/2}\tau^{-1}  \nn \\
	& \leq  \delta(r)   := 2^{2+\delta }\upsilon_{2+\delta} \tau^{-2-\delta}  + 4\kappa^2     \tau^{-1} r , \nn
\end{align}
where $\kappa>0$ is as in Proposition~\ref{prop:RSC}.
Putting together the pieces implies
\begin{align} \label{def.delta}
	 \sup_{\btheta \in \Theta_0(r)} \| \EE\{ \bB(\btheta) \}  \|_2 \leq  \delta(r)  r .
\end{align}

Turing to $\bB(\btheta ) - \EE \{\bB(\btheta)\} = \bSigma^{-1/2} \{ \nabla \bzeta(\btheta) - \nabla \bzeta(\btheta^*) \}$, we set
\begin{align}
	  \overline{\bB}(\bdelta )= \bB(\btheta ) - \EE \{\bB(\btheta)\} \mbox{ with } \bdelta = \bSigma^{1/2}( \btheta - \btheta^*). \nn
\end{align}
It is easy to see that $\overline{\bB}( \textbf{0}) = \textbf{0}$, $\EE\{ \overline{\bB}(\bdelta )\} = \textbf{0}$ and
$$
	\nabla_{\bdelta} \overline{\bB}(\bdelta) = \frac{1}{n} \sn [ \ell''_\tau(\varepsilon_i - \bZ_i^\T \bdelta) \bZ_i \bZ_i^\T - \EE \{\ell''_\tau(\varepsilon_i - \bZ_i^\T \bdelta) \bZ_i \bZ_i^\T\} ].
$$
In addition, for any $\bu , \bv \in \mathbb{S}^{d-1}$ and $\lambda \in \RR$, using the inequality $|e^z -1 - z | \leq  z^2 e^{|z|} /2$ for all $z\in \RR$ gives
\begin{align}
	& \EE \exp\{  \lambda \sqrt{n} \, \bu^\T \nabla_{\bdelta} \overline{\bB}(\bdelta) \bv  \} \nn \\
	& \leq \prod_{i=1}^n  \bigg[ 1 \! + \! \frac{\lambda^2 }{ n } \e  \{
  (  \bu^{\intercal} \bZ_i  \bv^\T \bZ_i )^2   \!+\! ( \e |  \bu^\T \bZ   \bv^\T \bZ | )^2 \}  e^{    \frac{|\lambda|}{\sqrt{n}}  ( | \bu^{\intercal} \bZ_i \bv^{\intercal} \bZ_i | + \e  | \bu^{\intercal}  \bZ  \bv^{\intercal} \bZ  |  )  }     \bigg]  . \nn
\end{align}
By the Cauchy-Schwarz inequality,
$$
	\EE | \bu^{\intercal}  \bZ  \bv^{\intercal} \bZ  |  \leq \{ \EE (\bu^\T \bZ)^2\}^{1/2} \{ \EE(\bv^\T \bZ)^2 \}^{1/2} = 1
$$
and
\begin{align}
	 & \e   (  \bu^{\intercal} \bZ_i  \bv^\T \bZ_i )^2 e^{    \frac{|\lambda|}{\sqrt{n}}    | \bu^{\intercal} \bZ_i \bv^{\intercal} \bZ_i |   }  \nn \\
& \leq    \e   (  \bu^{\intercal} \bZ_i  \bv^\T \bZ_i )^2 e^{    \frac{|\lambda|}{2\sqrt{n}}  (\bu^{\intercal} \bZ_i)^2 +      \frac{|\lambda|}{2\sqrt{n}}  (\bv^{\intercal} \bZ_i)^2  }  \nn \\
& \leq   \{ \EE (\bu^\T \bZ_i)^4 e^{    \frac{|\lambda|}{ \sqrt{n}}  (\bu^{\intercal} \bZ_i)^2   }  \}^{1/2}  \{ \EE (\bv^\T \bZ_i)^4 e^{    \frac{|\lambda|}{ \sqrt{n}}  (\bv^{\intercal} \bZ_i)^2   }   \}^{1/2}. \nn
\end{align}
Combining the last three displays, we arrive at	
\begin{align}
	& \EE \exp\{  \lambda \sqrt{n} \, \bu^\T \nabla_{\bdelta} \overline{\bB}(\bdelta) \bv  \} \nn \\
	&  \leq  \prod_{i=1}^n  \bigg[ 1 +    e^{\frac{| \lambda | }{\sqrt{n}}}  \frac{\lambda^2}{n}   \e ( e^{\frac{ | \lambda| }{\sqrt{n}} |  \bu^{\intercal} \bZ_i   \bv^{\intercal} \bZ_i | } )  \nn  \\
	& \qquad \qquad \qquad +   e^{\frac{| \lambda |}{\sqrt{n}}}  \frac{\lambda^2}{n} \e  \{    ( \bu^\T \bZ_i )^2 (  \bv^{\intercal} \bZ_i )^2  e^{\frac{ | \lambda | }{\sqrt{n}} | \bu^{\intercal} \bZ_i  \bv^{\intercal} \bZ_i | } \}  \bigg]  \nn \\
	& \leq  \prod_{i=1}^n \bigg[ 1 +  e^{\frac{ | \lambda| }{\sqrt{n}}}  \frac{\lambda^2}{n}    \max_{ \bu \in \mathbb{S}^{d-1} } \e \{  e^{ \frac{ | \lambda| }{\sqrt{n}} (\bu^{\intercal} \bZ )^2  } \} \nn \\
	&  \qquad \qquad \qquad  +  e^{\frac{| \lambda| }{\sqrt{n}}}  \frac{\lambda^2}{n}    \max_{ \bu \in \mathbb{S}^{d-1}  } \e \{  ( \bu^{\intercal} \bZ )^4 e^{ \frac{| \lambda | }{\sqrt{n}} (\bu^{\intercal} \bZ )^2  } \} \bigg]   \nn \\
	& \leq \exp\bigg[  e^{\frac{ | \lambda| }{\sqrt{n}}}   \lambda^2   \max_{ \bu \in \mathbb{S}^{d-1} } \e \{  e^{ \frac{ | \lambda | }{\sqrt{n}} (\bu^{\intercal} \bZ )^2  } \} \!+\!   e^{\frac{ | \lambda| }{\sqrt{n}}}  \lambda^2   \max_{ \bu \in \mathbb{S}^{d-1}  } \e \{  ( \bu^\T \bZ )^4 e^{ \frac{ | \lambda | }{\sqrt{n}} (\bu^{\intercal} \bZ )^2  } \} \bigg]. \nn
\end{align}
Under Condition~\ref{moment.cond}, there exists a constant $A_1  = A_1(A_0)>0$ such that, for all $|\lambda | \leq   \sqrt{n}/A_1$ and $\btheta \in \RR^{d}$,
\begin{align}
  \sup_{ \bu , \bv \in \mathbb{S}^{d-1} } \EE \exp\{  \lambda \sqrt{n} \, \bu^\T \nabla_{\bdelta} \overline{\bB}(\bdelta) \bv  \} \leq   \exp ( C^2 \lambda^2 /2 ),	\label{exp.moment.bound}
\end{align}
where $C   >0$ is an absolute constant. With the above preparations, applying Theorem~A.3 in \cite{S2013} which is a direct consequence of Corollary~2.2 in the supplement to \cite{S2012}, yields
\begin{align}
	\PP\Bigg\{ \sup_{\bdelta \in \mathbb{B}^d(r)} \| \sqrt{n} \, \overline{\bB}(\bdelta) \|_2 \geq 6 C  (8d + 2t)^{1/2} r \Bigg\}   \leq e^{-t} \nn
\end{align}
as long as $n\geq A_1^2(8d + 2t)$. Combining this and \eqref{def.delta}, we reach
\begin{align}
	\Delta(r) := \sup_{\btheta \in \Theta_0(r)}  \| \bB(\btheta ) \|_2   \leq   \delta(r) r    + 6   C  ( 8 d + 2t )^{1/2}  n^{-1/2} r  \label{Delta.r}
\end{align}
with probability at least $1-e^{-t}$, where $\bB(\btheta )$ is given in \eqref{process.B}.
Recalling the paragraph below \eqref{quadratic.bound}, we have $\PP \{ \hat{\btheta}_\tau \in  \Theta_0(4r_0) \} \geq 1 - 2 e^{-t}$ and $\nabla  \overline \cL_\tau(\hat{\btheta}_\tau)= \textbf{0}$. On the event $\{ \hat{\btheta}_\tau \in  \Theta_0(4r_0) \} $, it holds $\| \bB(\hat{\btheta}_\tau) \|_2 \leq \Delta(4r_0)$. Consequently, taking $r=4r_0$ in \eqref{Delta.r} proves \eqref{Bahadur.representation}.  \qed

\subsection{Proof of Theorem~\ref{wilks.thm}}

Keeping the notation appeared in the proof of Theorem~\ref{br.thm}, we consider the following local quadratic approximation of the Huber loss. For any $r>0$ and $\btheta  , \btheta' \in \Theta_0(r)$, define
$$
	R(\btheta, \btheta') = \cL_\tau(\btheta ) - \cL_\tau(\btheta') - (\btheta - \btheta')^\T \nabla \cL_\tau(\btheta') - \frac{n}{2} \| \bSigma^{1/2} (\btheta -\btheta' ) \|_2^2.
$$
Taking the gradient with respect to $\btheta$, we get $\nabla_{\btheta} R(\btheta, \btheta') = \nabla \cL_\tau(\btheta) - \nabla \cL_\tau(\btheta') - n \bSigma(\btheta - \btheta')$. Then, by the mean value theorem, $R(\btheta, \btheta')  = (\btheta - \btheta')^\T \{  \nabla \cL_\tau(\wt \btheta) - \nabla \cL_\tau(\btheta') - n \bSigma(\wt \btheta - \btheta') \}$, where $\wt \btheta$ is a convex combination of $\btheta$ and $\btheta'$ and hence $\wt \btheta \in \Theta_0(r)$. It follows that
\begin{align}
	& | R(\btheta, \btheta') |  \nn \\
	& \leq  n \| \bSigma^{1/2}(\btheta - \btheta') \|_2 \nn \\
	& \quad \times  \sup_{\btheta''  \in \Theta_0(r)} \|   \bSigma^{1/2} (\btheta'' - \btheta' ) -  \bSigma^{-1/2} n^{-1}  \{ \nabla \cL_\tau( \btheta'') - \nabla \cL_\tau(\btheta')  \}  \|_2 \nn \\
	& \leq 2 n \| \bSigma^{1/2}(\btheta - \btheta') \|_2 \times  \Delta(r) ,  \label{R.bound}
\end{align}
where $\Delta(r)$ is as in \eqref{Delta.r}. Recall from the proof of Theorem~\ref{br.thm} that $\PP\{\hat{\btheta}_\tau \in \Theta_0(4r_0)\} \geq 1- 2e^{-t}$ for $r_0$ given in \eqref{quadratic.bound}. Taking $r=4r_0$ in \eqref{Delta.r}, $(\btheta ,\btheta')=(\btheta^*, \hat{\btheta}_\tau)$ in \eqref{R.bound} and using the fact $\nabla \cL_\tau(\hat{\btheta}_\tau) = \textbf{0}$, we obtain that with probability greater than $1- 3e^{-t}$,
\begin{align}
 \bigg|  \cL_\tau(\btheta^* ) - \cL_\tau(\hat{\btheta}_\tau)   - \frac{n}{2} \| \bSigma^{1/2} (\btheta^* - \hat{\btheta}_\tau ) \|_2^2 \bigg| \leq
 8 n r_0 \Delta(4r_0). \nn
\end{align}

Write $\hat{\bdelta} = \bSigma^{1/2}(\hat{\btheta}_\tau - \btheta^*)$ and $\bGamma^* =  \bSigma^{-1/2}   \nabla  \overline \cL_\tau(\btheta^*)$. By \eqref{quadratic.bound} and \eqref{Delta.r}, we have $\| \bGamma^* \|_2 \leq r_0$, $\| \hat{\bdelta} +  \bGamma^* \|_2 \leq  \Delta(4r_0)$ and
\begin{align}
	& | \| \hat{\bdelta}  \|_2^2  - \| \bGamma^* \|_2^2 |   \leq  \|  \hat{\bdelta} + \bGamma^* \|_2^2 + 2 \| \bGamma^* \|_2  \|  \hat{\bdelta} + \bGamma^* \|_2  \leq \Delta(4r_0) \{ \Delta(4r_0) + 2 r_0 \} . \nn
\end{align}
Together, the last two displays imply that with probability at least $1- 3 e^{-t}$,
$$
\bigg|  \cL_\tau(\btheta^* ) - \cL_\tau(\hat{\btheta}_\tau)   - \frac{n}{2  } \| \bSigma^{-1/2}  \nabla \overline \cL_\tau(\btheta^*)  \|_2^2 \bigg| \leq  9n r_0 \Delta(4 r_0) + \frac{n}{2} \Delta^2(4r_0) ,
$$
which, together with \eqref{Delta.r}, proves \eqref{Wilks.expansion}

For the square-root Wilks' expansion, on $\{ \hat{\btheta}_\tau \in \Theta_0(4r_0)\}$ it holds
\begin{align}
 	&  \bigg| \sqrt{ 2 \{  \cL_\tau(\btheta^* ) - \cL_\tau(\hat{\btheta}_\tau)  \} } - \sqrt{n} \, \| \hat{\bdelta} \|_2  \bigg| \nn \\
   & \leq  \frac{ | 2\{  \cL_\tau(\btheta^* ) - \cL_\tau(\hat{\btheta}_\tau)\} - n \|  \hat{\bdelta} \|_2^2  | }{ \sqrt{n} \, \|  \hat{\bdelta} \|_2 } = \frac{2R(\btheta^* , \hat{\btheta}_\tau )}{\sqrt{n} \, \|  \hat{\bdelta} \|_2 }   \leq 4 \sqrt{n} \, \Delta(4r_0), \nn
\end{align}
where the last step follows from \eqref{R.bound}. Moreover, note that
\begin{align}
	& | \|  \hat{\bdelta} \|_2 -  \| \bGamma^* \|_2  |   \leq \|\hat{\bdelta} +  \bGamma^*\|_2 \leq \Delta(4r_0). \nn
\end{align}
Combining the last two displays with \eqref{Delta.r} proves \eqref{sqrt.Wilks.expansion}. \qed

\subsection{Proof of Proposition~\ref{prop:RSC}}

Since the Huber loss is convex and differentiable, we have
\begin{align}
	\cT(\btheta  ) & :=   \langle \nabla \overline \cL_\tau(\btheta ) -  \nabla \overline \cL_\tau(\btheta^*), \btheta  - \btheta^* \rangle \nn \\
	& = \frac{1}{n} \sn  \big\{ \ell_\tau'(Y_i - \bX_i^\T \btheta  )  - \ell_\tau'(Y_i - \bX_i^\T \btheta^* )  \big\} \bX_i^\T (\btheta - \btheta^* ) \nn \\
	& \geq  \frac{1}{n} \sn  \big\{ \ell_\tau'(Y_i - \bX_i^\T \btheta )  -\ell_\tau'(Y_i - \bX_i^\T \btheta^* )  \big\} \bX_i^\T (\btheta  - \btheta^*) I_{\cE_i}, \label{T12}
\end{align}
where $I_{\cE_i}$ is the indicator function of the event
\begin{align}
	\cE_i  :=   \big\{ |\varepsilon_i | \leq \tau/2 \big\}  \cap \big\{    |\langle  \bX_i, \btheta - \btheta^* \rangle |  \leq  (\tau/2r)   \|   \btheta - \btheta^*   \|_{\bSigma, 2} \big\}  ,  \nn
\end{align}
on which $|Y_i - \bX_i^\T \btheta  | \leq \tau$ for all $\btheta \in \Theta_0(r)$. Also, recall that $\ell_\tau''(u) = 1$ for $|u|\leq \tau$. For any $R>0$, define functions
\begin{align}
	\varphi_R(u) =\begin{cases}
		u^2   &  \mbox{ if }  |u | \leq  \frac{R}{2} , \\
		(u-R)^2   &  \mbox{ if }  \frac{R}{2} \leq u \leq R, \\
		(u+R)^2 &  \mbox{ if }  -R \leq u \leq - \frac{R}{2}, \\
		0  & \mbox{ if } |u| >R,
	\end{cases} ~\mbox{ and }~
	\psi_R(u) = I(|u| \leq R ) .
 \nn
\end{align}
In particular, $\varphi_{R}$ is $R$-Lipschitz and satisfies
\begin{align}
  u^2 I(|u| \leq R / 2)	\leq \varphi_R(u) \leq  u^2  I(|u| \leq R ) . \label{phi.bound}
\end{align}
It then follows that
\begin{align}
		 \cT(\btheta )  \geq   g(\btheta )  :=   \frac{1}{n} \sn  \varphi_{\tau \|  \btheta  - \btheta^*  \|_{\bSigma, 2}/ (2r)} (\langle \bX_i, \btheta  - \btheta^* \rangle ) \psi_{\tau/2}(\varepsilon_i)  . \label{def:g}
\end{align}

To bound the right-hand side of \eqref{def:g}, consider the supremum of a random process indexed by $\Theta_0(r)$:
\begin{align}
	\Delta_r    :=  \sup_{ \btheta  \in \Theta_0(r) }  \frac{ | g(\btheta   ) - \EE  g(\btheta  ) | }{\|   \btheta -\btheta^* \|_{\bSigma, 2}^2 }.  \label{def:Delta}
\end{align}
For any $ \btheta $ fixed, write $\bdelta = \btheta  - \btheta^*$.
By \eqref{phi.bound},
\begin{align}
  \EE  g(\btheta   )  & \geq \EE  \langle  \bX_i, \bdelta \rangle^2 \nn \\
  & \quad  - \EE  \bigg\{  \langle  \bX_i, \bdelta \rangle^2 I \bigg( \frac{ | \langle  \bX_i, \bdelta \rangle|}{\| \bdelta  \|_{\bSigma, 2}} \geq \frac{\tau}{4r}   \bigg) \bigg\}   -  \EE  \big\{  \langle  \bX_i,\bdelta \rangle^2 I(|\varepsilon_i |\geq \tau/2 ) \big\}   \nn \\
& \geq \| \bdelta  \|_{\bSigma, 2}^2 -   \Bigg\{   \bigg( \frac{4r}{\tau} \bigg)^2\frac{ \EE   \langle  \bX_i,   \bdelta \rangle^4 }{ \|  \bdelta  \|_{\bSigma, 2}^2} +   \bigg( \frac{2}{\tau} \bigg)^{2+\delta}  \EE \langle  \bX_i,\bdelta \rangle^2 |\varepsilon_i|^{2+\delta}   \Bigg\}.  \nn
\end{align}
Provided  $\tau \geq 2 \max \{ (4 \upsilon_{2+\delta})^{1/(2+\delta)} , 4\kappa^2 r\}$, it follows that
\begin{align}
	 \EE g(\btheta  ) \geq \| \bdelta  \|_{\bSigma, 2}^2   - \| \bdelta  \|_{\bSigma, 2}^2 \bigg(   \frac{16 \kappa^4 r^2}{\tau^2}   +  \frac{2^{2+\delta} \upsilon_{2+\delta}}{\tau^{2+\delta}} \bigg)     \geq \frac{1}{2}  \| \bdelta  \|_{\bSigma, 2}^2       \label{mean.g.2}
\end{align}
for all $ \btheta  \in \Theta_0(r)$. From \eqref{def:g}--\eqref{mean.g.2}, we conclude that
\begin{align}
	\frac{\cT(\btheta )}{\| \btheta  -  \btheta^* \|_{\bSigma ,2}^2 }  \geq \frac{1}{2}   - \Delta_r ~\mbox{ for all }~   \btheta  \in  \Theta_0(r) . \label{Lower.bound.1}
\end{align}

Next we deal with the stochastic term $\Delta_r$ defined in \eqref{def:Delta}. For $g(\btheta )$ given in \eqref{def:g}, we write $g(\btheta ) = (1/n) \sn g_i(\btheta )$.
Recalling that $0\leq \varphi_R(u)  \leq R^2/4$ and $0\leq  \psi_R (u)  \leq 1$ for all $u\in \RR$, it is easy to see that $0\leq  g_i(\btheta )  \leq    (\tau / 4r)^2 \| \btheta - \btheta^* \|_{\bSigma, 2}^2$.
By Talagrand's inequality (see, e.g. Theorem~7.3 in \cite{B2003}), we have for any $x>0$,
\begin{align}
	 \Delta_r  \leq  \EE \Delta_r +  (\EE \Delta_r)^{1/2} \frac{\tau}{2r} \sqrt{\frac{x}{n}}  +  \sigma_n \sqrt{     \frac{  2 x }{n}  }  +  +  \frac{\tau^2}{16 r^2} \times \frac{ x }{3n}   \label{Delta.concentration}
\end{align}
where $\sigma_n^2 = \sup_{ \btheta  \in \Theta_0(r) } \EE g_i^2(\btheta ) / \| \btheta - \btheta^* \|_{\bSigma ,2}^4$. By \eqref{phi.bound}, $\EE  g_i^2(\btheta ) \leq  \EE \langle \bX_i, \btheta  - \btheta^*  \rangle^4 \leq \kappa^4  \| \btheta  - \btheta^* \|_{\bSigma,2}^4$, implying $\sigma_n \leq \kappa^2$.

To bound the expectation $\EE \Delta_r$, applying the symmetrization inequality for empirical processes, and by the connection between Gaussian and Rademacher complexities, we have $\EE \Delta_r \leq 2 \sqrt{\pi/2}  \, \EE \{ \sup_{ \btheta  \in \Theta_0(r)  } |  \mathbb{G}_{\btheta } |   \}$,
where
$$
	\mathbb{G}_{\btheta } := 	 \frac{1}{n} \sn \frac{g_i}{\|\btheta - \btheta^*  \|_{\bSigma,2}^2}   \varphi_{\tau \| \btheta  - \btheta^*  \|_{\bSigma ,2} / (2r)} ( \langle \bX_i , \btheta  - \btheta^*  \rangle ) \psi_{\tau/2}(\varepsilon_i)   ,
$$
and $g_i$ are IID standard normal random variables that are independent of the observed data.  For any $\btheta_0  \in \Theta_0(r)$,  it holds
\begin{align}
	& \EE^*\bigg\{  \sup_{ \btheta \in \Theta_0(r) } | \GG_{\btheta } | \bigg\}     \leq  \EE^* |  \GG_{\btheta_0  } | +  2\EE^*\bigg\{ \sup_{ \btheta  \in \Theta_0(r)  }  \GG_{\btheta  }   \bigg\}    , \label{mean.G.1}
\end{align}
where $\EE^*$ denotes the conditional expectation given $\{ (Y_i,\bX_i) \}_{i=1}^n$. Taking the expectation with respect to $\{ (Y_i,\bX_i) \}_{i=1}^n$ on both sides, we see that \eqref{mean.G.1} remains valid with $\EE^*$ replaced by $\EE$.
To select a proper $\btheta_0$, first decompose $\btheta^*$ as $(\theta_0, \wt \btheta^{*\T})^\T$, where $\theta_0$ denotes the first coordinate of $\btheta^*$ and $\wt \btheta^{*} \in \RR^{d-1}$ consists of the remaining. Taking $\btheta_0 = (  \theta_0 + \sigma_{11}^{-1/2} r, \wt \btheta^{*\T}  )^\T$, we observe that $\| \btheta_0 - \btheta^* \|_{\bSigma, 2} = r$.
Since $\varphi_R(u) \leq \min(u^2, R^2/4)$, it holds
$$
\EE |  \GG_{\btheta_0  } | \leq (\EE  \GG_{\btheta_0  }^2)^{1/2}  \leq \frac{\tau }{4r\sqrt{n}} .
$$
As in the proof of Lemma~11 in \cite{LW2015}, we next use the Gaussian comparison theorem to bound the expectation of the (conditional) Gaussian supremum $\EE^*\{  \sup_{\Theta_0(r) }    \mathbb{G}_{\btheta }  \}$.

Let $\var^*$ be the conditional variance given $\{ (Y_i,\bX_i) \}_{i=1}^n$.
For  $ \btheta  , \btheta' \in \Theta_0(r)$, write $\bdelta = \btheta  - \btheta^*$ and $\bdelta' = \btheta'-\btheta^*$. Then
\begin{align}
	& \var^*(\GG_{\btheta } - \GG_{\btheta' }) \nn \\
	& \leq     \frac{1}{n^2} \sn \psi^2_{\tau/2}(\varepsilon_i)   \Bigg\{  \frac{\varphi_{\tau \|\bdelta \|_{\bSigma ,2}/(2r)} (\bX_i^\T \bdelta ) }{\| \bdelta \|_{\bSigma ,2}^2 }  - \frac{  \varphi_{\tau \|\bdelta' \|_{\bSigma ,2} /(2r)} (\bX_i^\T \bdelta' ) }{\| \bdelta' \|_{\bSigma,2}^2 }     \Bigg\}^2.   \nn
\end{align}
Note that $\varphi_{c R}(cu) = c^2 \varphi_{R}(u)$ for any $c>0$.  In particular, taking $R= \tau \|\bdelta' \|_{\bSigma,2}/(2r)$ and $c= \| \bdelta \|_{\bSigma,2} / \| \bdelta' \|_{\bSigma,2}$ delivers
$$
	 \varphi_{\tau \|\bdelta' \|_{\bSigma,2}/(2r)} (\bX_i^\T \bdelta' ) =  \frac{\| \bdelta' \|_{\bSigma,2}^2 }{ \|\bdelta \|_{\bSigma,2}^2}  \varphi_{\tau \|\bdelta  \|_{\bSigma,2}/(2r)} \bigg( \frac{\| \bdelta \|_{\bSigma,2} }{ \| \bdelta' \|_{\bSigma,2}} \bX_i^\T \bdelta' \bigg).
$$
Putting the above calculations together, we obtain
\begin{align}
&\var^*(\GG_{\btheta } - \GG_{\btheta'})  \nn \\
& \leq     \frac{1}{n^2} \sn \frac{1}{\| \bdelta \|_{\bSigma,2}^4 } \bigg\{  \varphi_{\tau \| \bdelta \|_{\bSigma,2}/(2 r)} (\bX_i^\T \bdelta )  -   \varphi_{\tau \|\bdelta  \|_{\bSigma,2}/(2r)} \bigg( \frac{\| \bdelta \|_{\bSigma,2} }{ \| \bdelta' \|_{\bSigma,2}} \bX_i^\T \bdelta' \bigg) \bigg\}^2   \nn \\
& \leq  \frac{1}{n^2} \sn \frac{1}{\| \bdelta \|_{\bSigma,2}^4} \frac{\tau^2 \| \bdelta \|_{\bSigma,2}^2}{4 r^2}  \bigg(  \bZ_i^\T \bdelta - \frac{\| \bdelta \|_{\bSigma,2} }{\| \bdelta'\|_{\bSigma,2}} \bZ_i^\T \bdelta'  \bigg)^2 \nn \\
& \leq     \frac{1}{n^2 } \sn  \frac{\tau^2 }{4 r^2} \bigg(  \frac{\bX_i^\T \bdelta }{\| \bdelta \|_{\bSigma,2}} - \frac{\bX_i^\T \bdelta'}{\| \bdelta'\|_{\bSigma,2}} \bigg)^2 . \label{var.G.1}
\end{align}
Next, define another (conditional) Gaussian process indexed by $\btheta$:
\begin{align}
	\ZZ_{\btheta  } :=  \frac{\tau }{2r n}  \sn g_i' \frac{\bX_i^\T(\btheta  - \btheta^*)}{\| \btheta  - \btheta^*  \|_{\bSigma,2} } ,  \nn
\end{align}
where $g'_i $ are IID standard normal random variables that are independent of all other random variables. By \eqref{var.G.1},
$ \var^*(\GG_{\btheta } - \GG_{\btheta' })  \leq \var^* (\ZZ_{\btheta } - \ZZ_{\btheta'} )$. By the Gaussian comparison inequality \citep{LT1991},
\begin{align}
	 &  \EE^*\bigg\{ \sup_{ \btheta   \in \Theta_0(r) } \GG_{\btheta  } \bigg\} \leq 2 \EE^*  \bigg\{  \sup_{ \btheta  \in \Theta_0(r) } \ZZ_{\btheta } \bigg\} \leq     \frac{  \tau   }{  r }     \EE^* \bigg\| \frac{1}{n} \sn g_i  \bZ_i \bigg\|_2  . \nn
\end{align}
Together with the unconditional version of \eqref{mean.G.1}, this implies
\begin{align}
	\EE  \Delta_r  \leq \sqrt{2\pi} \bigg(  \frac{ 2 \tau}{r}  \EE \bigg\| \frac{1}{n} \sn g_i     \bZ_i \bigg\|_2 + \frac{ \tau }{4r\sqrt{n}} \bigg)  \leq \sqrt{2\pi} \bigg(  \frac{ 2 \tau}{r} \sqrt{\frac{d}{n}} + \frac{ \tau }{4r\sqrt{n}} \bigg) . \nn
\end{align}
Combining this with \eqref{Delta.concentration}  and \eqref{Lower.bound.1}, we obtain that with probability at least $1-e^{-t}$,
\begin{align}
	\frac{\cT(\btheta )}{\| \btheta  -\btheta^* \|^2_{\bSigma , 2} } \geq \frac{1}{4}  ~\mbox{ uniformly over }  \btheta \in \Theta_0(r)  \nn
\end{align}
for all sufficiently large $n$ that scales as $(\tau /r)^2  (d+t)$ up to an absolute constant. This proves \eqref{RSC.lbd}.   \qed

\subsection{Proof of Theorem~\ref{boot.concentration.thm}}

Throughout we assume $t\geq 1$ and keep the notations used in the proof of Theorem~\ref{br.thm}.

\medskip
\noindent
{\sc Proof of \eqref{boot.est.deviation}}. Recall the weighted loss function $\mathcal{L}_\tau^\B(\btheta) = \sn  W_i  \ell_\tau(Y_i - \bX_i^\T \btheta)$, $\btheta \in \RR^d$ and the parameter set $\Theta_0(r)$ defined in \eqref{set.def}. Write $r_1 = 4r_0$ for $r_0$ as in \eqref{quadratic.bound}, so that
$$
	\PP \{ \hat{\btheta}_\tau \in \Theta_0(r_1) \} \geq 1-  2 e^{-t} ~\mbox{ provided } n \geq C_1  (d+t)
$$
for some $C_1 = C_1(A_0)>0$. By the definition of $\hat{\btheta}^\B_\tau$, $\cL_\tau^\B(\hat{\btheta}^\B_\tau)  - \cL_\tau^\B(\hat \btheta_\tau) \leq  0$ and
\begin{align}
	  \|   \hat{\btheta}_\tau^\B - \btheta^*   \|_{\bSigma, 2}   \leq \|   \hat{\btheta}_\tau^\B - \hat{\btheta}_\tau  \|_{\bSigma, 2} + \|  \hat{\btheta}_\tau - \btheta^*   \|_{\bSigma, 2} \leq   R_1 + \|  \hat{\btheta}_\tau - \btheta^*  \|_{\bSigma, 2} , \nn
\end{align}
where $R_1 := \overline \lambda_{\bSigma}^{1/2} R$. If we can show that, for some $r_2 \geq r_1$ to be specified,
\begin{align}
	   \cL_\tau^\B(\btheta) -  \cL_\tau^\B(\hat \btheta_\tau)  > 0 ~\mbox{ for all }  \btheta \in \partial \Theta_0(r) \mbox{ with } r_2 \leq r \leq r_2 + R_1  \nn
\end{align}
with high probability, then we must have $\hat{\btheta}^\B_\tau \in \Theta_0(r_2)$ with high probability. Here and below, we set $\partial \Theta_0(r) := \{ \btheta \in \RR^d: \| \btheta -\btheta^* \|_{ \bSigma, 2} =r \}$.

Centering the weighted Huber loss function, we define
\begin{align}
	\bzeta^\B(\btheta)= \cL^\B_\tau(\btheta) - \EE^*\{\cL^\B_\tau(\btheta) \} = \cL^\B_\tau(\btheta) - \cL_\tau(\btheta). \label{def:zetab}
\end{align}
Note that
\begin{align}
	& \cL_\tau^\B(\btheta) - \cL_\tau^\B(\hat \btheta_\tau ) \nn\\
	& =  \bzeta^\B(\btheta) - \bzeta^\B(\hat{\btheta}_\tau) + \cL_\tau(\btheta) - \cL_\tau(\btheta^*) + \cL_\tau(\btheta^* ) - \cL_\tau(\hat{\btheta}_\tau) \nn \\
	& \geq  \underbrace{\bzeta^\B(\btheta) - \bzeta^\B(\hat{\btheta}_\tau) }_{\Pi_1(\btheta)} + \underbrace{ \cL_\tau(\btheta) - \cL_\tau(\btheta^*)  }_{\Pi_2(\btheta)}. \label{diff.lbd1}
\end{align}
In the following, we bound $\Pi_1(\btheta)$ and $\Pi_2(\btheta)$ separately, starting with the latter which only depends on the observed data. As before, define $\bzeta(\btheta) = \cL_\tau(\btheta) - \EE\{\cL_\tau(\btheta)\}$ and  consider the decomposition
\begin{align}
	 \Pi_2(\btheta)  & = \underbrace{\bzeta(\btheta) - \bzeta(\btheta^*) - (\btheta - \btheta^*)^\T \nabla \bzeta(\btheta^*) }_{\Pi_{21}(\btheta)} \nn \\
& \quad  + \underbrace{(\btheta - \btheta^*)^\T \nabla \bzeta(\btheta^*)}_{\Pi_{22}(\btheta)} + \underbrace{\EE \{\cL_\tau(\btheta) - \cL_\tau(\btheta^*)\}}_{\Pi_{23}(\btheta)} .\label{Pi2.dec}
\end{align}

First we deal with $\Pi_{21}(\btheta)$. For every $r>0$, define the random process
\begin{align}
	U_r(\btheta) =  \frac{1}{r\sqrt{n}} \{  \bzeta(\btheta) - (\btheta - \btheta^*)^\T \nabla \bzeta(\btheta^*) \}, \ \ \btheta \in \Theta_0(r) . \label{define.Uprocess}
\end{align}
We will use Theorem~A.1 in \cite{S2013} to bound the local fluctuation $| U_r(\btheta) - U_r(\btheta^*) |$ over $\btheta \in \Theta_0(r)$. For any random variable $X$, we write $(\II  - \EE)X = X- \EE(X)$.
For every $\btheta \in \Theta_0(r )$, $\bv \in \RR^d$ and $\lambda \in \RR$, putting $\overline{\bv} = \bSigma^{1/2} \bv / \|  \bv\|_{ \bSigma ,2}$ and $\bdelta_r = \bSigma^{1/2}(\btheta - \btheta^*) / r$, and by the mean value  theorem, we have
\begin{align}
	& \EE \exp\bigg\{ \lambda \frac{\bv^\T \nabla U_r(\btheta)}{\|   \bv  \|_{ \bSigma ,2} } \bigg\} \nn \\
	& = \EE \exp\bigg\{  \frac{\lambda}{  \sqrt{n} } \sn (\II - \EE) \ell''_\tau(Y_i - \bX_i^\T \wt \btheta_i )  \overline{\bv}^\T \bZ_i \bZ_i^\T \bdelta_r \bigg\} \nn \\
	& \leq \prod_{i=1}^n  \bigg( 1 + \frac{\lambda^2}{ n }  \EE \Big[  \{  ( \overline{\bv}^\T \bZ_i  \bZ_i^\T \bdelta_r  )^2 + (\EE |\overline{\bv}^\T \bZ_i  \bZ_i^\T \bdelta_r | )^2 \} \nn \\
	& \quad \quad\quad \quad\quad\quad\quad \quad\quad\quad\quad\quad\quad \times e^{\frac{|\lambda|}{\sqrt{n} } (  |\overline{\bv}^\intercal \bZ_i  \bZ_i^\intercal \bdelta_r | + \EE |\overline{\bv}^\intercal \bZ_i  \bZ_i^\intercal \bdelta_r | )  } \Big] \bigg) \nn \\
	& \leq \prod_{i=1}^n \bigg\{ 1 +  \frac{\lambda^2  }{n}  e^{\frac{|\lambda |}{\sqrt{n}} } \EE  e^{\frac{|\lambda |}{\sqrt{n} } |\overline{\bv}^\intercal \bZ_i  \bZ_i^\intercal \bdelta_r | }   + \frac{\lambda^2  }{n} e^{\frac{|\lambda |}{\sqrt{n}} }  \EE (\overline{\bv}^\T \bZ_i \bZ_i^\T \bdelta_r )^2 e^{\frac{|\lambda |}{\sqrt{n} } |\overline{\bv}^\intercal \bZ_i  \bZ_i^\intercal \bdelta_r | } \bigg\}  \nn \\
	& \leq  \exp  \bigg\{ e^{\frac{|\lambda |}{\sqrt{n}}}  \lambda^2 \max_{\bu\in \mathbb{S}^{d-1}} \EE e^{\frac{|\lambda |}{\sqrt{n}} (\bu^\intercal \bZ)^2} +  e^{\frac{|\lambda |}{\sqrt{n}}}  \lambda^2 \max_{\bu\in \mathbb{S}^{d-1}} \EE (\bu^\intercal \bZ)^4 e^{\frac{|\lambda |}{\sqrt{n}} (\bu^\intercal \bZ)^2} \bigg\} . \nn
\end{align}
Similarly to \eqref{exp.moment.bound}, it can be shown that for all $|\lambda | \leq c_1 \sqrt{n}$ and $\btheta \in \Theta_0(r)$,
\begin{align}
	\EE \exp\bigg\{ \lambda \frac{\bv^\T \nabla U_r(\btheta)}{\| \bSigma^{1/2} \bv  \|_2 } \bigg\}  \leq \exp( C_2^2 \lambda^2 /2 ). \nn
\end{align}
Using Theorem~A.1 in \cite{S2013}, we deduce that with $\PP^{\dagger}$-probability at least $1-e^{-t}$,
\begin{align}
\sup_{\btheta \in \Theta_0(r)} | U_r(\btheta ) - U_r(\btheta^*)  |  \leq 3 C_2   (4d+2t)^{1/2} r .\nn
\end{align}
as long as $n \geq c_1^{-2}(4d+2t)$. In view of \eqref{Pi2.dec} and \eqref{define.Uprocess}, it holds for every $r>0$ that
\begin{align}
	\sup_{\btheta \in \Theta_0(r)} | \Pi_{21}(\btheta) |  \leq  3 C_2  (4d+2t)^{1/2}  r^2 \sqrt{n}  \label{Pi21.bound}
\end{align}
with $\PP^{\dagger}$-probability at least $1-e^{-t}$. The bound in \eqref{Pi21.bound} holds for any given $r>0$. Following the slicing argument similar to that used in the proof of Theorem~A.2 in \cite{S2013}, it can be shown that with $\PP^{\dagger}$-probability at least $1- e^{-t}$,
\begin{align}
	 \sup_{\btheta \in \Theta_0(r)}
	| \Pi_{21}(\btheta) |  \leq 6 C_2\{ 4d + 2t + 2 \log(2r/r_2) \}^{1/2}  r^2 \sqrt{n}  \label{Pi21.unif.bound}
\end{align}
for all $r_2 \leq r \leq r_2 + R_1$ as long as $n\geq c_1^{-2} \{4d + 2t +2 \log(2+2 R_1/r_2)\}$.

For $\Pi_{22}(\btheta)$, note that
\begin{align}
	  \sup_{\btheta \in \Theta_0(r)} | \Pi_{22}(\btheta) |  \leq   r  \| \bSigma^{-1/2} \{ \nabla \cL_\tau(\btheta^*) - \EE \nabla \cL_\tau(\btheta^*)   \}  \|_2 = r n  \| \bgamma \|_2  \nn
\end{align}
for $\bgamma$ as given in \eqref{L2.dec}. This, together with \eqref{xi.concentration}, implies that with $\PP^{\dagger}$-probability at least $1- e^{-t}$,
\begin{align}
  \sup_{\btheta \in \Theta_0(r)} | \Pi_{22}(\btheta) | \leq C_3 v  ( d+  t)^{1/2} r \sqrt{n} ~\mbox{ for any } r>0, \label{Pi22.bound}
\end{align}
where $C_3= C_3(A_0)>0$.

Turning to $\Pi_{23}(\btheta)$, we define the function $h(\btheta) = (1/n) \EE \{ \cL_\tau(\btheta)\}$, $\btheta \in \RR^d$ so that $\Pi_{23}(\btheta) = n\{ h(\btheta) - h(\btheta^*) \}$. Put $\bdelta =\bSigma^{1/2}( \btheta - \btheta^*)$. By \eqref{Zi.def} and the mean value theorem, it follows that
\begin{align}
	h(\btheta) = h(\btheta^*) - \EE\{ \psi_\tau(\varepsilon ) \bZ^\intercal \bdelta \} + \frac{1}{2 }   \EE \{ \ell''_\tau(Y  - \bX^\T \wt \btheta) ( \bZ^\intercal \bdelta )^2 \} , \nn
\end{align}
where $\wt \btheta$ is a point lying between $\btheta$ and $\btheta^*$. Since $\EE(\varepsilon | \bZ )=0$, we have $-\EE\{\psi_\tau(\varepsilon)| \bZ\} = \EE[ \{  \varepsilon - \tau \sgn(\varepsilon)\} I(|\varepsilon | > \tau )  | \bZ ]$, which further implies
$$
	|  \EE\{\psi_\tau(\varepsilon ) \bZ^\intercal \bdelta \}  | \leq  \upsilon_4 \tau^{-3} \EE( |\bZ^\intercal \bdelta |) \leq   \upsilon_4 \tau^{-3} \|\bdelta \|_2.
$$
Moreover,
\begin{align}
	 & \EE \{ \ell''_\tau(Y  - \bX^\T \wt \btheta) ( \bZ^\intercal \bdelta )^2 \} \nn \\
	 &  = \EE  ( \bZ^\intercal \bdelta )^2 - \EE \{ I(|Y  - \bX^\T \wt \btheta |>\tau)( \bZ^\intercal \bdelta )^2 \} \nn \\
	 & \geq  ( 1- \sigma^2 \tau^{-2} ) \| \bdelta \|_2^2 - \tau^{-2} \EE \{\bX^\T(\btheta^* - \wt \btheta) \}^2 (\bZ^\T \bdelta)^2 \nn \\
	 &  \geq   ( 1- \sigma^2 \tau^{-2} ) \| \bdelta \|_2^2 - \kappa^4 \tau^{-2}   \| \bdelta \|_2^4, \nn
\end{align}
where $\kappa>0$ is as in Proposition~\ref{prop:RSC}.
Putting the above calculations together yields that for any $\btheta \in \partial \Theta_0(r)$,
\begin{align}
	\frac{1}{n} \Pi_{23}(\btheta) \geq ( 1 - \sigma^2 \tau^{-2} - \kappa^4 \tau^{-2}  r^2 ) \frac{r^2}{2} -  \upsilon_4 \tau^{-3}r =  \frac{1}{2} b(r) r^2 ,  \label{mean.diff.bound}
\end{align}
where $b(r) := 1 - \sigma^2 \tau^{-2} - \kappa^4 \tau^{-2}  r^2  - 2  \upsilon_4 \tau^{-3} r^{-1}$, $r>0$.


Combining \eqref{Pi2.dec}, \eqref{Pi21.unif.bound}, \eqref{Pi22.bound} and \eqref{mean.diff.bound}, it follows that with $\PP^{\dagger}$-probability at least $1 - 2 e^{-t}$,
\begin{align}
	    \Pi_2(\btheta)    \geq   r^2 n  \Bigg\{ \frac{1}{2} b(r) - 6 C_2   \sqrt{ \frac{ 4d+2t + 2\log(2r/r_2) }{n}}  - C_3  \frac{ v (d+t)^{1/2}}{r \sqrt{n}} \Bigg\} \label{Pi2.bound}
\end{align}
for all $\btheta \in \partial \Theta_0(r)$ with $r\in [r_2, r_2 + R_1]$ provided $n\geq c_1^{-2} \{4d + 2t + 2\log(2+2 R_1/r_2)\}$.

Next we deal with the process $\Pi_1(\btheta)=\bzeta^\B(\btheta) - \bzeta^\B(\hat{\btheta}_\tau)$ in \eqref{diff.lbd1}, where
$$
	\bzeta^\B(\btheta) = \sn \ell_\tau(Y_i - \bX_i^\T \btheta) (W_i - 1), \ \ \btheta \in \RR^d.
$$
Decompose $\Pi_1(\btheta)$ as
\begin{align}
	\Pi_1(\btheta) = \underbrace{  \bzeta^\B(\btheta) - \bzeta^\B(\hat{\btheta}_\tau) - (\btheta - \hat{\btheta}_\tau)^\T \nabla \bzeta^\B(\hat{\btheta}_\tau)  }_{\Pi_{11}(\btheta)} + \underbrace{  (\btheta - \hat{\btheta}_\tau)^\T \nabla \bzeta^\B(\hat{\btheta}_\tau)  }_{\Pi_{12}(\btheta)} . \label{Pi1.dec}
\end{align}
We will use a conditional version of Theorem~A.1 in \cite{S2013} to bound $\Pi_{11}(\btheta)$. Similarly to \eqref{define.Uprocess}, define, for each $r>0$,
\begin{align}
	U^\B_r(\btheta) =  \frac{1}{r\sqrt{n}}  \{  \bzeta^\B(\btheta) - (\btheta -  \hat{\btheta}_\tau)^\T \nabla \bzeta^\B( \hat{\btheta}_\tau) \}, \ \ \btheta \in \RR^d. \label{define.UBprocess}
\end{align}
Similarly to \eqref{set.def}, define
\begin{align}
	\hat \Theta_0(r) =  \{ \btheta \in \RR^d: \|   \btheta - \hat{\btheta}_\tau  \|_{\bSigma, 2} \leq r \}, \ \ r>0. \label{boot.set.def}
\end{align}
For every $\btheta \in \hat \Theta_0(r)$ and $\bv  \in \RR^d$, by the mean value theorem we have
$$
	\bv^\T \nabla 	U^\B_r(\btheta) =  \frac{1}{r \sqrt{n}} \bv^\T \{  \nabla \bzeta^\B( \btheta ) - \nabla \bzeta^\B( \hat{\btheta}_\tau) \} =  \frac{1}{r\sqrt{n}} \bv^\T \nabla^2 \bzeta^\B(\wt\btheta) (\btheta - \hat{\btheta}_\tau ),
$$	
where $\wt \btheta$ is a convex combination of $\btheta$ and $\hat{\btheta}_\tau$. Putting $\overline \bdelta_r = \bSigma^{1/2} (\btheta - \hat{\btheta}_\tau)/r$ and $\overline{\bv} = \bSigma^{1/2} \bv / \| \bSigma^{1/2} \bv \|_2$, we deduce that
\begin{align}
	& \EE^*  \exp  \bigg\{  \lambda \frac{\bv^\T \nabla U_r^\B(\btheta)}{\| \bSigma^{1/2} \bv \|_2 } \bigg\}  \nn \\
	& = \EE^*  \exp \bigg\{ \frac{\lambda}{r \sqrt{n} } \overline \bv^\T \bSigma^{-1/2} \nabla^2 \bzeta^\B(\wt \btheta) (\btheta - \hat{\btheta}_\tau)   \bigg\} \nn \\
	& = \prod_{i=1}^n \EE^* \exp \bigg\{ \frac{\lambda}{r \sqrt{n} } \overline \bv^\T \bSigma^{-1/2} \nabla_{\btheta}^2 \ell_\tau(Y_i - \bX_i^\T  \wt \btheta)  (\btheta - \hat{\btheta}_\tau)  U_i \bigg\}  \nn \\
	& = \prod_{i=1}^n \EE^* \exp\bigg\{ \frac{\lambda}{\sqrt{n}} \eta_i (\btheta, \bv) U_i \bigg\} , \label{boot.exp.bound1}
\end{align}	
where
\begin{align*}
& \eta_i(\btheta, \bv)  \\
& :=\overline{\bv}^\T \bSigma^{-1/2} \nabla_{\btheta}^2 \ell_\tau(Y_i - \bX_i^\T  \wt \btheta)  (\btheta - \hat{\btheta}_\tau)/r  =  I ( |Y_i - \bX_i^\T \wt \btheta |\leq \tau )  \overline{\bv}^\T  \bZ_i \bZ_i^\T  \overline \bdelta_r  .
\end{align*}
Under Condition~\ref{weight.cond}, it holds
\begin{align}
 & \EE^* \exp\bigg\{ \frac{\lambda}{\sqrt{n}} \eta_i (\btheta, \bv) U_i  \bigg\} \nn \\
 &  \leq  \exp\bigg\{ \frac{\lambda^2  }{2 n } B_U^2 \eta_i^2(\btheta, \bv)  \bigg\} \leq \exp\bigg\{ \frac{\lambda^2  }{2 n }  B_U^2 (\overline{\bv}^\T \bZ_i )^2 ( \overline \bdelta_r^\T \bZ_i )^2  \bigg\} , \nn
\end{align}
where $B_U = B_U(A_U)>0$. Plugging this into \eqref{boot.exp.bound1} shows
\begin{align}
	\EE^*  \exp  \bigg\{  \lambda \frac{\bv^\T \nabla U_r^\B(\btheta)}{\| \bSigma^{1/2} \bv \|_2 } \bigg\}   \leq \exp\bigg(\frac{ \lambda^2 }{2 }  B_U^2 M_{n,4} \bigg). \nn
\end{align}
With the above preparations in place, it follows from \eqref{define.UBprocess} and Theorem~A.1 in \cite{S2013} that for any $r>0$,
\begin{align}
	  \PP^* \Bigg\{  \sup_{\btheta \in \hat \Theta_0(r)} | \Pi_{11}(\btheta) |   \geq  3 B_U M_{n,4}^{1/2}   (4d+2t)^{1/2} r^2 \sqrt{n} \Bigg\}  \leq e^{-t} \nn
\end{align}
almost surely, where $M_{n,4}$ is given in \eqref{def:Mn}. Again, using the slicing technique and applying the preceding bound to each slice separately, we obtain that with $\PP^*$-probability at least $1- e^{-t}$,
\begin{align}
	\sup_{\btheta \in \hat \Theta_0(r)}  | \Pi_{11}(\btheta) |  \leq  6 B_U  M_{n,4}^{1/2} \{  4d + 2t + 2\log(2r/ r_1)  \}^{1/2}  r^2 \sqrt{n}  \nn
\end{align}
for all $r_1  \leq r\leq 2 r_2 + R_1$. Note that, on the event $\{ \hat{\btheta}_\tau \in \Theta_0(r_1) \}$ that occurs with $\PP^{\dagger}$-probability at least $1- 2 e^{-t}$,
$$
	\Theta_0(r) \subseteq \hat{\Theta}_0(r+ r_1).
$$
Combining the last two displays and taking $x=2t$ in Lemma~\ref{lem.design}, we obtain that with $\PP^{\dagger}$-probability at least $1- 3 e^{-t}$,
\begin{align}
	\PP^*  \Bigg\{  & \sup_{\btheta \in \Theta_0(r)}  | \Pi_{11}(\btheta) |     \leq  C_4   \sqrt{  d +  t +  \log(2 + 2r/r_1)   }  \nn \\
& \ \   \ \  \times (r + r_1)^2  \sqrt{n} ~\mbox{ for all } r_2 \leq r\leq r_2 + R_1 \Bigg\} \geq 1 -  e^{-t}  \label{Pi11.unif.bound}
\end{align}
as long as $n\geq C_0 (d  +t )^2$, where $C_0 = C_0(A_0)$ and $C_4 = C_4(A_0, A_U)$.

For $\Pi_{12}(\btheta)$ in \eqref{Pi1.dec}, it holds on the event $\{ \hat{\btheta}_\tau \in \Theta_0(r_1) \}$ that, for every $\btheta \in \Theta_0(r)$,
\begin{align}
   | \Pi_{12}(\btheta) |  & = | (\btheta -\btheta^* + \btheta^* - \hat{\btheta}_\tau)^\T \nabla \bzeta^\B (\hat{\btheta}_\tau) | \nn \\
 &  \leq   (r + r_1)  \| \bSigma^{-1/2} \nabla \bzeta^\B(\hat{\btheta}_\tau)  \|_2  \nn \\
 &   \leq  (r+ r_1)  \{   \| \bxi^\B(\hat{\btheta}_\tau) -\bxi^\B( {\btheta}^*)  \|_2  + \| \bxi^\B(\btheta^*)  \|_2 \}  \nn \\
 & \leq  (r+ r_1)  \Bigg\{ \sup_{\btheta \in \Theta_0(r_1)}  \| \bxi^\B( {\btheta} ) -\bxi^\B( {\btheta}^*)  \|_2 + \| \bxi^\B(\btheta^*)  \|_2 \Bigg\} , \label{Pi12.dec}
\end{align}
where $\bxi^\B(\btheta) =  \bSigma^{-1/2} \nabla \bzeta^\B( {\btheta} ) =  - \sn \psi_\tau(Y_i - \bX_i^\T \btheta)  U_i \bZ_i$ is as in \eqref{xiB.def}.

By Lemma~\ref{lem.bootfluctuation}, it holds for each $r>0$ that
\begin{align}
	\PP^* \Bigg\{ \sup_{\btheta \in \Theta_0(r)} \frac{1}{\sqrt{n}} \| \bxi^\B( {\btheta} )  -  \bxi^\B( {\btheta}^*)  \|_2  \geq  C M_{n,4}^{1/2}  (8 d   & +  2 t)^{1/2} r   \Bigg\}  \leq e^{-t}  \nn
\end{align}
almost surely. Combining this and Lemma~\ref{lem.design}, we see that, conditioning on the same event where \eqref{Pi11.unif.bound} holds,
\begin{align}
\PP^* \Bigg\{ \sup_{\btheta \in \Theta_0(r_1)} \frac{1}{\sqrt{n}}  \| \bxi^\B( {\btheta} ) - \bxi^\B( {\btheta}^*)  \|_2  \geq   C_5  (d  +  t)^{1/2}  r_1    \Bigg\}  \leq  e^{-t}   \label{Pi12.bound-1}
\end{align}
as long as $n\geq C_0(d  + t)^2$, where $C_5=C_5(A_0, A_U) >0$.

For $\| \bxi^\B(\btheta^*) \|_2$, taking $x= 2t$ in Lemmas~\ref{cov.concentration} and \ref{lem.l2bootscore}, we see that with $\PP^{\dagger}$-probability at least $1  -   e^{-t} $,
\begin{align}
  \PP^* \{ \| \bxi^\B(\btheta^*) \|_2 \geq C_6  (d+t)^{1/2} \sqrt{n}  \} \leq   e^{-t}, \label{Pi12.bound-2}
\end{align}
where $C_6=C_6(A_U)>0$.
Combining \eqref{Pi1.dec}, \eqref{Pi11.unif.bound}, \eqref{Pi12.dec}, \eqref{Pi12.bound-1} and \eqref{Pi12.bound-2}, we conclude that conditioning on some event that occurs with $\PP^{\dagger}$-probability at least $1- 4 e^{-t}$, it holds
\begin{align}
	& \sup_{\btheta \in \Theta_0(r) }  |  \Pi_1(\btheta) |  \nn \\
	& \leq r^2 n \Bigg[ C_4 (r+r_1)^2  \frac{\{ d+ t + \log(2+ 2r/r_1)\}^{1/2}}{r^2 \sqrt{n}} +  C_5    (r+ r_1 )   \frac{r_1  (d+t)^{1/2} }{r^2 \sqrt{n}}     \nn \\
	& \quad\quad\quad \quad \quad \ \ + C_6   (r+r_1) \frac{v (d+t)^{1/2}}{r^2 \sqrt{n}}  \,\Bigg] ~\mbox{ for all } r_2 \leq r \leq r_2 +R_1   \label{Pi1.unif.bound}
\end{align}
with $\PP^*$-probability at least $1-3 e^{-t}$ provided $n\geq  C_0 (d  +t)^2$.

Finally, combining \eqref{diff.lbd1}, \eqref{Pi2.bound} and \eqref{Pi1.unif.bound}, and taking
$$
	r_2 = C_7 v \sqrt{(d+t)/n} \geq r_1
$$
for some sufficiently large constant $C_7= C_7(A_0, A_U)>0$, we conclude that, conditioning on some event that occurs with $\PP^{\dagger}$-probability at least $1-5  e^{-t}$,
\begin{align}
  	\PP^* \bigg\{ \hat{\btheta}_\tau^\B \!\in\! \Theta_0(r_1\!+\!R_1)  \mbox{ and } \cL^\B_\tau(\btheta)\! >\! \cL_\tau^\B(\hat{\btheta}_\tau)  ,   \forall \btheta \!\in\! \partial \Theta_0(r) ,   r \!\in \! &\, [r_2 , r_2\!+\!R_1]  \bigg\} \nn \\
  	& \geq 1- 3 e^{-t} \nn
\end{align}
provided $n \geq C_0 (d+t)^2$ and $n\geq C_8 \overline \lambda_{\bSigma} $, where $C_8 = C_8(A_0)>0$. Reinterpret this we obtain \eqref{boot.est.deviation}.

\medskip
\noindent
{\sc Proof of \eqref{boot.br}}. An argument similar to that given in the proof of Theorem~\ref{br.thm} can be used to prove \eqref{boot.br}. Define the random process
\begin{align}
&	\bB^\B(\btheta , \btheta' )   =  \bSigma^{-1/2} n^{-1} \{ \nabla \cL_\tau^\B(\btheta) - \nabla \cL_\tau^\B(\btheta') \}   -   \bSigma^{1/2} (\btheta  - \btheta') \nn \\
	&  =    \bSigma^{-1/2}  n^{-1} \{ \nabla \bzeta^\B(\btheta) - \nabla \bzeta^\B(\btheta') + \nabla  \cL_\tau(\btheta) - \nabla  \cL_\tau( \btheta')    -  n \bSigma (\btheta  - \btheta') \} , \nn
\end{align}
where $\bzeta^\B(\cdot)$ is given in \eqref{def:zetab}. The stated result follows from a bound on
\begin{align}
	&   \sup_{\btheta,  \, \btheta' \in \Theta_0(r)} \| \bB^\B(\btheta, \btheta') \|_2 \nn \\
& \leq  \sup_{\btheta, \, \btheta' \in \Theta_0(r)} \|  \bB^\B(\btheta, \btheta') - \EE^* \{ \bB^\B(\btheta , \btheta') \} \|_2 + \sup_{\btheta , \, \btheta' \in \Theta_0(r)} \|  \EE^* \{ \bB^\B(\btheta, \btheta') \} \|_2 ,  \nn
\end{align}
and the facts that $\hat{\btheta}_\tau \in \Theta_0(r_1)$ and $\hat{\btheta}^\B_\tau \in \Theta_0(r_2)$ with high probability.

Note that $\EE^* \{ \bB^\B(\btheta, \btheta' ) \}  = \bB(\btheta) - \bB(\btheta')$ for $\bB(\cdot)$ as in \eqref{process.B}. It then follows from \eqref{Delta.r} that with $\PP^\dagger$-probability greater than $1- e^{-t}$,
\begin{align}
\sup_{\btheta ,  \, \btheta' \in \Theta_0(r)} \| \EE^* \{ \bB^\B(\btheta, \btheta') \} \|_2   \leq   2\delta(r) r    + C_9  ( d+t )^{1/2}  n^{-1/2} r  , \label{EBb.bound}
\end{align}
where $\delta(\cdot)$ is defined above \eqref{def.delta} and $C_9=C_{9}(A_0) >0$.

For $\bB^\B(\btheta , \btheta' ) - \EE^* \{\bB^\B(\btheta , \btheta')\}$, note that
\begin{align}
& \bB^\B(\btheta , \btheta' ) - \EE^* \{\bB^\B(\btheta , \btheta')\} \nn \\
& = \bSigma^{-1/2} n^{-1} \{ \nabla \bzeta^\B(\btheta) - \nabla \bzeta^\B(\btheta') \} \nn \\
& =  \bSigma^{-1/2} n^{-1} \{ \nabla \bzeta^\B(\btheta) - \nabla \bzeta^\B(\btheta^*) \} - \bSigma^{-1/2} n^{-1} \{ \nabla \bzeta^\B(\btheta') - \nabla \bzeta^\B(\btheta^*) \} .\nn
\end{align}
Since we are interested in the case where both $\btheta$ and $\btheta'$ are in a neighborhood of $\btheta^*$, it suffices to focus on $\btheta$. To proceed, we change the variable by $\bdelta = \bSigma^{1/2}( \btheta - \btheta^*)$ and define
\begin{align}
	  & \overline{\bB}^\B(\bdelta )= \bSigma^{-1/2} n^{-1} \{ \nabla \bzeta^\B(\btheta) - \nabla \bzeta^\B(\btheta^*) \} \nn \\
	  &   = -\frac{1}{n} \sn  \{ \psi_\tau( \varepsilon_i - \bZ_i^\T\bdelta) - \psi_\tau (\varepsilon_i  ) \} U_i  \bZ_i . \nn
\end{align}
It is easy to see that $\overline{\bB}^\B( \textbf{0}) = \textbf{0}$, $\EE^*\{ \overline{\bB}^\B(\bdelta )\} = \textbf{0}$ and
$$
	\nabla \overline{\bB}^\B(\bdelta) = \frac{1}{n} \sn  \ell''_\tau(\varepsilon_i - \bZ_i^\T \bdelta) U_i \bZ_i \bZ_i^\T  .
$$
For any $\bu , \bv \in \mathbb{S}^{d-1}$ and $\lambda \in \RR$, by Condition~\ref{weight.cond} we have
\begin{align}
	& \EE^* \exp \{  \lambda \sqrt{n} \, \bu^\T \nabla \overline{\bB}^\B(\bdelta) \bv  \} \nn \\
	& \leq \prod_{i=1}^n  \exp\bigg\{ \frac{\lambda^2 }{2n} B_U^2 ( \bu^\T \bZ_i )^2 (  \bv^{\intercal} \bZ_i )^2  \bigg\} \leq \exp\bigg( \frac{\lambda^2}{2}  B_U^2 M_{n,4}\bigg) ,\nn
\end{align}
where $B_U = B_U(A_U)>0$. Applying a conditional version of Theorem~A.1 in \cite{S2013} delivers
\begin{align}
	\PP^*\Bigg\{ \sup_{\bdelta \in \mathbb{B}^d(r)} \| \sqrt{n} \, \overline{\bB}^\B(\bdelta) \|_2 \geq  3 B_U  M_{n,4}^{1/2}  (4d +  2t)^{1/2}  r \Bigg\}  \leq   e^{-t}  \label{Bb.bound}
\end{align}
almost surely, where $M_{n,4}$ is given in \eqref{def:Mn}.

Finally, we take $(\btheta, \btheta') = (\hat{\btheta}_\tau^\B , \hat{\btheta}_\tau)$. By \eqref{concentration.MLE} and \eqref{boot.est.deviation}, $\hat{\btheta}_\tau \in \Theta_0(r_1)$ with probability at least $1-3e^{-t}$, and with $\PP^{\dagger}$-probability at least $1-5 e^{-t}$,
$$
	\PP^*\{ \hat{\btheta}^\B_\tau \in \Theta_0(r_2) \} \geq 1- 3 e^{-t}.
$$
Since $\nabla  \cL_\tau(\hat \btheta_\tau) = \textbf{0}$, it holds $\bSigma^{-1/2} \nabla \cL^\B_\tau(\hat{\btheta}_\tau)  = \bxi^\B(\hat{\btheta}_\tau)$, where $\bxi^\B(\cdot)$ is given in \eqref{xiB.def}. Then, on the event $\{ \hat{\btheta}_\tau \in \Theta_0(r_1) \}$, it holds
$$
	\| \bxi^\B(\hat{\btheta}_\tau) - \bxi^\B(\btheta^*) \|_2 \leq \sup_{\btheta \in \Theta_0(r_1) } \|  \bxi^\B(\btheta) - \bxi^\B(\btheta^*) \|_2 ,
$$
so that the bound in \eqref{Pi12.bound-1} can be applied. Moreover, by the triangle inequality, $\| \bSigma^{1/2} ( \hat{\btheta}^\B_\tau - \hat{\btheta}_\tau ) \|_2 \leq r_1 + r_2$ with high probability, which in turn implies that $\hat{\btheta}^\B_\tau$ falls in the interior of $\{ \btheta : \| \btheta - \hat \btheta_\tau \|_2 \leq R\}$ for all sufficiently large $n$ and hence $\nabla \cL^\B_\tau(\hat{\btheta}^\B_\tau) = \textbf{0}$. This, together with \eqref{EBb.bound}, \eqref{Bb.bound} and the definition of $\bB^{\B}(\hat{\btheta}_\tau^\B, \hat{\btheta}_\tau)$, proves \eqref{boot.br}.  \qed

\subsection{Proof of Theorem~\ref{boot.wilks.thm}}

The proof is based on a similar argument to that used in the proof of Theorem~\ref{wilks.thm}. To begin with, define the bootstrap random process: for $\btheta , \btheta' \in \Theta_0(r)$,
\begin{align}
	 &  R^\B(\btheta , \btheta' ) \nn \\
	 &  =   \cL^\B_\tau(\btheta ) - \cL^\B_\tau(\btheta') - (\btheta - \btheta')^\T \nabla \cL^\B_\tau(\btheta') - \frac{n}{2} \| \bSigma^{1/2} (\btheta -\btheta' ) \|_2^2 . \label{R.randomprocess}
\end{align}
By the mean value theorem, $R^\B(\btheta , \btheta' ) =  (\btheta - \btheta' )^\T \nabla_{\btheta} R^\B(\btheta, \btheta' ) |_{\btheta = \wt{\btheta}}$, where $\wt{\btheta}$ is a convex combination of $\btheta$ and $\btheta'$ and thus satisfies $\wt{\btheta} \in \Theta_0(r)$. It follows that
\begin{align}
R^\B(\btheta , \btheta' ) \leq 2 r \sup_{\btheta \in \Theta_0(r) }  \| \bG^\B(\btheta, \btheta')   \|_2 , \label{RB.bound}
\end{align}
where $\bG^\B(\btheta, \btheta'):= \bSigma^{-1/2} \nabla_{\btheta} R^\B(\btheta, \btheta' ) = \bSigma^{-1/2} \{  \nabla \cL^\B_\tau(\btheta) -  \nabla \cL^\B_\tau(\btheta')  \} - n   \bSigma^{1/2} (\btheta - \btheta ')$. To bound the right-hand side of \eqref{RB.bound}, we will deal with
$$
	\bD(\btheta ,\btheta') := \bG^\B(\btheta, \btheta') -\bG(\btheta, \btheta') ~\mbox{ and }~ \bG(\btheta, \btheta')
$$
separately, where $\bG(\btheta, \btheta')=\EE^* \{\bG^\B(\btheta, \btheta')\}  =  \bSigma^{-1/2} \{  \nabla \cL_\tau(\btheta) -  \nabla \cL_\tau(\btheta')  \} - n  \bSigma^{1/2} (\btheta - \btheta ')$. For the latter, we have
\begin{align}
	\sup_{\btheta , \btheta' \in \Theta_0(r)} \| \bG(\btheta, \btheta')  \|_2 \leq 2 n \Delta(r) , \label{RB.bound-1}
\end{align}
where $\Delta(r)$ is given in \eqref{Delta.r}. For the former term, note that
\begin{align}
   \bD(\btheta ,\btheta') = \sn \{  \psi_\tau(Y_i -\bX_i^\T \btheta) -\psi_\tau(Y_i -\bX_i^\T \btheta' ) \}   U_i \bZ_i  .  \nn
\end{align}
Define new variables $\bdelta = \bSigma^{1/2}(\btheta - \btheta^*)$ and $\bdelta' = \bSigma^{1/2}(\btheta' - \btheta^*)$, so that
\begin{align}
\sup_{\btheta ,\btheta' \in \Theta_0(r)} \| \bD(\btheta ,\btheta') \|_2 = \sup_{\bdelta, \bdelta' \in \mathbb{B}^d(r)} \| \overline \bD(\bdelta ,\bdelta') \|_2 \leq 2 \sup_{\bdelta \in \mathbb{B}^d(r)} \| \overline \bD(\bdelta , \textbf{0}) \|_2 , \label{process.D.bound}
\end{align}
where $\overline \bD(\bdelta ,\bdelta') = \sn \{ \psi_\tau(\varepsilon_i -\bZ_i^\T \bdelta ) - \psi_\tau(\varepsilon_i  -\bZ_i^\T \bdelta') \}   U_i \bZ_i$. It is easy to see that the random process $\{\overline \bD(\bdelta ,\textbf{0} ) , \bdelta \in  \mathbb{B}^d(r) \}$ satisfies $\overline \bD(\textbf{0} ,\textbf{0} ) = \textbf{0}$ and $\EE^* \{  \overline \bD(\bdelta , \textbf{0} )\} = \textbf{0}$. Moreover, for any $\bu, \bv \in \mathbb{S}^{d-1}$ and $\lambda\in \RR$,
\begin{align}
	& \EE^*  \exp\bigg\{ \frac{ \lambda }{\sqrt{n}} \bu^\T \nabla_{\bdelta} \overline\bD(\bdelta  , \textbf{0} ) \bv \bigg\} \nn \\
	& =   \prod_{i=1}^n \EE^*   \exp\bigg\{  - \frac{ \lambda}{\sqrt{n}} \ell''_\tau( \varepsilon_i - \bZ_i^\T \bdelta )(\bu^\T \bZ_i)(\bv^\T \bZ_i) U_i \bigg\} \nn \\
	& \leq  \prod_{i=1}^n  \exp\bigg\{ \frac{\lambda^2  }{2 n}B_U^2 (\bu^\T \bZ_i)^2 (\bv^\T \bZ_i)^2 \bigg\} \leq \exp\bigg( \frac{\lambda^2   }{2}  B_U^2 M_{n,4} \bigg) . \nn
\end{align}
It then follows from Theorem~A.3 in \cite{S2013} that
\begin{align}
	\PP^* \Bigg\{ \sup_{\bdelta \in \mathbb{B}^d(r)} \| \overline \bD(\bdelta , \textbf{0}) \|_2  \geq  6 B_U M_{n,4}^{1/2}  ( 8d  + 2t)^{1/2}  r \sqrt{n} \Bigg\} \leq e^{-t}. \label{process.D.cond.bound}
\end{align}

Together, the estimates \eqref{RB.bound}--\eqref{process.D.cond.bound} imply that, with $\PP^*$-probability at least  $1 - e^{-t}$,
\begin{align}
	\sup_{\btheta ,\btheta' \in \Theta_0(r)} \| R^\B(\btheta , \btheta' )   \|_2 \leq 4r \{ n \Delta(r) + 6 B_U M_{n,4}^{1/2}  ( 8d + 2t)^{1/2} r \sqrt{n}  \}.  \label{RB.cond.bound}
\end{align}

Recall the proof of Theorem~\ref{boot.concentration.thm} and note that
\begin{align}
	&  \bigg|   ( \hat \btheta_\tau^\B -  \hat \btheta_\tau)^\T \nabla \cL^\B_\tau(\hat \btheta_\tau) + \frac{n}{2}  \| \bSigma^{1/2} (\hat \btheta_\tau^\B - \hat \btheta_\tau  ) \|_2^2  + \frac{1}{2 n}  \| \bSigma^{-1/2} \nabla \cL_\tau^\B(\hat{\btheta}_\tau ) \|_2^2  \bigg| \nn \\
	& = \frac{1}{2 n}   \| \bG^\B( \hat \btheta_\tau^\B ,  \hat \btheta_\tau) \|_2^2 \nn \\
	&   \leq  \frac{1}{2n}(  \| \bD( \hat \btheta_\tau^\B ,  \hat \btheta_\tau) \|_2 + \| \bG( \hat \btheta_\tau^\B ,  \hat \btheta_\tau) \|_2 )^2.   \nn
\end{align}
This, together with \eqref{RB.bound-1}--\eqref{RB.cond.bound} and the proof of Theorem~\ref{boot.concentration.thm}, yields that, conditioning on some event that occurs  with $\PP^\dagger$-probability at least $1- 5 e^{-t}$,
$$
  M_{n,4}^{1/2}   \leq   C_{10}
$$
for some $C_{10} = C_{10}(A_0)>0$ and moreover, the following inequalities
\begin{align}
	& \bigg|  \cL^\B_\tau( \hat \btheta_\tau^\B ) - \cL^\B_\tau( \hat{\btheta}_\tau ) - ( \hat \btheta_\tau^\B -  \hat \btheta_\tau)^\T \nabla \cL^\B_\tau(\hat \btheta_\tau) - \frac{n}{2} \| \bSigma^{1/2} (\hat \btheta_\tau^\B - \hat \btheta_\tau  ) \|_2^2 \bigg| \nn \\
	& \qquad \qquad \qquad \qquad \qquad \leq   4 r_2 \{ n \Delta(r_2) + 6 B_U  C_{10} ( 8d + 2t)^{1/2} r_2  \sqrt{n} \}  \nn
\end{align}
and
\begin{align}
	&   \bigg|   ( \hat \btheta_\tau^\B -  \hat \btheta_\tau)^\T \nabla \cL^\B_\tau(\hat \btheta_\tau) + \frac{n}{2}  \| \bSigma^{1/2} (\hat \btheta_\tau^\B - \hat \btheta_\tau  ) \|_2^2  + \frac{1}{2 n}  \| \bSigma^{-1/2} \nabla \cL_\tau^\B(\hat{\btheta}_\tau ) \|_2^2  \bigg| \nn \\
 & \qquad \qquad \qquad \qquad \qquad  \leq     \frac{2}{n} \{  n \Delta(r_2) +    6 B_U  C_{10}  (8d+2t)^{1/2}  r_2  \sqrt{n} \}^2  \nn
\end{align}
hold with $\PP^*$-probability greater than $1- 4 e^{-t}$. Recall that $\nabla \cL_\tau(\hat{\btheta}_\tau) = \textbf{0}$ and by \eqref{xiB.def}, $\bSigma^{-1/2} \nabla \cL_\tau^\B(\hat{\btheta}_\tau )  =  \bxi^\B(\hat{\btheta}_\tau)$. It then follows that
\begin{align}
	& |  \| \bSigma^{-1/2} \nabla \cL_\tau^\B(\hat{\btheta}_\tau ) \|_2^2   -   \|  \bxi^\B(\btheta^*) \|_2^2  | \nn \\
	& = |  \|  \bxi^\B(\hat{\btheta}_\tau) \|_2^2   -   \|  \bxi^\B(\btheta^*) \|_2^2  |  \nn \\
	&  \leq  \| \bxi^\B(\hat{\btheta}_\tau) - \bxi^\B( \btheta^* ) \|_2 \{  \| \bxi^\B(\hat{\btheta}_\tau) - \bxi^\B( \btheta^* ) \|_2  + 2 \| \bxi^\B( \btheta^* )  \|_2 \} . \nn
\end{align}
Putting $\Delta^\B(r) = \sup_{\btheta\in \Theta_0(r)}  \| \bxi^\B( \btheta ) - \bxi^\B( \btheta^* ) \|_2$ and combining the last three displays, we conclude that
\begin{align}
	&    \bigg|  \cL^\B_\tau( \hat{\btheta}_\tau )  - \cL^\B_\tau( \hat \btheta_\tau^\B ) -  \frac{1}{2 n }  \|  \bxi^\B(\btheta^*) \|_2^2 \bigg| \nn \\
   & \leq 4 r_2 \{ n \Delta(r_2) + 6 B_U C_{10}  ( 8d + 2t)^{1/2} r_2  \sqrt{n} \}  \nn \\
   & \quad \, +   \frac{2}{n} \{  n \Delta(r_2) +    6 B_U C_{10} (8d+2t)^{1/2} r_2  \sqrt{n}  \}^2 \nn \\
   & \quad \, \, + \frac{ \Delta^\B(r_1) }{2n}\{ \Delta^\B(r_1) + 2 \| \bxi^\B( \btheta^* )  \|_2 \}  \leq   C_{11} v^2 \frac{ (d+t)^{3/2}}{\sqrt{n}}   ,  \label{boot.wilks.bound-1}
\end{align}
for some $C_{11}= C_{11}(A_0, A_U) >0$. This proves \eqref{boot.wilks} immediately.

For the square-root Wilks expansion, note that
\begin{align}
	& \bigg|  \sqrt{ 2 \{ \cL^\B_\tau( \hat{\btheta}_\tau )  - \cL^\B_\tau( \hat \btheta_\tau^\B )  \} } - \sqrt{n} \,  \| \bSigma^{1/2} (\hat{\btheta}_\tau - \hat{\btheta}_\tau^\B)  \|_2 \bigg|  \nn \\
	& \leq   \frac{ | 2\{ \cL^\B_\tau( \hat{\btheta}_\tau )  \!-\! \cL^\B_\tau( \hat \btheta_\tau^\B )  \}  \!- \! n  \,\| \bSigma^{1/2} (\hat{\btheta}_\tau \!- \! \hat{\btheta}_\tau^\B)  \|_2^2 | }{ \sqrt{n} \, \| \bSigma^{1/2} (\hat{\btheta}_\tau \!- \!\hat{\btheta}_\tau^\B)  \|_2 }  \!= \! \frac{2}{\sqrt{n}} \frac{  | R^\B(\hat{\btheta}_\tau, \hat{\btheta}_\tau^\B ) | }{  \| \bSigma^{1/2} (\hat{\btheta}_\tau \!- \! \hat{\btheta}_\tau^\B)  \|_2} , \nn
\end{align}
where $R^\B(\cdot, \cdot )$ is given in \eqref{R.randomprocess}. For any $r>0$, similarly to \eqref{RB.bound}, it holds
\begin{align}
 \sup_{\btheta, \btheta' \in \Theta_0(r)} \frac{| R^\B(\btheta, \btheta') | }{\|  \bSigma^{1/2} (\btheta - \btheta') \|_2} \leq \sup_{\btheta, \btheta' \in \Theta_0(r)} \|   \bG^\B(\btheta, \btheta') \|_2 .  \nn
\end{align}
Again, the estimates \eqref{RB.bound-1}--\eqref{process.D.cond.bound} imply that, with $\PP^*$-probability at least $1-e^{-t}$,
\begin{align}
	\sup_{\btheta, \btheta' \in \Theta_0(r)} \|   \bG^\B(\btheta, \btheta') \|_2    \leq 2 n \Delta(r) + 12 B_U  M_{n,4}^{1/2}  ( 8d + 2t)^{1/2}  r  \sqrt{n},  \nn
\end{align}
where $\Delta(r)$ is given in \eqref{Delta.r}. Recall that $\bG^\B( \hat{\btheta}_\tau , \hat{\btheta}_\tau^\B )   = \bSigma^{-1/2}   \nabla \cL^\B_\tau(\hat{\btheta}_\tau )   - n  \bSigma^{1/2} (\hat{\btheta}_\tau - \hat{\btheta}_\tau^\B) =  \bxi^\B( \hat{\btheta}_\tau)  - n  \bSigma^{1/2} (\hat{\btheta}_\tau - \hat{\btheta}_\tau^\B) $. Following the same argument that delivers \eqref{boot.wilks.bound-1}, we reach
\begin{align}
	& \bigg|  \sqrt{ 2 \{ \cL^\B_\tau( \hat{\btheta}_\tau )  - \cL^\B_\tau( \hat \btheta_\tau^\B )  \} } -  n^{-1/2} \|   \bxi^\B(\btheta^*) \|_2   \bigg|  \nn \\
& \leq  6 \sqrt{n} \,\Delta(r_2) + 36 B_U C_{10} ( 8d + 2t)^{1/2}  r_2  \nn \\
& \quad   +  \frac{1}{\sqrt{n}} \sup_{\btheta \in \Theta_0(r_1) } \| \bxi^\B(\btheta) - \bxi^\B(\btheta^*) \|_2   \leq C_{12}  v  \frac{d+t}{\sqrt{n}}, \nn
\end{align}
where $C_{12}=C_{12}(A_0, A_U)>0$. This is the bound stated in \eqref{boot.sqwilks}. \qed

\subsection{Proof of Theorem~\ref{Boot.consistency}}

We divide the proof into three steps. In the first step, we revisit the non-asymptotic square-root Wilks approximations for the excess loss and its bootstrap counterpart. The second step is on Gaussian approximation for the $\ell_2$-norm of the standardized score vector $ \bSigma^{-1/2} \nabla \cL_\tau(\btheta^*) $. The last step links the distributions of the excess loss and its bootstrap counterpart via a Gaussian comparison inequality. Without loss of generality, we assume $t\geq 1$ throughout the proof.

\medskip
\noindent
{\sc Step 1}~(Wilks approximations). Define $\bxi^* \!=\! - \sn \xi_i\bZ_i$
and recall that $\bxi^\B = - \sn \xi_i U_i \bZ_i$.

For any $x \geq 0$, it follows from \eqref{sqrt.Wilks.expansion} that
\begin{align}
	& \PP\bigg[ \sqrt{2 \{ \cL_\tau(\btheta^*) - \cL_\tau(\hat{\btheta}_\tau ) \} } \leq x  \bigg]   \leq \PP\bigg(  \frac{ \| \bxi^* \|_2 }{\sqrt{n}}\leq x + R_1  \bigg) + 3 e^{-t} , \label{sqwilks1}
\end{align}
where $R_1>0$ satisfies $R_1 \asymp  v (d+t) n^{-1/2}$. Similarly, applying \eqref{boot.sqwilks} yields that, with probability (over $\mathcal{D}_n$) at least $1-5e^{-t}$,
\begin{align}
	& \PP\bigg[ \sqrt{2 \{ \cL^\B_\tau( \hat \btheta_\tau ) - \cL^\B_\tau(\hat{\btheta}_\tau^\B ) \} } \leq x  \bigg| \mathcal{D}_n \bigg] \nn \\
	& \geq \PP\bigg\{ \frac{\| \bxi^\B \|_2}{\sqrt{n}}  \leq \max( x - R_2 , 0 ) \bigg| \mathcal{D}_n \bigg\}  - 4 e^{-t} ,  \label{sqwilks2}
\end{align}
where $R_2>0$ satisfies $R_2 \asymp v (d+t)n^{-1/2}$. In the following two steps, we validate the approximation of the distribution of $\| \bxi^\B \|_2$ by that of $\| \bxi^* \|_2$ in the Kolmogorov distance. To that end, define random vectors
$$
	\Sb_1 = \frac{1}{\sqrt{n}} \sn  \bV_i  ~\mbox{ and }~  \Sb_2= \frac{1}{\sqrt{n}} \sn U_i \bV_i     ~\mbox{ with }~ \bV_i = \xi_i \bZ_i, \ \  \xi_i = \psi_\tau(\varepsilon_i) .
$$
In this notation, we have $\| \bxi^* \|_2 = \sqrt{n} \| \Sb_1 \|_2$ and $\| \bxi^\B \|_2 = \sqrt{n} \| \Sb_2 \|_2$.


\medskip
\noindent
{\sc Step 2}~(Gaussian approximation for $\| \Sb_1 \|_2$). Recall the truncated mean and second moment $m_\tau = \EE(\xi_i)$ and $\sigma^2_\tau = \EE( \xi_i^2)$, and consider the centered sum
$$
	\overline{\Sb}_1 =  \frac{1}{\sqrt{n}} \sn ( \bV_i - \EE \bV_i  ) = \frac{1}{\sqrt{n}} \sn (\bV_i - m_\tau \bupsilon ) = \Sb_1 - \sqrt{n} \, m_\tau \bupsilon ,
$$
where $\bupsilon = \EE(\bZ)$ is such that $\| \bupsilon \|_2 \leq 1$. Here, $\bV_1, \ldots, \bV_n$ are independent copies of the random vector $\bV = \psi_\tau(\varepsilon) \bZ \in \RR^d$ with mean $m_\tau \bupsilon$ and covariance matrix $\bSigma_1 = \sigma^2_\tau \, \bI_d - m_\tau^2 \,\bupsilon \bupsilon^\T$. For $(m_\tau, \sigma_\tau^2)$, applying Proposition~A.2 with $\kappa=4$ in the supplement of \cite{ZBFL2017} gives
\begin{align}
	| m_\tau | \leq  \upsilon_4 \tau^{-3} ~\mbox{ and }~  \sigma^2  -  \upsilon_4 \tau^{-2}  \leq  \sigma_\tau^2 \leq \sigma^2  . \nn
\end{align}
Hence, for any $\bu \in \mathbb{S}^{d-1}$, it holds
\begin{align}
	  \sigma^2 ( 1  - \upsilon_4 \sigma^{-2} \tau^{-2}  - \upsilon_4^2 \sigma^{-2} \tau^{-6}   )  \leq  \| \bSigma_1^{1/2} \bu \|_2^2  \leq \sigma^2 . \nn
\end{align}
Taking $\tau = v \{ n/(d+t) \}^{\eta}$ for $v\geq \upsilon_4^{1/4}$, this implies $\overline \lambda_{\bSigma_1} \leq \sigma^2$ and
\begin{align}
 \underline	\lambda_{ \bSigma_1 } \geq  \sigma^2 \left\{   1  -  \frac{ \upsilon_4^{1/2} }{\sigma^{2}  } \bigg( \frac{d+t}{n} \bigg)^{2\eta}  - \frac{ \upsilon_4^{1/2} }{\sigma^2}  \bigg( \frac{d+t}{n} \bigg)^{6\eta}   \right\} \geq   \frac{1}{2} \sigma^2 \nn
\end{align}
provided $n\geq   ( 4\upsilon_4^{1/2} / \sigma^2)^{1/(2\eta)} (d+t)$. Also, under this scaling condition, it holds $\sigma_\tau^2 \geq 3\sigma^2 /4$. It then follows from a multivariate central limit theorem \citep{B2005} that
\begin{align}
	& \sup_{ y \geq 0} | \PP (  \| \overline{\Sb}_1   \|_2 \leq y  ) - \PP( \| \bG_1   \|_2 \leq y ) |  \nn \\
 & \leq \frac{C_1}{\sqrt{n}}  \EE \{ \| \bSigma_1^{-1/2} (\bV - m_\tau \bupsilon ) \|_2^3  \} \leq C_2   \max_{1\leq j\leq d} \EE ( | Z_j|^3 ) \frac{\upsilon_3}{ \sigma^3  } \frac{d^{3/2}}{\sqrt{n}}   , \label{m.clt.1}
\end{align}
where $\bG_1 \sim \mathcal{N}(\textbf{0}, \bSigma_1)$ and $C_1, C_2>0$ are absolute constants.

Let $\bG_0 \sim \mathcal{N}(\textbf{0},  \bSigma_0 )$ with $ \bSigma_0 := \sigma_\tau^2 \, \bI_d$. Note that $ \bSigma_0^{-1/2}  \bSigma_1 \bSigma_0^{-1/2} \!-\! \bI_d \!=\!  m_\tau^2 \sigma^{-2}_\tau \bupsilon \bupsilon^\T$,
\begin{align}
	\|  m_\tau^2 \sigma^{-2}_\tau \bupsilon \bupsilon^\T \|_2 \leq \delta_\tau :=  \upsilon_4^2 \sigma_\tau^{-2} \tau^{-6}  ~\mbox{ and }~
	\tr \{( m_\tau^2 \sigma^{-2}_\tau \bupsilon \bupsilon^\T )^2 \} \leq \delta_\tau^2 . \nn
\end{align}
Applying Lemma~A.7 in the supplementary material of \cite{SZ2015} gives
\begin{align}
\sup_{ y \geq 0} | \PP (  \| \bG_1 \|_2 \leq  y ) - \PP ( \| \bG_0 \|_2 \leq y ) |  \leq \delta_\tau /2 \label{normal.compare}
\end{align}
provided $\delta_\tau\leq 1/2$. In addition, the Gaussian random vector $\bG_0$ satisfies the following anti-concentration inequality \citep{B1993}: for any $ \epsilon \geq 0$,
\begin{align}
	\sup_{y \geq 0}  \PP (  y \leq \| \bG_0  \|_2 \leq y +\epsilon ) \leq  \frac{C_3}{\sigma_\tau} \epsilon , \label{anti.concentration}
\end{align}
where $C_3>0$ is an absolute constant.

For the deterministic term $\EE(\Sb_1) =\sqrt{n} \, m_\tau \bupsilon$, we have
$$
	\gamma_1 := \| \EE(\Sb_1) \|_2 \leq  \upsilon_4 \tau^{-3} \sqrt{n}  \leq \upsilon_4^{1/4}  \frac{(d+t)^{3\eta} }{n^{3\eta - 1/2}} .
$$
Combining this with \eqref{m.clt.1}--\eqref{anti.concentration}, we arrive at
\begin{align}
	& \PP\bigg( \frac{\| \bxi^* \|_2}{\sqrt{n}}  \leq x + R_1 \bigg) = \PP ( \| \Sb_1 \|_2 \leq x+ R_1 ) \nn \\
	& \leq \PP( \| \overline  \Sb_1 \|_2 \leq x+ R_1 + \gamma_1  ) \nn \\
	& \leq \PP ( \| \bG_1 \|_2 \leq x+ R_1 + \gamma_1 ) +  C_2   \max_{1\leq j\leq d} \EE ( | Z_j|^3 ) \frac{\upsilon_3 }{ \sigma^3} \frac{d^{3/2}}{\sqrt{n}}  \nn \\
	& \leq \PP ( \| \bG_0 \|_2 \leq x \!+\! R_1 \!+\! \gamma_1 ) \!+ \!  C_2   \max_{1\leq j\leq d} \EE ( | Z_j|^3 ) \frac{\upsilon_3 }{ \sigma^3} \frac{d^{3/2}}{\sqrt{n}} \! + \! \frac{\upsilon_4^{1/2}}{2\sigma_\tau^2}  \bigg( \frac{d  +  t}{n} \bigg)^{6\eta}\nn \\
	& \leq  \PP \{ \| \bG_0 \|_2 \leq \max( x \!-\! R_2 , 0 )  \}  \nn \\
	& \quad~  +  C_2   \max_{1\leq j\leq d} \EE ( | Z_j|^3 ) \frac{\upsilon_3 }{ \sigma^3} \frac{d^{3/2}}{\sqrt{n}}  + \frac{\upsilon_4^{1/2}}{2\sigma_\tau^2}  \bigg( \frac{d + t}{n} \bigg)^{6\eta} + \frac{ C_3   }{\sigma_\tau}(R_1 + R_2 + \gamma_1)  . \label{m.clt.2}
\end{align}

\medskip
\noindent
{\sc Step 3}~(Gaussian comparison). Note that, conditional on $\mathcal{D}_n$, $\Sb_2$ follows a multivariate normal distribution with mean $\EE(\Sb_2 | \mathcal{D}_n) = \textbf{0}$ and covariance matrix
$$
	\bS_{n} = \cov(\Sb_2 |  \mathcal{D}_n ) = \frac{1}{n} \sn \xi_i^2 \bZ_i \bZ_i^\T \in \RR^{d\times d} .
$$
Applying Lemma~\ref{cov.concentration} with $x=2t$ yields that, with probability at least  $1- e^{- t}$,
\begin{align}
  	\|   \bSigma_0^{-1/2} \bS_{n}  \bSigma_0^{-1/2} \!-\! \bI_d \|_2   \leq   C_4 \frac{v^2}{\sigma^2} \bigg( \frac{d + t}{n} \bigg)^{1-2\eta} \leq  \frac{1}{2} \nn \\
  	\mbox{and } \tr \{ (  \bSigma_0^{-1/2} \bS_{n}  \bSigma_0^{-1/2}  -  \bI_d )^2 \}  \leq  C_4^2  \frac{v^4  d}{\sigma^2}   \bigg( \frac{d + t}{n} \bigg)^{2-4\eta}   \nn
\end{align}
as long as $n$ is sufficiently large, where $C_4=C_4(A_0)>0$. Hence, it follows from a conditional version of Lemma~A.7 in the supplement of \cite{SZ2015} that, with probability (over $\mathcal{D}_n$) at least $1- e^{-t}$,
\begin{align}
\sup_{y \geq 0} |  \PP (  \| \Sb_2 \|_2 \leq y | \mathcal{D}_n ) -  \PP ( \| \bG_0 \|_2 \leq y ) | \leq C_4 \frac{v^2  }{2\sigma^2}  \sqrt{d}  \bigg( \frac{d + t}{n} \bigg)^{1-2\eta}   .  \nn
\end{align}
In particular, taking $y=\max( x - R_2 , 0 ) $ gives
\begin{align}
 &\PP \{ \| \bG_0 \|_2 \leq \max( x - R_2 , 0 )  \} \nn \\
 &\leq   \PP \bigg\{ \frac{\|\bxi^\B \|_2 }{\sqrt{n}} \leq \max( x - R_2 , 0 ) \bigg| \mathcal{D}_n \bigg\} +  C_4 \frac{v^2}{2\sigma^2}   \sqrt{d}  \bigg( \frac{d + t}{n} \bigg)^{1-2\eta} . \label{b.clt}
\end{align}

Combining the inequalities \eqref{sqwilks1}, \eqref{sqwilks2}, \eqref{m.clt.2} and \eqref{b.clt}, we conclude that with probability (over $\mathcal{D}_n$) at least $1-6 e^{-t}$,
\begin{align}
	&  \PP\bigg[ \sqrt{2 \{ \cL_\tau(\btheta^*) - \cL_\tau(\hat{\btheta}_\tau ) \} } \leq x  \bigg]  \nn \\
	& \leq \PP\bigg[ \sqrt{2 \{ \cL^\B_\tau( \hat \btheta_\tau ) - \cL^\B_\tau(\hat{\btheta}_\tau^\B ) \} } \leq x  \bigg| \mathcal{D}_n \bigg] + 7 e^{-t}  + \frac{\upsilon_4^{1/2}}{2\sigma_\tau^2}  \bigg( \frac{d+t}{n} \bigg)^{6 \eta}   \nn \\
	& \quad~ +   C_2   \max_{1\leq j\leq d} \EE ( | Z_j|^3 ) \frac{\upsilon_3 }{ \sigma^3} \frac{d^{3/2}}{\sqrt{n}}  + \frac{ C_3   }{\sigma_\tau}(R_1 + R_2 + \gamma_1) +  C_4 \frac{v^2}{2\sigma^2} \sqrt{d}  \bigg( \frac{d + t}{n} \bigg)^{1-2\eta} . \nn
\end{align}
A similar argument leads to the reverse inequality and thus completes the proof by taking $z= x^2/2$.  \qed

\subsection{Proof of Theorem~\ref{Boot.validity}}

For $\alpha \in (0,1)$, let $q^\flat_{\alpha}$ and $q_\alpha$ be the upper $\alpha$-quantiles of
$$
	\sqrt{2 \{ \cL_\tau^\B(\hat{\btheta}_\tau) - \cL_\tau^\B(\hat{\btheta}^\B_\tau) \}} ~\mbox{ and }~ \sqrt{2 \{ \cL_\tau( {\btheta}^* ) - \cL_\tau(\hat{\btheta}_\tau) \} } ,
$$
respectively, under $\PP^*$ and $\PP$. By the definitions of $z^{\flat}_\alpha$ and $z_\alpha$ in \eqref{def:z*} and \eqref{def:zalpha}, it is easy to see that $z_\alpha^\flat = (q_\alpha^\flat )^2/2$ almost surely and $z_\alpha = q_\alpha^2/2$. According to Theorem~\ref{Boot.consistency}, there exists an event $\mathcal{E}_t$ satisfying $\PP(\mathcal{E}_t ) \geq 1- 6e^{-t}$ such that
\begin{align}
	 & \PP^* \{ \cL_\tau^\B(\hat{\btheta}_\tau) - \cL_\tau^\B(\hat{\btheta}^\B_\tau) >  q^2_{\alpha - \Delta_1 } /2 \}   \nn \\
	 &   \begin{cases}
	 =  0 < \alpha, & \mbox{ if } \alpha \leq \Delta_1,   \\
	 \leq \PP \{ \cL_\tau( {\btheta}^* ) - \cL_\tau(\hat{\btheta}_\tau) > q^2_{\alpha - \Delta_1 }/2 \} + \Delta_1 \leq   \alpha, & \mbox{ if } \alpha > \Delta_1,
	 \end{cases} \nn
\end{align}
and
\begin{align}
	&  \PP^* \{ \cL_\tau^\B(\hat{\btheta}_\tau) - \cL_\tau^\B(\hat{\btheta}^\B_\tau) > ( q_{\alpha+\Delta_1 } - \sigma /n )^2/2  \} \nn \\
	&   \geq   \PP \{ \cL_\tau( {\btheta}^* ) - \cL_\tau(\hat{\btheta}_\tau) > ( q_{\alpha+\Delta_1 } -\sigma /n )^2/2 \}  - \Delta_1 \geq  \alpha \nn
\end{align}
hold almost surely on $\mathcal{E}_t$, where $\Delta_1 = \Delta_1(n,d,t)$. Together, these inequalities imply
\begin{align}
 q_{\alpha +\Delta_1  }  - \sigma /n \leq 	q^\flat_\alpha \leq  q_{\alpha -  \Delta_1 } ~\mbox{ almost surely on } \mathcal{E}_t . \label{quantile.bound}
\end{align}

Next, define the L\'evy concentration function of the non-negative random variable $T := \sqrt{ 2 \{ \cL_\tau(\btheta^* ) - \cL_\tau(\hat{\btheta}_\tau) \} }$:
$$
	L(\epsilon ) = \sup_{x \geq 0}   \PP ( | T - x | \leq \epsilon ) , \ \ \epsilon \geq 0.
$$
It then follows from \eqref{quantile.bound} that
\begin{align}
	& \PP \{ \cL_\tau(\btheta^* ) - \cL_\tau(\hat{\btheta}_\tau) > z^\flat_\alpha \} \nn \\
	& \geq \PP \{ \cL_\tau(\btheta^* ) - \cL_\tau(\hat{\btheta}_\tau) > q_{\alpha - \Delta_1 }^2/2 \} - 6 e^{-t} \nn \\
	& \geq   \PP \{ \cL_\tau(\btheta^* ) - \cL_\tau(\hat{\btheta}_\tau) > ( q_{\alpha - \Delta_1 } - \sigma /n )^2/2 \} - L( \sigma /n )  - 6 e^{-t} \nn \\
	& \geq \alpha - \Delta_1 -  L( \sigma /n ) -  6 e^{-t} . \label{Boot.lbd}
\end{align}
Similarly, using \eqref{quadratic.bound} and the definition of $L(\cdot)$, we get
\begin{align}
& \PP \{ \cL_\tau(\btheta^* ) - \cL_\tau(\hat{\btheta}_\tau) > z^\flat_\alpha \} \nn \\
& \leq \PP \{ \cL_\tau(\btheta^* ) - \cL_\tau(\hat{\btheta}_\tau) > ( q_{\alpha + \Delta_1 } - \sigma /n )^2/2 \} +  6 e^{-t} \nn \\
& \leq  \PP \{ \cL_\tau(\btheta^* ) - \cL_\tau(\hat{\btheta}_\tau) > q_{\alpha + \Delta_1 }^2/2   \} + L( \sigma /n ) +  6 e^{-t} \nn \\
& \leq \alpha + \Delta_1 + L(\sigma /n ) +   6 e^{-t}  . \label{Boot.ubd}
\end{align}

To complete the proof, it remains to bound $L(\epsilon)$ for any given $ \epsilon >0$. Keeping the notations used in the proof of Theorem~\ref{Boot.consistency}, and following \eqref{sqwilks1}, \eqref{m.clt.1} and \eqref{anti.concentration}, we obtain that for any $x\geq 0$,
\begin{align}
	&   \PP ( | T - x |  \leq \epsilon ) \nn \\
	& \leq \PP ( | \| \Sb_1 \|_2 - x | \leq \epsilon + R_1 )  + 3 e^{-t }  \nn \\
	& \leq \PP ( | \| \overline{\bG}_1 \|_2  - x |  \leq \epsilon +  R_1  ) +  2 C_2   \max_{1\leq j\leq d} \EE ( | Z_j|^3 ) \frac{\upsilon_3 }{ \sigma^3} \frac{d^{3/2}}{\sqrt{n}}  + 3 e^{-t} \nn\\
	& \leq  C \frac{ \epsilon + R_1}{ \underline \lambda_{\bSigma_1}^{1/2} }  +  2 C_2   \max_{1\leq j\leq d} \EE ( | Z_j|^3 ) \frac{\upsilon_3 }{ \sigma^3} \frac{d^{3/2}}{\sqrt{n}}  +3 e^{-t}  , \label{levy.concentration.bound}
\end{align}
where $R_1 \asymp v (d+t) n^{-1/2}$ is as in \eqref{sqwilks1}, $\overline{\bG}_1 \sim \mathcal{N} (\EE(\Sb_1), \bSigma_1 )$ and $C>0$ is an absolute constant.

Finally, combining \eqref{Boot.lbd}, \eqref{Boot.ubd} and \eqref{levy.concentration.bound} to reach \eqref{Boot.bound}. \qed

\section{Proofs for Sections~3 and 4}

\subsection{Proof of Theorem~\ref{thm:lepski}}
This proof is based on an argument similar to that used in the proof of Theorem~5.1 in \cite{M2016}.
Let $j^* = \min\{ j\in \cJ : v_j \geq \sigma\}$ and note that $v_{j^*} \leq a \sigma$.
From the definition of $\hat j_{{\rm L}}$ in \eqref{lepski.choice1} with $c_0 \geq 2c_1 \underline{\lambda}_{\bSigma}^{-1/2}$, we see that
\begin{align}
	\{  \hat j_{{\rm L}} > j^* \} & \subseteq  \bigcup_{k\in \cJ: k> j^*} \Bigg\{  \| \hat \btheta^{(k)} - \hat \btheta^{(j^*)}  \|_2  > c_0 v_k \sqrt{\frac{d+t}{n}} \Bigg\}  \nn \\
& \subseteq \bigcup_{k\in \cJ: k \geq  j^*} \Bigg\{  \| \hat \btheta^{(k)} - \btheta^*  \|_2  > c_1 \underline{\lambda}_{\bSigma}^{-1/2} v_k \sqrt{\frac{d+t}{n}} \Bigg\} . \nn
\end{align}
Define the event
$$
	\cB = \bigcap_{k\in \cJ: k\geq j^*} \Bigg\{ \ \| \hat \btheta^{(k)} - \btheta^* \|_2 \leq  c_1 \underline{\lambda}_{\bSigma}^{-1/2}  v_k\sqrt{\frac{d+t}{n}} \Bigg\},	
$$
such that $\cB \subseteq \{ \hat j_{{\rm L}} \leq  j^* \}$. Recalling Theorem~\ref{br.thm}, we have for any $v\geq \sigma$, $\hat \btheta_\tau$ with $\tau = v\sqrt{n/(d+t)}$ satisfies the bound
$$
	 \| \hat \btheta_\tau - \btheta^* \|_2 \leq c_1 \underline{\lambda}_{\bSigma}^{-1/2} v \sqrt{\frac{d+t}{n}}
$$
with probability at least $1-3e^{-t}$ as long as $n\gtrsim d+t$. Together with the union bound, this implies
\begin{align}
 \PP (\cB^{{\rm c}}) & \leq \sum_{k\in \cJ: k\geq j^* } \PP\Bigg( \| \hat \btheta^{(k)} -  \btheta^*  \|_2 > c_1 \underline{\lambda}_{\bSigma}^{-1/2} v_k \sqrt{\frac{d+t}{n}} \Bigg) \nn \\
 & \leq 3| \cJ | e^{-t} \leq 3 \{ 1 + \log_a(v_{\max} / v_{\min}) \} e^{-t}.  \nn
\end{align}
On the event $\cB$, $\hat j_{{\rm L}} \leq  j^*$ and thus
\begin{align}
	 \|  \hat{\btheta}^{(\hat j_{\rm L})}  - \btheta^* \|_2 & \leq \| \hat{\btheta}^{(\hat j_{\rm L})}   -  \hat \btheta^{(j^*)} \|_2 +  \|  \hat \btheta^{(j^*)} - \btheta^* \|_2 \nn \\
	 & \leq  ( c_0 + c_1 \underline{\lambda}_{\bSigma}^{-1/2} ) v_{j^*} \sqrt{\frac{d+t}{n}} \leq \frac{3a}{2}  c_0 \sigma \sqrt{\frac{d+t}{n}} . \nn
\end{align}
Together, the last two displays lead to the stated result. \qed

\subsection{Proof of Theorem~\ref{thm:two-step}}
\label{sec:proof-two-step}

To begin with, define $\mathcal{D}_n^{(1)}$ and $\mathcal{D}_n^{(2)}$ to be the two independent samples $\{( Y_i^{(1)} , \bX_i^{(1)} )  \}_{i=1}^n$ and $\{ (Y_i^{(2) } , \bX_i^{(2)})\}_{i=1}^n$, respectively, such that $\bar{\mathcal{D}}_n = \mathcal{D}_n^{(1)} \cup \mathcal{D}_n^{(2)}$.
Under the assumption that $ \EE(|\varepsilon|^{4+\delta}) \leq \upsilon_{4+\delta} $ for some $\delta>0$, we have $  \EE |Y-\mu_Y|^{4+\delta} <\infty$. For each $j=1,\ldots, m$ with $m$ denoting the number of blocks, by Chebyshev's inequality, one can show that for any $\delta \in (0,1/2]$,
$$
 | \hat \upsilon_{Y,j}  -   \upsilon_{Y} | \lesssim  \bigg( \frac{ \EE |Y-\mu_Y|^{4+\delta}}{\delta n^{\delta/4}} \bigg)^{4/(4+\delta)},
$$
with probability at least $1-\delta$. Then, it follows from a variant of Lemma~2 in \cite{BCL2013} that, with $m = \lfloor 8\log n+1\rfloor$, 
$$
 \mathbb P \Bigg\{ 	| \hat \upsilon_{Y , {\rm mom} }   - \upsilon_Y | \gtrsim     ( \EE |Y-\mu_Y|^{4+\delta})^{4/(4+\delta)}  \bigg( \frac{\log n}{n} \bigg)^{\delta/(4+\delta)} \Bigg\} \lesssim n^{-1}
$$ 
as long as $n\gtrsim \log n$, where the probability is over the training set $ \mathcal{D}_n^{(1)}$. Therefore, with the same probability (over $ \mathcal{D}_n^{(1)}$),  $ | \hat \upsilon_{Y , {\rm mom} }   - \upsilon_Y |  \leq \upsilon_Y/2$ for all sufficiently large $n$. Then, it follows that, with high probability over $ \mathcal{D}_n^{(1)}$, 
$$
	\upsilon_4^{1/4}<  \upsilon_Y^{1/4}  \leq v_{\max}    \leq  ( 3\upsilon_Y)^{1/4} ~\mbox{ and }~
	 v_{\min} =  \frac{v_{\max} }{a^K} \leq  \frac{  ( 3\upsilon_Y)^{1/4 } }{a^K} < \upsilon_4^{1/4},
$$
where the last inequality holds provided $K \geq \lfloor \log_a ( 3\upsilon_Y / \upsilon_4 )^{1/4} \rfloor +1$. This, together with Theorem~\ref{thm:lepski} with slight modifications, implies that with probability (over $\mathcal{D}_n^{(1)}$) at least $1- O(Kn^{-1})$,
\begin{align}
    \| \hat \btheta^{(1) } - \btheta^* \|_2 \lesssim   \upsilon_4^{1/4} \sqrt{\frac{d+\log n}{n}}~\mbox{ and }~ \tau^*  \leq  \hat \tau  \leq   (3\upsilon_Y/\upsilon_4)^{1/4} \tau^*, \label{training.sample.1}
\end{align}
where $\tau^* :=  \upsilon_4^{1/4}  ( \frac{n}{d+\log n} )^{1/4}$.

For the second step, write $\varepsilon_i^{(2)} = Y_i^{(2)} - \langle \bX_i^{(2)} , \btheta^* \rangle$ for $i=1,\ldots, n$. Define random vectors 
$$
\bxi^* =  \sn \ell'_{\hat \tau}(\varepsilon_i^{(2)}) \bSigma^{-1/2}\bX_i^{(2)} ~\mbox{ and }~ \bxi^\B  = \sn \ell'_{\hat \tau} (\varepsilon_i^{(2)}) (1-W_i)  \bSigma^{-1/2}\bX_i^{(2)} .
$$
Conditioning on the event that \eqref{training.sample.1} holds, applying Theorems~\ref{wilks.thm} and \ref{boot.wilks.thm} we obtain that as long as $n\gtrsim d+ \log n$,
$$
   \bigg| \sqrt{ 2\{  \hat \cL(\btheta^* ) - \hat  \cL(\hat{\btheta} )\} } -  \frac{\| \bxi^* \|_2}{\sqrt{n}}    \bigg| \lesssim   ( \upsilon_Y/\upsilon_4)^{1/4} \frac{d+\log n}{\sqrt{n}}
$$
with probability (over $\mathcal{D}_n^{(2)}$) at least $1-O(n^{-1})$, and
$$
	\bigg|  \sqrt{ 2 \{ \hat \cL^\B ( \hat{\btheta}  )   - \hat  \cL^\B ( \hat \btheta^\B )   \} } - \frac{\| \bxi^\B  \|_2}{\sqrt{n}}     \bigg|  \lesssim   ( \upsilon_Y/\upsilon_4)^{1/4}   \frac{d+\log n}{\sqrt{n}}  
$$
with probability  (over $\mathcal{D}_n^{(2)}$ and $\{ W_i \}_{i=1}^n$)  at least $1-O(n^{-1})$. With the above preparations, the stated result follows from the same argument as in the proof of Theorem~\ref{Boot.validity}. \qed

\subsection{Proof of Theorem~\ref{FDP.thm}}

The proof consists of two main steps.

\medskip
\noindent
{\sc Step 1}~(Accuracy of bootstrap approximations). For each $1\leq k\leq m$, write $\mathcal{D}_{kn} = \{ (y_{ik}, \bx_i) \}_{i=1}^n$ and  $T_k^\B = \sn  \psi_{\tau_k}(\varepsilon_{ik}) U_i$. Then, applying Theorem~\ref{boot.concentration.thm} with $\bX=(1,\bx^\T)^\T$ and $\btheta^* = (\mu_k , \bbeta_k^\T)^\T$ gives that, with probability (over $\mathcal{D}_{kn})$ at least $1 -  6/(nm)^2 $,
\begin{align}
	\PP \{ | \sqrt{n} \, ( \hat{\mu}^\B_k - \hat \mu_k )  -  T^\B_k  / \sqrt{n} | \geq \delta_{2k}  | \mathcal{D}_{kn} \}  \leq 4 (nm)^{-2} ,  \label{cond.BR}
\end{align}
where $\delta_{2k} := C_{1k}\, v_k \{ s+ \log(nm) \} n^{-1/2}$ and $C_{1k} = C_{1k}(A_0, A_U)>0$. Observe that, conditional on $\mathcal{D}_{kn}$, $n^{-1/2} T^\B_k$ follows a normal distribution with mean zero and variance $\hat{\sigma}_{k,\tau_k}^2 = (1/n) \sn  \{ \ell_{\tau_k}'(\varepsilon_{ik}) \}^2$. With $\tau_k = v_k [ n/\{ s+2\log(nm)\} ]^{1/3}$, an argument similar to that used to derive Lemma~\ref{cov.concentration} may be employed to show that, with probability at least $1- (nm)^{-2}$,
\begin{align}
	 | \hat{\sigma}_{k,\tau_k}^2 - \sigma_{k}^2 | &  \leq | \hat{\sigma}_{k,\tau_k}^2 - \sigma_{k,\tau_k}^2 | + |\sigma_{k,\tau_k}^2 - \sigma_{k}^2 |   \nn \\
	& \leq  C_{2k}  \,v_k^2 \bigg\{ \frac{\log(nm)}{n} \bigg\}^{1/3} + \frac{\upsilon_{k,4}}{v_k^2} \bigg\{ \frac{s+\log(nm)}{n} \bigg\}^{2/3} , \label{cond.var.bound}
\end{align}
where $C_{2k} = C_{2k} (A_0)>0$. Combining \eqref{cond.BR}, \eqref{cond.var.bound}, Lemma~A.7 in \cite{SZ2015} and the union bound, we conclude that with probability (over $\mathcal{D}_{kn})$ at least $1 -7 /(n^2 m) $,
\begin{align}
   \sup_{x\in \RR} | \PP \{  \sqrt{n} \, ( \hat{\mu}^\B_k - \hat \mu_k )  \leq x |  \mathcal{D}_{kn} \}  & - \Phi(x/\sigma_k) | \nn \\
  & \leq  C_{3k} \, v_k^2 \bigg\{ \frac{\log(nm)}{n} \bigg\}^{1/3} + \frac{4}{(nm)^2}  \nn
\end{align}
for all $k=1,\ldots, m$. Combining this with \eqref{GAR.3}, \eqref{GAR.4} and taking $a_n =  2\log(nm)$ in Lemma~\ref{robust.md}, we conclude that on some event that occurs with probability at least $1- 7 /(n^2m)$,
\begin{align}
		\frac{\PP (   \sqrt{n} \, | \hat{\mu}^\B_k - \hat \mu_k  |  \geq z |  \mathcal{D}_{kn} ) }{  2 \{ 1- \Phi(z/\sigma_k)  \} }  = 1 + o(1)  \label{cond.md}
\end{align}
uniformly in $ 0\leq z/\sigma_k \leq o\{  \min (  n^{1/6}, \sqrt{n} /\log m ) \}$ and $1\leq k\leq m$.

\medskip
\noindent
{\sc Step 2}~(FDP control with bootstrap calibration). For $k=1,\ldots, m$ and $z\geq 0$, define $\hat{T}_k = \sqrt{n} \, \hat{\mu}_k$, $G(z) = 2\{ 1-\Phi(z)\}$,
\begin{align}
	G_k(z) = \PP_{H_{0k}}(  | \hat{T}_k | \geq z  ) ~\mbox{ and }~ 	G_k^\B(z) = \PP \{  \sqrt{n} \, | \hat{\mu}_k^\B - \hat{\mu}_k | \geq z |  \mathcal{D}_{kn} \} . \nn
\end{align}
In this notation, we have $p^\B_k = G^\B_k(|\hat{T}_k|)$ for $k=1,\ldots, m$. As a direct consequence of Lemma~1 in \cite{STS2004}, the BH procedure with $p$-values $\{ p^\B_k \}_{k=1}^m$ is equivalent to Storey's procedure, that is, reject $H_{0k}$ if and only if $p^\B_k \leq t^\B_{{\rm S}} $, where
\begin{align}
	t^\B_{{\rm S}} : = \sup\bigg\{ t\in [0,1] : t \leq \frac{\alpha \max\{ \sum_{k=1}^m I(p^\B_k \leq t) , 1 \}}{m} \bigg\}. \nn
\end{align}
By the definition of $t^\B_{{\rm S}} $, we have
\begin{align}
	t^\B_{{\rm S}}  = \frac{\alpha \max\{ \sum_{k=1}^m I(p^\B_k \leq t^\B_{{\rm S}}  ) , 1 \} }{m } . \label{BH.threshold}
\end{align}
For the bootstrap $p$-values $p^\B_k$ and data-driven threshold $t^\B_{{\rm S}} $, we claim that, as $(n, m ) \to \infty$,
\begin{align}
	\PP \{  t^\B_{{\rm S}}  \geq \alpha m_{1,\lambda_0} /m  \} \to 1   \label{BH.threshold.bound} \\
	\mbox{ and }  \sup_{  b_m/ m \leq t \leq 1} \bigg| \frac{\sum_{k\in \mathcal{H}_0} I( p^\B_k \leq t )}{m_0 \,t} - 1 \bigg|  \xrightarrow {\PP}  0 \label{unif.bound}
\end{align}
for any sequence $b_m >0$ satisfying $b_m \to \infty$ and $b_m = o(m)$, where $m_{1,\lambda_0} ={\rm card} \{ 1\leq k\leq m: | \mu_k | / \sigma_k \geq \lambda_0\sqrt{(2 \log m) / n} \, \}$.  Under condition~\eqref{cond.SNR}, it follows
$$
 \frac{\sum_{k\in \mathcal{H}_0} I( p^\B_k \leq  t^\B_{{\rm S}}  )}{m_0 \, t^\B_{{\rm S}}  }  \xrightarrow {\PP}  1  ,
$$
which, together with \eqref{BH.threshold}, proves the stated result \eqref{FDP.consistency}.

It remains to verify \eqref{BH.threshold.bound} and \eqref{unif.bound}. By \eqref{BH.threshold}, it is clear that $t^\B_{{\rm S}}  \in [\alpha/m , 1]$. Recall that $\log m = o( n^{1/3})$. Then, by \eqref{cond.md},
$$
	G^\B_k( \sigma_k \sqrt{ 2\log m}) = G(\sqrt{2\log m}) \{ 1 + o_{\PP}(1) \}
$$
uniformly in $1\leq k\leq m$ as $(n,m) \to \infty$. Note that
$$
	 G(\sqrt{2 \log m} ) = 2 \{ 1-\Phi(\sqrt{2\log m}) \} \sim \sqrt{\frac{2}{\pi}} \frac{1}{ m \sqrt{2\log m}} = o(m^{-1}).
$$
Combining the last two displays, we see that with probability tending to $1$, $t^\B_{{\rm S}}  \geq  G^\B_k( \sigma_k \sqrt{2 \log m})$ for all $1\leq k\leq m$. It follows
\begin{align}
	t_{{\rm S}} ^\B &  \geq \frac{\alpha}{m}  \sum_{k=1}^m I \{ G^\B_k(|\hat{T}_k|) \leq  G^\B_k( \sigma_k \sqrt{2\log m} ) \}  = \frac{\alpha}{m} \sum_{k=1}^m I \{ |\hat{T}_k| \geq \sigma_k \sqrt{2\log m}   \}  .\nn
\end{align}
Furthermore,
\begin{align}
	&  \sum_{k=1}^m I \{ |\hat{T}_k| \geq \sigma_k \sqrt{2\log m}   \} \nn \\
	& \geq  \sum_{k=1}^m I\bigg\{  \sqrt{n} \, \frac{ | \mu_k |}{\sigma_k } \geq  \sqrt{2\log m} + \sqrt{n} \max_{1\leq k\leq m} \frac{ |\hat{\mu}_k - \mu_k| }{ \sigma_k  }  \bigg\} . \nn
\end{align}
For any $\epsilon >0$, define the event
$$
	\mathcal{A}(\varepsilon) =  \bigg\{ \sqrt{n}\max_{1\leq k\leq m} \frac{ |\hat{\mu}_k - \mu_k | }{\sigma_k} \leq (1+ \epsilon ) \sqrt{ 2 \log m} \bigg\},
$$
on which it holds
$$
 \sum_{k=1}^m I \{ |\hat{T}_k| \geq \sigma_k \sqrt{2\log m}   \} \geq  \sum_{k=1}^m I\bigg\{    \frac{ | \mu_k |}{\sigma_k } \geq   (2+ \epsilon ) \sqrt{\frac{ 2\log m}{n}}  \, \bigg\}.
$$
Using Lemma~\ref{robust.md} and the union bound shows that, as $(n,m) \to \infty$,
\begin{align}
	  \PP \{ \mathcal{A}(\epsilon)^{{\rm c}} \}  &\leq  \sum_{k=1}^m \PP\Big\{ \sqrt{n}\, | \hat{\mu}_k - \mu_k |  > \sigma_k (1+ \epsilon ) \sqrt{2\log m} \Big\} \nn \\
	& \leq 2m \exp\{  - (1+\epsilon )^2 \log m  \} = 2 m^{-2\epsilon- \epsilon^2}  = o(1) .  \label{event.prob}
\end{align}
Putting the above calculations together leads to the claim \eqref{BH.threshold.bound}.

Finally we verify \eqref{unif.bound}. By Lemma~\ref{robust.md} and \eqref{cond.md}, it is easy to see that
\begin{align}
	\max_{1\leq k\leq m} \sup_{0\leq z \leq \sigma_k\sqrt{ 2\log (nm) }} \bigg| \frac{G^\B_k(z)}{G(z/\sigma_k)} -1 \bigg| \xrightarrow {\PP} 0  \nn \\
	\mbox{ and }  	\max_{1\leq k\leq m} \sup_{0\leq z \leq   \sigma_k\sqrt{ 2\log (nm)}} \bigg| \frac{G^\B_k(z)}{G_k(z)} -1 \bigg| \xrightarrow {\PP} 0 \nn
\end{align}
as $(n, m) \to \infty$. Also, consider the event
$$
	\mathcal{A}_0 =  \bigg\{  \max_{ k \in \mathcal{H}_0}   |\hat{T}_k| / \sigma_k \leq    \sqrt{ 2\log(nm) } \bigg\},
$$
Similarly to \eqref{event.prob}, it can be shown that $\PP( \mathcal{A}_0^{{\rm c}} ) \to 0$. Consequently, there exits a sequence $\{ \alpha_n \}_{n \geq 1}$ of positive numbers satisfying $\alpha_n \to 0$ such that
\begin{align}
	&   \sum_{k\in \mathcal{H}_0} I\{ G( |\hat{T}_k| /\sigma_k ) \leq (1-\alpha_n ) t \} \nn \\
	 &   \leq  \sum_{k\in \mathcal{H}_0} I(p_k^\B \leq t ) \leq  \sum_{ k \in \mathcal{H}_0 } I \{ G( |\hat{T}_k| /\sigma_k) \leq  (1+\alpha_n ) t \}. \label{FD.bound}
\end{align}
Again, using Lemma~\ref{robust.md} gives
\begin{align}
	\max_{k \in \mathcal{H}_0} \sup_{0\leq z \leq \sigma_k \sqrt{ 2\log m}}  \bigg| \frac{\PP(|\hat{T}_k| \geq z)}{G(z/\sigma_k )}  -1 \bigg|  \to 1 . \label{null.unif}
\end{align}
Note that, with $0\leq z \leq \sigma_k \sqrt{2\log m}$, it holds
$$
	\sqrt{\frac{2}{\pi}} \frac{\sqrt{2\log m}}{1+2 \log m} \frac{1}{m}  \leq G(z/\sigma_k) \leq 1.
$$
In \eqref{null.unif}, we change the variable by $t= G(z/\sigma_k)$ to obtain
$$
	\max_{k \in \mathcal{H}_0} \sup_{ m^{-1} \leq t \leq 1 }  \bigg| \frac{\PP\{ G( |\hat{T}_k| / \sigma_k ) \leq t\}}{t}  -1 \bigg|  \to 0.
$$
By an argument similar to that in the proof of Proposition~B.3 in \cite{ZBFL2017}, it follows that for any sequence $b_m>0$ satisfying $b_m \to \infty$ and $b_m = o(m)$,
\begin{align}
	\sup_{b_m / m \leq t\leq 1}  \bigg| \frac{\sum_{k \in \mathcal{H}_0} I\{ G( | \hat{T}_k |/\sigma_k) \leq t \} }{m_0 \, t}  - 1 \bigg| \xrightarrow {\PP} 0.\nn
\end{align}
Together with \eqref{FD.bound}, this proves \eqref{unif.bound} as desired. \qed

\section{Implementation}
\label{app:C}

Since the bootstrap Huber estimator needs to be computed many times,  an efficient optimization solver is critical for applications. Ideally,  second order methods such as Newton's method should be adopted due to fast convergence.
Denote the gradient of the weighted Huber loss in \eqref{bootHuber.est} by
\begin{align}\label{eq:g}
 \bg(\btheta)=\sum_{i=1}^n W_i  &  \{   I ( |Y_i-\bX_i^\T\btheta|\leq \tau  ) (\bX_i^\T  \btheta-Y_i)\bX_i  \\
 & ~ +   I ( |Y_i-\bX_i^\T\btheta| > \tau )  \tau \cdot \sgn(\bX_i^\T\btheta-Y_i) \bX_i   \}, \ \ \btheta \in \RR^d. \nn
\end{align}
Although $\bg(\btheta)$ is not differentiable everywhere with respect to $\btheta$, we can still compute a generalized Jacobian of $\bg(\btheta)$:
\begin{equation}\label{eq:H}
 \bH(\btheta)=\sum_{i=1}^n W_i   I ( |Y_i-\bX_i^\T\btheta|\leq \tau  )  \bX_i \bX_i^\T,
\end{equation}
which serves as an ``approximate Hessian matrix''. Given \eqref{eq:H}, the generalized Newton method can be directly implemented via the following iterative procedure (for $t=1,2,\ldots$):
\begin{equation}\label{eq:gen_newton}
\btheta^{t+1}=\btheta^{t}-\eta_t  \{ \bH(\btheta^t) \}^{-1} \bg(\btheta^t),
\end{equation}
where $\eta_t$ is the step-size. We note that the constraint in \eqref{bootHuber.est} is omitted here, since it is introduced mainly for theoretical analysis and will not affect the empirical performance.

Although \eqref{eq:gen_newton} is easy to implement, there remains a practical issue that the Hessian matrix $\bH(\btheta^t)$ is not always invertible. To address this issue, we adopt the damped semismooth Newton method, which is a combination of Newton's method and gradient descent. The idea is straightforward: when $\bH(\btheta^t)$ is invertible, $\btheta^{t+1}$ is computed via the generalized Newton step in \eqref{eq:gen_newton}; otherwise, the gradient descent step is performed, that is,
\begin{equation}
\btheta^{t+1}=\btheta^{t}-\eta_t \,\bg(\btheta^t).
\end{equation}
The step-size $\eta_t$ is determined via the backtracking-Armijo line search rule.  

 Now we briefly discuss the the convergence of the damped semismooth Newton method. Note that the random weights $W_i$ may sometimes take negative values, our objective function could be non-convex, and thus we only discuss the convergence to a stationary point, i.e.  some $\hat\btheta$ such that $g(\hat\btheta)=0$. The following proposition from \cite{QS1999} and \cite{DFK1996} provides the local convergence rate for solving a system $g(\btheta)=0$.
\begin{proposition}\label{prop:convergence}
{\rm	Suppose that $g(\hat\btheta)=0$, where $g$ is locally Lipschitz, and that all $V \in \partial g(\hat\btheta)$ are non-singular. If $g$ is strongly semismooth at $\hat\btheta$, then the method is quadratically convergent in a neighborhood of $\hat\btheta$. }
\end{proposition}

Now, let us verify the conditions in Proposition \ref{prop:convergence} for the weighted Huber regression. Given the Huber loss $\ell_\tau(x)$ and its gradient $\ell_\tau'(x) = xI(|x| \leq \tau) + \tau\cdot {\rm sign}(x)I(|x|>\tau)$, the Clarke's generalized Jacobian of $\ell_\tau'(x)$ \citep{HL2001} can be calculated as 
\begin{equation}\label{eq:generalized_jacob}
\partial \ell_\tau'(x)  =\begin{cases}
0       & x > \tau \;\; \text{or} \;\; x<-\tau\\
1       & |x| <\tau \\
[0,1]      & x = \pm \tau  \\
\end{cases},
\end{equation}

The boundedness of $\ell_\tau'(x)$ implies that $g$ in \eqref{eq:g} is locally Lipschitz. Moreover, we can easily verify that $\ell_\tau'(x)$ is a strongly semismooth function. Since the semi-smoothness is preserved under linear transformation, the function $g$ in \eqref{eq:g}  is also  strongly semismooth. Then the remaining condition is on the non-singularity of  $V \in \partial g(\hat\btheta)$, where  $\partial$ denotes the Clarke's generalized Jacobian of $g$. According to \eqref{eq:generalized_jacob}, we have 
\[
 \partial g(\hat\btheta) = \left\{ \sum_{i=1}^n W_i  \bX_i \bX_i^\T \{ I ( |Y_i-\bX_i^\T\hat\btheta|< \tau  ) + v_i I ( |Y_i-\bX_i^\T\hat\btheta|= \pm \tau)  \} : v_i \in [0,1]\right\}.
\]
We note that $ \bH$ in \eqref{eq:H} is also a member of $\partial g(\hat\btheta)$. 
The non-singularity condition depends on the realization of random weights $W_i$ and $\hat\btheta$. However, since the dimension $d$ is small as compared to $n$, $W_i$ and $(Y_i, \bX_i)$ are IID random variables, as long as $\hat\btheta$ is not too extreme (e.g. there are at least $d$ terms such that $|Y_i-\bX_i^\T\hat\btheta|< \tau $), the non-singularity condition will be easily satisfied.
 
\section{Selecting robustification parameter: A data-driven approach}
\label{app:D}

\subsection{Preliminaries}
\label{sec:pre}

Let $X$ be a real-valued random variable with finite variance. For $z \geq 0$, define
\begin{align}
G(z) = \pr(|X|>z), ~ P(z) = \EE \{ X^2 I(|X|\leq z) \} , ~ Q(z) = \EE   \{\psi^2_z(X)\}, \label{def.GPQ}
\end{align}
where $\psi_z(x)  = ( |x|\wedge z ) \sgn(x)$, $x\in \RR$. Moreover, for $z>0$, we define
\begin{align}
 p(z) =z^{-2} P(z) ~\mbox{ and }~ q(z) = z^{-2} Q(z). \label{def.pq}
\end{align}
It is easy to see that $Q(z) = P(z) + z^2 G(z)$ and $q(z) = p(z) + G(z)$. The following result provides some useful connections among these functions.  See (2.3) and (2.4) in \cite{HKW1990}. We reproduce them here for the sake of readability.

\begin{lemma} \label{prop1}
{\rm
Let functions $G, Q, p$ and $q$ be given in \eqref{def.GPQ} and \eqref{def.pq}.
\begin{enumerate}
\item[(i)] The function $Q:[0,\infty) \to \RR$ is non-decreasing with $\lim_{z \to \infty}Q(z) = \EE (X^2)$. For any $z >0$, we have
\begin{align}
  Q(z) = 2 \int_0^z  y G(y) \, dy, \quad  q'(z) = - 2 \frac{p(z)}{z}  , \label{prop1.1}
\end{align}
and
\begin{align}
  q(z) = \pr(X\neq 0) - 2 \int_0^z  \frac{p(y)}{y} dy.  \label{prop1.2}
\end{align}

\item[(ii)] The function $q: (0,\infty) \to \RR$ is non-increasing and positive everywhere with $q(0+):=\lim_{s \downarrow 0}q(s) =\pr(X\neq 0)$. Moreover,
\begin{align}
	q(s)=\pr(X\neq 0)  \label{def.Delta}
\end{align}
for all $0\leq s \leq \Delta := \inf\{ y>0: G(y)< \pr(X\neq 0)\}$, and $q(s)$ decreases strictly and continuously on $(\Delta, \infty)$ with $\lim_{z \to \infty} q(z) = 0$.
\end{enumerate}
}
\end{lemma}

\begin{proof}[Proof of Lemma~\ref{prop1}]
Note that
\begin{align*}
	 (|X| \wedge z )^2 & = 2\int_0^z I(|X|> z )y \,dy + 2\int_0^{|X|} I(|X|\leq z )y \,dy \nn \\
& =  2\int_0^z  I(|X|>z )y \,dy + 2\int_0^z  I(|X|>y) I(|X|\leq z)y \,dy  \\
&= 2\int_0^z I(|X|>y) y\, dy.
\end{align*}
Taking expectations on both sides implies $Q(z)= \EE (|X| \wedge z )^2   = 2\int_0^z \pr(|X|>y) y\, dy=2\int_0^z yG(y)dy$, as stated. It follows that $Q'(z) = 2 zG(z)$ and thus $Q$ is non-decreasing. Moreover, by the monotone convergence theorem we see that $\lim_{z\to \infty} Q(z) = \EE(X^2)$.

Next, taking derivatives  with respect to $z$ on both sides of \eqref{def.pq} gives $2z q(z) + z^2 q'(z) =  2 zG(z) = 2z \{q(z) - p(z)\}$, which proves the the second equation in \eqref{prop1.1}. To prove \eqref{prop1.2}, note that, for any $0<s<z$, $q(z) = q(s) - 2 \int_s^z  y^{-1} p(y) \, dy $. On event $\{ X\neq 0\}$, it holds almost surely that
$$
	0< \frac{(|X| \wedge s)^2}{s^2} \leq 1, ~\mbox{ and }~ \frac{(|X| \wedge s)^2}{s^2} \to 1 ~\mbox{ as } s \to 0.
$$
By the dominated convergence theorem,
$$
	q(s) = \EE \{ s^{-2} (|X| \wedge s)^2\} = \EE \{ s^{-2} (|X| \wedge s)^2 I(|X| >0)\} \to \pr(|X|>0)
$$
as $s\to 0$.
In the equation $q(z) = q(s) - 2 \int_s^z  y^{-1} p(y) \,dy $  for $0<s<z$, letting $s$ tend to zero proves \eqref{prop1.2}.

Move to part (ii), by the definition of $\Delta$, we have $\pr(0< |X| \leq y)=0$ and thus $p(y)=0$ for all $0< y<\Delta$. This, together with \eqref{prop1.2}, implies $q(s) = \pr(X\neq 0)>0$ for all $0\leq s\leq \Delta$. It is easy to see that $p(y)>0$ for any $y> \Delta$, and therefore $q(\cdot)$ is strictly decreasing on $(\Delta, \infty)$. Finally, note that
$$
0< \frac{(|X| \wedge s)^2}{s^2} \leq 1  ~\mbox{ and }~ \frac{(|X| \wedge s)^2}{s^2} \to  0 ~\mbox{ as } s \to \infty.
$$
By the dominated convergence theorem, $\lim_{z \to \infty} q(z) = 0$ as desired.
\end{proof}

\subsection{Catoni's lower bound of sample mean}
\label{sec:catoni}

Let $X_1,\ldots, X_n$ be IID random variables from $X$ with mean zero and variance $\sigma^2>0$.
Let $\cA_{\sigma^2}$ be the set of probability measures on the real line with variance bounded by $\sigma^2$.
\cite{C2012} proved a lower bound for the deviations of the empirical mean $\bar X_n$ when the underlying distribution is the least favorable in $\cA_{\sigma^2}$: for any $t \geq 2 e$, there exists some distribution with mean zero and variance $\sigma^2$ such that the IID sample of size drawn from it satisfies
\begin{align}
 	| \bar X_n | \geq  \bigg( 1 - \frac{1}{n} \bigg)^{(n-1)/2}  \sigma \sqrt{\frac{t}{2n} } \nn
\end{align}
with probability at least $2t^{-1}$. This shows that the worst case deviations of $\bar{X}_n$ are suboptimal with heavy-tailed data.

 \subsection{Proof of Proposition~\ref{prop2}}

For any $\tau>0$, note that $\psi_\tau(X_i)$'s are independent random variables satisfying $| \psi_\tau(X_i) | \leq \tau$ and $\EE \psi_\tau^2(X_i) = \sigma_\tau^2$. By Bernstein's inequality,
\begin{align}
	 | \hat m_\tau - \mu_\tau | \leq \sigma_\tau \sqrt{\frac{2t}{n}} +  \frac{\tau t}{3n} \nn
\end{align}
with probability at least $1- 2 e^{-t}$. Taking $\tau = \tau_t$ in the last display leads to the first inequality in \eqref{single.concentration}, which, together with \eqref{bias.bound}, proves the second one.

To prove \eqref{uniform.concentration}, we first make a finite approximation of the interval $[1/2,3/2]$ using a sequence $\{ c_k \}_{k=1}^n$ of equidistant points $c_k = 1/2 + k/n$. Then for any $\tau_t  / 2 \leq \tau \leq 3\tau_t /2$ with $\tau_t = \sigma_{\tau_t} \sqrt{n/t}$, there exists some $1\leq k\leq n$ such that $|\tau - \tau_{t,k}| \leq \sigma_{\tau_t} (n t )^{-1/2}$, where $\tau_{t,k} := c_k \sigma_{\tau_t} \sqrt{n/t}$. It follows that
\begin{align}
	\sup_{  \tau_t /2 \leq \tau \leq  3\tau_t /2} | \hat m_{\tau}   | \leq \max_{1\leq k\leq n} | \hat m_{\tau_{t, k} } | + \frac{\sigma_{\tau_t}}{\sqrt{nt}}.  \label{discrete.approxi}
\end{align}
For every $1\leq k\leq n$, we have
$$
	| \hat m_{\tau_{t, k} } - \mu_{\tau_{t, k} } | \leq  \sigma_{\tau_{t,k}} \sqrt{\frac{2t}{n}} + \frac{\tau_{t,k} }{3} \frac{t}{n}
$$
with probability at least $1- 2 e^{-t}$. By \eqref{bias.bound}, $| \mu_{\tau_{t,k}} | \leq  ( \sigma^2 - \sigma_{\tau_{t,k}}^2 )  / \tau_{t,k}$.
Apply the union bound over $1\leq k\leq n$ to see that
\begin{align}
	 \max_{1\leq k\leq n} | \hat m_{\tau_{t,k} }  | \leq  \max_{1\leq k\leq n} \bigg(   \sigma_{\tau_{t,k}} \sqrt{2} + \frac{c_k \sigma_{\tau_t} }{3} + \frac{\sigma^2}{c_k \sigma_{\tau_t}}   - \frac{\sigma_{\tau_{t,k}}^2}{c_k \sigma_{\tau_t}} \bigg)  \sqrt{\frac{t}{n}}   \label{discrete.bound}
\end{align}
with probability at least $1- 2n e^{-t}$. Together, \eqref{discrete.approxi} and \eqref{discrete.bound} prove \eqref{uniform.concentration}. \qed

 \subsection{Proof of Proposition~\ref{prop:exist}}

Using the notation in Section~\ref{sec:pre}, equation \eqref{population.tau} can be written as $q(\tau) = t/n$. By Lemma~\ref{prop1}, the function $q$ satisfies $\max_{z \geq 0} q(z) = \lim_{z \to 0} q(z) = \PP(|X|>0)$, $\lim_{z\to \infty} q(z) = 0$ and is strictly decreasing on $(\Delta, \infty)$. Provided $t /n < \PP(|X| >0)$, equation \eqref{population.tau} has a unique solution that lies in  $(\Delta, \infty)$.

By definition, this unique solution $\tau_t$ satisfies
\begin{align} \label{eqn.tauz}
	\tau_t^2 = \EE(X^2\wedge \tau_t^2 ) \frac{n}{t}  \leq \sigma^2 \frac{n}{t}.
\end{align}
On the other hand, note that $\EE(X^2 \wedge \tau^2) \geq \tau^2 \PP(|X| > \tau)$ for any $\tau>0$. It follows that $\PP(|X|>\tau_t) \leq  t/n$, which implies $\tau_t \geq q_{t/n}$.  Substituting this into \eqref{eqn.tauz} gives $\tau_t^2 \geq \EE(X^2 \wedge q_{t/n}^2) (n/t)$.

To prove Part~(ii), recall that $q(\tau_t) = t/n$. Since $t/n \to 0$ and $q(z)$ strictly decreases to zero as $z\to \infty$, we have $\tau_t \to \infty$ and therefore $\EE (X^2 \wedge \tau_t^2) \to \sigma^2$ as $n\to \infty$. \qed

\subsection{Proof of Theorem~\ref{thm1}}

By Proposition~\ref{prop:exist2}, $\hat \tau_t$ is uniquely determined and positive on the event $\{ t < \sn I(|X_i|>0) \}$.
Under the condition $\PP(X=0)=0$ and when $t<n$, this event occurs with probability one. We divide the rest of the proof into four steps.

\noindent
{\sc Step 1}. Define functions
$$
	p_n(z) = \frac{1}{n} \sn \frac{X_i^2 I(|X_i|\leq z)}{z^2} ~\mbox{ and }~ q_n (z) =  \frac{1}{n} \sn \frac{\psi_z^2(X_i)}{z^2} , \ \ z >0.
$$
Applying Lemma \ref{prop1} to $p_n$ and $q_n$ implies $q_n'(z)=-2 z^{-1} p_n(z)$. Therefore,
\begin{align}
   q_n(\tau_t ) - q_n(\hat{\tau}_t)  = 2 \int^{\hat{\tau}_t}_{\tau_t} \frac{p_n(z)}{z} dz  = 2\int_0^{(\hat{\tau}_t- \tau_t)/\tau_t} \frac{p_n(  \tau_t +\tau_t u  )}{1+u} d u \nn
\end{align}
by change of variables $u = (z -\tau_t)/\tau_t$. By definition, $q_n(\hat{\tau}_t)= t/n = q(\tau_t)$. It then follows that
$$
  q_n(\tau_t) - q(\tau_t) =  2\int_0^{(\hat{\tau}_t- \tau_t)/\tau_t} \frac{p_n( \tau_t  + \tau_t u )}{1+u} d u. $$ 
For any $r\in (0,1)$, it holds on the event $\{   ( \hat{\tau}_t- \tau_t)/\tau_t  \geq  r  \}$ that
\begin{align}
  q_n(\tau_t) - q(\tau_t) &\geq 2 \int_0^r \frac{p_n(\tau_t  + \tau_t u )}{1+u} d u \nn \\
&  =   2 \int_0^r  \frac{p_n(\tau_t  + \tau_t u) - p(\tau_t  + \tau_t u)}{1+u} du + 2\int_0^r \frac{p(\tau_t  + \tau_t u)}{1+u} du \nn \\
& =  { 2 \int_0^r  \frac{p_n(\tau_t  + \tau_t u) - p(\tau_t  + \tau_t u)}{1+u} du }  +   \{ q(\tau_t) - q(\tau_t +\tau_t r ) \}  \nn\\
& =: R_1 +  D_1 . \nn
\end{align}
Similarly, on the event $\{ ( \hat{\tau}_t- \tau_t)/\tau_t  \leq  -r \}$, it holds
\begin{align*}
 & q_n(\tau_t) - q(\tau_t) \nn  \\
 \leq  &  - \{  q(\tau_t -\tau_t r ) - q(\tau_t)  \} - 2 \int_{-r}^0 \frac{p_n(\tau_t +\tau_t u) - p(\tau_t +\tau_t u)}{1+u} du\\
  =: & -D_2 + R_2 . \nn
\end{align*}
Putting the above calculations together, we arrive at
\begin{align}
 & \pr( |\hat{\tau}_t / \tau_t -1 | \geq r ) \nn \\
 &  \leq  \pr \{  q_n(\tau_t ) - q(\tau_t) \geq D_1 + R_1  \} + \pr \{  q_n(\tau_t ) - q(\tau_t) \leq  - D_2+  R_2   \}. \label{tail.prob.dec}
\end{align}
Set $\zeta_i = (X_i^2 \wedge \tau_t^2)/\tau_t^2$ such that $q_n(\tau_t ) - q(\tau_t) =(1/n)\sn \{\zeta_i - \EE (\zeta_i )\}$. Note that $0\leq \zeta_i\leq 1$ and $\EE (\zeta_i^2) \leq \EE( X_i^2 \wedge \tau_t^2) / \tau_t^2 = t/n$. By Bernstein's inequality, for any $u>0$ it holds
\begin{align}
	\pr\{ q_n(\tau_t ) - q(\tau_t)  \geq u  /n\} \leq \exp  \{ -u^2/(2t+2u/3) \}. \label{qn.ubd}
\end{align}
On the other hand, applying Theorem~2.19 in \cite{DLS2009} with $X_i=\zeta_i/n$ therein gives that, for any $0<u< t$,
\begin{align}
 \pr\{ q_n(\tau_t ) - q(\tau_t)  \leq -u  /n \} \leq \exp\{ -u^2/(2t) \} . \label{qn.lbd}
\end{align}

\noindent
{\sc Step 2} (Controlling $R_1$ and $R_2$). Note that $R_1$ and $R_2$ can be written, respectively, as $R_1 =   (2/n) \sn \{\xi_i -\EE (\xi_i)\}$ and $R_2 = - (2/n) \sn \{\eta_i - \EE (\eta_i)\}$, where
\begin{align}
	 \xi_i =  \int_0^r \frac{X_i^2 I\{ |X_i| \leq \tau_t(1+u)\}}{ \tau_t^2(1+u)^3} du ~\mbox{ and }~
	 \eta_i = \int_{-r}^0 \frac{X_i^2 I\{ |X_i| \leq \tau_t(1+u)\}}{ \tau_t^2(1+u)^3} du  \nn
\end{align}
are bounded, non-negative random variables satisfying
$$
 \xi_i \leq \int_0^r \frac{du}{1+u}   \leq r, \quad  \eta_i   \leq \int_{-r}^0 \frac{du}{1+u} \leq \frac{r}{1-r}. $$
In addition,
$$	  \EE (\xi_i^2) \leq  \frac{ \EE[ X_i^2I\{|X_i| \leq \tau_t(1+r)\}] }{\tau_t^2} \bigg\{  \int_0^r  \frac{du}{  (1+u)^2} \bigg\}^2 \leq q(\tau_t +\tau_t r ) r^2  \leq q(\tau_t) r^2 , $$ and
$$   \EE (\eta_i^2) \leq  \frac{\EE \{X_i^2 I(|X_i|\leq \tau_t)\}}{\tau_t^2}  \bigg\{  \int_{-r}^0  \frac{du}{  (1+u)^2} \bigg\}^2 \leq  \frac{q(\tau_t) r^2}{(1-r)^2} . \nn
$$
Recall that $q(\tau_t) = t/n$.
Again, by Theorem~2.19 in \cite{DLS2009} we have, for any $v>0$,
\begin{align}
	\pr(   R_1 \leq  - 2r v/n ) \leq  \exp \{ - v^2 / (2 t  ) \}   \label{R1.bound}
\end{align}
and
\begin{align}
 \pr \{  R_2  \geq  2rv/(1-r)n \} \leq   \exp\{- v^2/(2t ) \}. \label{R2.bound}
\end{align}

\noindent
{\sc Step 3} (Bounding $D_1$ and $D_2$). Starting with $D_1$, by Lemma~\ref{prop1} we have
\begin{align}
 	  D_1 & = q(\tau_t) - q(\tau_t +\tau_t r)  = 2\int_{\tau_t}^{\tau_t(1+r)} \frac{P(u)}{u^3} du \nn \\
 	  &  \geq 2 P(\tau_t)\int_{\tau_t}^{\tau_t(1+r)} \frac{1}{u^3} du =   \frac{ r^2+2r}{ (1+r)^2}\frac{P(\tau_t)}{\tau_t^2} . \label{D1.lbd}
\end{align}
Similarly,
\begin{align}
	& D_2 = q(\tau_t -\tau_t r) - q( \tau_t ) = 2\int^{\tau_t}_{\tau_t(1-r)} \frac{P(u)}{u^3} du \geq  \frac{2r -r^2}{(1-r)^2} \frac{P(\tau_t -\tau_t r)}{\tau_t^2} . \label{D2.lbd}
\end{align}

\noindent
{\sc Step 4}. Together, \eqref{tail.prob.dec} and \eqref{R1.bound}--\eqref{D2.lbd} imply that, for any $0<r<1$ and $v>0$,
\begin{align}
& \pr( |\hat{\tau}_t / \tau_t -1 | \geq r ) \nn \\
& \leq 2 \exp\{-v^2/(2t  )\} +  \pr \bigg\{ q_n(\tau_t) - q(\tau_t) \geq    \frac{ r^2+2r}{ (1+r)^2} \frac{P(\tau_t)}{\tau_t^2}  - \frac{2rv}{n} \bigg\} \label{tail.prob.1} \\
& \quad +  \pr \bigg\{ q_n(\tau_t) - q(\tau_t) \leq  -   \frac{2r -r^2}{(1-r)^2} \frac{P(\tau_t - \tau_t r)}{\tau_t^2} +  \frac{2rv}{(1-r)n}\bigg\} .\nn
\end{align}
Note that
$$
 \frac{ r^2+2r}{ (1+r)^2} \frac{P(\tau_t)}{\tau_t^2}  - \frac{2rv}{n} =   \bigg\{ \frac{P(\tau_t) }{ Q(\tau_t)} \frac{ 2+r }{(1+r)^2} t - 2v\bigg\} \frac{r}{n} $$
and
\begin{align}
	 \frac{2r -r^2}{(1-r)^2} \frac{P(\tau_t -\tau_t r)}{\tau_t^2} -  \frac{2rv}{(1-r)n}   =   \bigg\{ \frac{P(\tau_t -\tau_t r)}{Q(\tau_t)} \frac{2-r}{1-r}  t - 2v\bigg\}   \frac{r}{(1-r)n} .\nn
\end{align}
Taking $v = (a_1 \wedge a_2) t/2 $ for $a_1$ and $a_2$ as in \eqref{def.c+-}, the right-hand side of \eqref{tail.prob.1} can further be bounded by
\begin{align}
	\PP  \bigg\{ q_n(\tau_t) - q(\tau_t) \geq \frac{ a_1 rt}{n} \bigg\}   + \PP  \bigg\{ q_n(\tau_t) - q(\tau_t) \leq - \frac{a_2 rt}{n} \bigg\}  + 2 \exp\{-v^2/(2t )\} . \nn
\end{align}
Combining this with  \eqref{qn.ubd}, \eqref{qn.lbd} and \eqref{tail.prob.1} proves \eqref{tau.tail.prob}. \qed

\section{Additional Simulation Studies}

\subsection{Standard deviations of estimated quantiles}
\label{sec:std}

In this section, we report the standard deviations of the estimated quantiles for the results in Table \ref{table:noise} and Table \ref{table:noise2}; see Table \ref{table:stdn100} and Table \ref{table:stdn200} below, which correspond to Table \ref{table:noise}  and Table \ref{table:noise2}, respectively. For each setting in Tables \ref{table:stdn100} and \ref{table:stdn200}, the standard deviation of the estimated quantiles for the boot-Huber is slightly smaller than that for the boot-OLS method.
In a sense both the two bootstrap-based methods are rather stable, although the latter one suffers from distorted empirical coverage due to heavy-tailedness.
For settings in other tables in the main text, the observations are similar and thus we omit the details.

\subsection{Correlated design}
\label{sec:corr}

In this section, we consider some more challenging cases in which the designs are highly correlated and/or non-Gaussian. Specifically, we consider the following two scenarios:
\begin{enumerate}
\item The covariate vector $\bX = (X_1,\ldots, X_d)^\T \in \RR^d$ follows a multivariate uniform distribution $\mathrm{Unif}([0,1]^d)$ with $\mathrm{Corr}(X_j,X_k)=0.5^{|j-k|}$ for $1\leq j\neq k\leq d$. See \cite{falk1999simple} for the construction of a multivariate uniform distribution.  Each component of $\btheta^*$ follows a Bernoulli distribution with probability 0.5, i.e. $\text{Ber}(0.5)$. The results for this case are presented in Tables \ref{table:unif_n100} (with $n=100$) and \ref{table:unif_n200} (with $n=200$).
\item The covariate vector $\bX$ follows $\mathcal{N}(\textbf{0}, \bSigma)$, where the covariance matrix $\bSigma=(\sigma_{jk})_{1\leq j, k \leq d}$ has a Toeplitz structure with $\sigma_{jk}=0.9^{|j-k|}$. The components of $\btheta^*$ are equally spaced in $[0,1]$.      The results for this case are presented in Tables \ref{table:toep_n100} (with $n=100$) and \ref{table:toep_n200} (with $n=200$).
\end{enumerate}

From Tables~\ref{table:unif_n100}--\ref{table:toep_n200} we find that the average coverage probabilities of the boot-Huber method are in general close to  nominal levels, while the boot-OLS leads to severe under-coverage for many heavy-tailed noise settings.

\subsection{Simulations on multiple testing}
\label{sec:exp_multi}

In this section, we evaluate the empirical performance of the proposed robust multiple testing procedure described in Algorithm \ref{algo:huber_multi_infer}. Recall the multi-response regression model \eqref{panel.data}:
\[
y_{ik}=\mu_k+\bx_i^\T \bbeta_k+\varepsilon_{ik},\quad  i=1,2,\dots, n, \quad k=1,\ldots, m,
\]
where $\bbeta_k \in \mathbb{R}^s$. We choose $\mu_k = \gamma \sigma\sqrt{ (2 \log m)/n}$ for $1 \leq  k \leq m_1$ with $m_1 = 0.05m$ and $\mu_k = 0$ for $m_1 +1 \leq  k \leq m$, where $\sigma^2 =\var(\varepsilon_i)=1$. The parameter $\gamma$   takes the value either  1.5 (i.e. the weaker signal strength case)  or 3 (i.e. the stronger signal strength case). We generate $\{\bx_i\}_{i=1}^n$ from $\mathcal{N}( \textbf{0}, \bI_s)$ and $\bbeta_k$ from the uniform distribution on $[-1,1]^s$ for $k=1,\ldots, m$. The settings of error distributions are the same as in Section \ref{sec:exp_conf_sets}. The number of tests $m$ is set to be $1000$.  The bootstrap weights $\{ w_{ik} , 1\leq i\leq n, 1\leq k\leq m \}$ are IID from $\mathcal{N}(1,1)$. For each setup, we report the average false discovery proportion and empirical power based on 1000 simulations. The FDP nominal level takes value in  $\{ 5\%, 10\%, 15\%, 20\%, 25\%\}$.


Tables \ref{table:multiple1} and \ref{table:multiple2} show the empirical FDPs and powers for the weaker signal case with $\mu_k = 1.5  \sqrt{2 (\log m)/n}$; while Tables \ref{table:multiple3} and \ref{table:multiple4} show the results for the stronger signal case with $\mu_k = 3  \sqrt{2 (\log m)/n}$. Moreover, Tables \ref{table:multiple1} and \ref{table:multiple3} consider different error distributions when $n=100$ and $s=5$.
When the error is from a $t$-Weilbull mixture distribution, Tables \ref{table:multiple2} and \ref{table:multiple4} present the results for different combinations of $(s,n)$, revealing the influence of $s$ on the difficulty of the problem.
In particular, the combination of $s=10$, $n=100$ and ${\rm signal~strength}=1.5\sqrt{2(\log m)/n}$ corresponds to the most challenging scenario. Increasing either the sample size or the signal strength improves both the FDP control and power performance, which is consistent with our theoretical result in Theorem \ref{FDP.thm}.
In summary, with various types of heavy-tailed errors and across different settings, the proposed robust testing procedure performs well and steadily in terms of FDP control and power.

%


\setcounter{table}{5}

\begin{table}[!t]
\centering
\caption{Standard deviations of the estimated quantiles for $(n,d)=(100,5)$ and nominal levels  $1-\alpha=[0.95,0.9,0.85,0.8,0.75]$. The weights $W_i$ are generated from $\mathcal{N}(1,1)$}
\begin{tabular}{lllllllll}
\hline
Noise ~~~~~& Approach  &  $0.95$ & $0.9$ & $0.85$&$0.8$&$0.75$ \\
\hline
\multicolumn{4}{l}{Gaussian}\\
& boot-Huber & 0.210 & 0.301 & 0.365 & 0.412 & 0.443\\
& boot-OLS & 0.214 & 0.299 & 0.365 & 0.416 & 0.446\\
\multicolumn{4}{l}{$t_\nu$}\\
& boot-Huber & 0.191 & 0.295 & 0.364 & 0.399 & 0.442\\
& boot-OLS & 0.222 & 0.336 & 0.420 & 0.467 & 0.491\\
\multicolumn{4}{l}{Gamma}\\
& boot-Huber & 0.191 & 0.292 & 0.366 & 0.410 & 0.441\\
& boot-OLS & 0.220 & 0.309 & 0.384 & 0.423 & 0.456\\
\multicolumn{4}{l}{Wbl mix}\\
& boot-Huber & 0.176 & 0.265 & 0.343 & 0.392 & 0.435\\
& boot-OLS & 0.201 & 0.308 & 0.392 & 0.440 & 0.472\\
\multicolumn{4}{l}{Par mix}\\
& boot-Huber & 0.181 & 0.286 & 0.358 & 0.408 & 0.443\\
& boot-OLS & 0.196 & 0.324 & 0.414 & 0.466 & 0.488\\
\multicolumn{4}{l}{Logn mix}\\
& boot-Huber & 0.176 & 0.268 & 0.331 & 0.392 & 0.425\\
& boot-OLS & 0.189 & 0.327 & 0.416 & 0.461 & 0.488\\
\hline
\end{tabular}
\label{table:stdn100}
\end{table}

\bigskip
\begin{table}[!t]
\centering
\caption{Standard deviations of the estimated quantiles for $(n,d)=(200,5)$ and nominal levels $1-\alpha=[0.95,0.9,0.85,0.8,0.75]$. The weights $W_i$ are generated from $\mathcal{N}(1,1)$}
\begin{tabular}{lllllllll}
\hline
Noise ~~~~~& Approach  &  $0.95$ & $0.9$ & $0.85$&$0.8$&$0.75$ \\
\hline
\multicolumn{4}{l}{Gaussian}\\
& boot-Huber & 0.232 & 0.333 & 0.390 & 0.427 & 0.456\\
& boot-OLS & 0.236 & 0.340 & 0.388 & 0.431 & 0.455\\
\multicolumn{4}{l}{$t_\nu$}\\
& boot-Huber & 0.205 & 0.291 & 0.357 & 0.407 & 0.445\\
& boot-OLS & 0.236 & 0.342 & 0.437 & 0.481 & 0.497\\
\multicolumn{4}{l}{Gamma}\\
& boot-Huber & 0.212 & 0.295 & 0.358 & 0.401 & 0.440\\
& boot-OLS & 0.232 & 0.307 & 0.366 & 0.415 & 0.451\\
\multicolumn{4}{l}{Wbl mix}\\
& boot-Huber & 0.194 & 0.292 & 0.364 & 0.409 & 0.439\\
& boot-OLS & 0.220 & 0.335 & 0.409 & 0.447 & 0.480\\
\multicolumn{4}{l}{Par mix}\\
& boot-Huber & 0.168 & 0.252 & 0.333 & 0.395 & 0.433\\
& boot-OLS & 0.189 & 0.314 & 0.415 & 0.469 & 0.495\\
\multicolumn{4}{l}{Logn mix}\\
& boot-Huber & 0.232 & 0.314 & 0.379 & 0.418 & 0.447\\
& boot-OLS & 0.245 & 0.363 & 0.452 & 0.491 & 0.500\\
\hline
\end{tabular}
\label{table:stdn200}
\end{table}

\begin{table}[!t]
\centering
\caption{Average coverage probabilities for $(n,d)=(100,5)$ and nominal levels $1-\alpha=[0.95,0.9,0.85,0.8,0.75]$ when $\bX_i$ are IID from a multivariate uniform distribution. Each component of $\btheta^*$ follows $\text{Ber}(0.5)$, and $W_i$ are generated from $\mathcal{N}(1,1)$.}
\begin{tabular}{lllllllll}
\hline
Noise ~~~~~& Approach  &  $0.95$ & $0.9$ & $0.85$&$0.8$&$0.75$ \\
\hline
\multicolumn{4}{l}{Gaussian}\\
& boot-Huber & 0.946 & 0.898 & 0.848 & 0.781 & 0.725\\
& boot-OLS & 0.944 & 0.898 & 0.843 & 0.782 & 0.720\\
\multicolumn{4}{l}{$t_\nu$}\\
& boot-Huber & 0.968 & 0.919 & 0.877 & 0.825 & 0.777\\
& boot-OLS & 0.954 & 0.881 & 0.803 & 0.729 & 0.642\\
\multicolumn{4}{l}{Gamma}\\
& boot-Huber & 0.961 & 0.911 & 0.868 & 0.812 & 0.751\\
& boot-OLS & 0.958 & 0.900 & 0.842 & 0.778 & 0.716\\
\multicolumn{4}{l}{Wbl mix}\\
& boot-Huber & 0.963 & 0.907 & 0.866 & 0.808 & 0.748\\
& boot-OLS & 0.947 & 0.880 & 0.817 & 0.724 & 0.663\\
\multicolumn{4}{l}{Par mix}\\
& boot-Huber & 0.974 & 0.928 & 0.882 & 0.842 & 0.775\\
& boot-OLS & 0.963 & 0.897 & 0.815 & 0.715 & 0.634\\
\multicolumn{4}{l}{Logn mix}\\
& boot-Huber & 0.972 & 0.936 & 0.888 & 0.834 & 0.780\\
& boot-OLS & 0.962 & 0.901 & 0.804 & 0.701 & 0.615\\
\hline
\end{tabular}
\label{table:unif_n100}
\end{table}

\bigskip
\begin{table}[!t]
\centering
\caption{Average coverage probabilities for $(n,d)=(200,5)$ and nominal levels $1-\alpha=[0.95,0.9,0.85,0.8,0.75]$ when $\bX_i$ are IID from a multivariate uniform distribution. Each component of $\btheta^*$ follows $\text{Ber}(0.5)$, and $W_i$ are generated from $\mathcal{N}(1,1)$.}
\begin{tabular}{lllllllll}
\hline
Noise ~~~~~& Approach  &  $0.95$ & $0.9$ & $0.85$&$0.8$&$0.75$ \\
\hline
\multicolumn{4}{l}{Gaussian}\\
& boot-Huber & 0.953 & 0.893 & 0.844 & 0.799 & 0.743\\
& boot-OLS & 0.955 & 0.893 & 0.846 & 0.798 & 0.744\\
\multicolumn{4}{l}{$t_\nu$}\\
& boot-Huber & 0.960 & 0.910 & 0.850 & 0.795 & 0.750\\
& boot-OLS & 0.948 & 0.860 & 0.759 & 0.657 & 0.586\\
\multicolumn{4}{l}{Gamma}\\
& boot-Huber & 0.949 & 0.896 & 0.836 & 0.782 & 0.729\\
& boot-OLS & 0.947 & 0.886 & 0.817 & 0.765 & 0.704\\
\multicolumn{4}{l}{Wbl mix}\\
& boot-Huber & 0.957 & 0.906 & 0.861 & 0.811 & 0.766\\
& boot-OLS & 0.941 & 0.879 & 0.805 & 0.722 & 0.656\\
\multicolumn{4}{l}{Par mix}\\
& boot-Huber & 0.963 & 0.924 & 0.862 & 0.798 & 0.743\\
& boot-OLS & 0.958 & 0.869 & 0.751 & 0.669 & 0.581\\
\multicolumn{4}{l}{Logn mix}\\
& boot-Huber & 0.958 & 0.909 & 0.849 & 0.795 & 0.739\\
& boot-OLS & 0.947 & 0.861 & 0.723 & 0.600 & 0.531\\
\hline
\end{tabular}
\label{table:unif_n200}
\end{table}

\begin{table}[!t]
\centering
\caption{Average coverage probabilities for $(n,d)= (100,5)$ and nominal levels $1-\alpha=[0.95,0.9,0.85,0.8,0.75]$ when $\bX_i$ are IID from a multivariate normal distribution with Toplitz covariance structure and $\btheta^*$ is a vector of equally spaced points in $[0,1]$. The weights $W_i$ are generated from $\mathcal{N}(1,1)$.}
\begin{tabular}{lllllllll}
\hline
Noise ~~~~~& Approach  &  $0.95$ & $0.9$ & $0.85$&$0.8$&$0.75$ \\
\hline
\multicolumn{4}{l}{Gaussian}\\
& boot-Huber & 0.954 & 0.899 & 0.842 & 0.783 & 0.732\\
& boot-OLS & 0.952 & 0.901 & 0.842 & 0.778 & 0.726\\
\multicolumn{4}{l}{$t_\nu$}\\
& boot-Huber & 0.962 & 0.904 & 0.843 & 0.802 & 0.734\\
& boot-OLS & 0.948 & 0.870 & 0.772 & 0.678 & 0.594\\
\multicolumn{4}{l}{Gamma}\\
& boot-Huber & 0.962 & 0.906 & 0.841 & 0.786 & 0.736\\
& boot-OLS & 0.949 & 0.893 & 0.820 & 0.767 & 0.706\\
\multicolumn{4}{l}{Wbl mix}\\
& boot-Huber & 0.968 & 0.924 & 0.864 & 0.811 & 0.747\\
& boot-OLS & 0.958 & 0.894 & 0.811 & 0.737 & 0.665\\
\multicolumn{4}{l}{Par mix}\\
& boot-Huber & 0.966 & 0.910 & 0.849 & 0.790 & 0.733\\
& boot-OLS & 0.960 & 0.881 & 0.780 & 0.681 & 0.609\\
\multicolumn{4}{l}{Logn mix}\\
& boot-Huber & 0.968 & 0.922 & 0.875 & 0.811 & 0.763\\
& boot-OLS & 0.963 & 0.878 & 0.778 & 0.693 & 0.608\\
\hline
\end{tabular}
\label{table:toep_n100}
\end{table}

\bigskip
\begin{table}[!t]
\centering
\caption{Average coverage probabilities for $(n,d)= (200,5)$ and nominal levels $1-\alpha=[0.95,0.9,0.85,0.8,0.75]$ when $\bX_i$ are IID from a multivariate normal distribution with Toplitz covariance structure and $\btheta^*$ is a vector of equally spaced points in $[0,1]$. The weights $W_i$ are generated from $\mathcal{N}(1,1)$.}
\begin{tabular}{lllllllll}
\hline
Noise ~~~~~& Approach  &  $0.95$ & $0.9$ & $0.85$&$0.8$&$0.75$ \\
\hline
\multicolumn{4}{l}{Gaussian}\\
& boot-Huber & 0.943 & 0.873 & 0.813 & 0.761 & 0.706\\
& boot-OLS & 0.941 & 0.867 & 0.815 & 0.754 & 0.708\\
\multicolumn{4}{l}{$t_\nu$}\\
& boot-Huber & 0.956 & 0.907 & 0.850 & 0.791 & 0.729\\
& boot-OLS & 0.941 & 0.865 & 0.744 & 0.639 & 0.561\\
\multicolumn{4}{l}{Gamma}\\
& boot-Huber & 0.953 & 0.904 & 0.849 & 0.799 & 0.738\\
& boot-OLS & 0.943 & 0.895 & 0.841 & 0.779 & 0.717\\
\multicolumn{4}{l}{Wbl mix}\\
& boot-Huber & 0.961 & 0.906 & 0.843 & 0.788 & 0.739\\
& boot-OLS & 0.949 & 0.871 & 0.788 & 0.724 & 0.641\\
\multicolumn{4}{l}{Par mix}\\
& boot-Huber & 0.971 & 0.932 & 0.873 & 0.807 & 0.751\\
& boot-OLS & 0.963 & 0.889 & 0.779 & 0.674 & 0.572\\
\multicolumn{4}{l}{Logn mix}\\
& boot-Huber & 0.943 & 0.889 & 0.826 & 0.775 & 0.725\\
& boot-OLS & 0.936 & 0.844 & 0.715 & 0.597 & 0.516\\
\hline
\end{tabular}
\label{table:toep_n200}
\end{table}

\begin{table}[!t]
\centering
\caption{Empirical FDP and power for $(n,s)=(100,5)$. The nominal level $\alpha$ takes value in $\{0.05,0.10,0.15,0.20,0.25\}$. The signal strength $\mu_k = 1.5\sqrt{2 (\log m)/n}$ for $1 \leq  k \leq m_1$.}
~\\
\begin{tabular}{llllllll}
\hline   Noise & $\alpha$ & 0.05 &0.10 &0.15 &0.20 &0.25 \\
\hline
Gaussian\\
&FDP&0.027 &0.044 &0.075 &0.106 &0.138 \\
&Power&0.935 &0.962 &0.978 &0.986 &0.989 \\
$t_\nu$\\
&FDP&0.017 &0.030 &0.053 &0.080 &0.105 \\
&Power&0.928 &0.953 &0.969 &0.978 &0.983 \\
Gamma\\
&FDP&0.048 &0.076 &0.119 &0.159 &0.197 \\
&Power&0.957 &0.981 &0.993 &0.996 &0.999 \\
Wbl mix\\
&FDP&0.038 &0.060 &0.098 &0.130 &0.160 \\
&Power&0.951 &0.977 &0.988 &0.993 &1.000 \\
Par mix \\
&FDP&0.033 &   0.056   & 0.094&    0.129 &    0.165 \\
&Power&0.998 &0.999 &1.000 &1.000 &1.000 \\
Logn mix\\
&FDP&0.029 &0.067 &0.108 &0.149 &0.184 \\
&Power&0.999&1.000 &1.000 &1.000 &1.000 \\
\hline
\end{tabular}
\label{table:multiple1}
\end{table}

\begin{table}[!t]
\centering
\caption{Empirical FDP and power for the Wbl mix model. The nominal level $\alpha$ takes value in $\{0.05,0.10,0.15,0.20,0.25\}$. The signal strength $\mu_k = 1.5\sqrt{2 (\log m)/n}$ for $1 \leq  k \leq m_1$.}
~\\
\begin{tabular}{llllllll}
\hline
$s$ & $n$& $\alpha$ & 0.05 &0.10 &0.15 &0.20 &0.25 \\
\hline
2&100&FDP&0.061 &0.098 &0.143 &0.186 &0.232 \\
&&Power&0.981 &0.991 &0.995 &0.997 &0.998 \\
&200&FDP&0.048 &0.101 &0.149 &0.195 &0.242 \\
&&Power&0.989 &0.995 &0.997 &0.998 &0.998 \\
5&100&FDP&0.038 &0.060 &0.098 &0.130 &0.160 \\
&&Power&0.951 &0.977 &0.988 &0.993 &1.000 \\
&200&FDP&0.056 &0.092 &0.137 &0.180 &0.223 \\
&&Power&0.938 &0.991 &0.996 &0.997 &0.999 \\
10&100&FDP&0.006 &0.014 &0.022 &0.035&0.052 \\
&&Power&0.607 &0.825 &0.897 &0.940 &0.961 \\
&200&FDP&0.043 &0.070 &0.110 &0.145 &0.187\\
&&Power&0.971 &0.983 &0.989 &0.993 &0.995 \\
\hline
\end{tabular}
\label{table:multiple2}
\end{table}

\begin{table}[!t]
\centering
\caption{Empirical FDP and power for $(n,s)=(100,5)$. The nominal level $\alpha$ takes value in $\{0.05,0.10,0.15,0.20,0.25\}$. The signal strength $\mu_k = 3\sqrt{2 (\log m)/n}$ for $1 \leq  k \leq m_1$.}
~\\
\begin{tabular}{llllllll}
\hline   Noise & $\alpha$ & 0.05 &0.10 &0.15 &0.20 &0.25 \\
\hline
Gaussian\\
&FDP&0.015 &0.039 &0.063 &0.093 &0.124 \\
&Power&1.000 &1.000 &1.000 &1.000 &1.000 \\
$t_\nu$\\
&FDP&0.009 &0.027 &0.046 &0.072 &0.098 \\
&Power&0.999 &1.000 &1.000 &1.000 &1.000 \\
Gamma\\
&FDP&0.038 &0.063 &0.103 &0.136 &0.178 \\
&Power&1.000 &1.000 &1.000 &1.000 &1.000 \\
Wbl mix\\
&FDP&0.038 &0.049 &0.089 &0.120 &0.156 \\
&Power&0.999 &1.000 &1.000 &1.000 &1.000 \\
Par mix \\
&FDP&0.037 &   0.060   & 0.100&    0.135 &    0.167 \\
&Power&1.000 &1.000 &1.000 &1.000 &1.000 \\
Logn mix\\
&FDP&0.024 &0.067 &0.099 &0.128 &0.157 \\
&Power&1.000 &1.000 &1.000 &1.000 &1.000 \\
\hline
\end{tabular}
\label{table:multiple3}
\end{table}

\begin{table}[!t]
\centering
\caption{Empirical FDP and power for the Wbl mix model. The nominal level $\alpha$ takes value in $\{0.05,0.10,0.15,0.20,0.25\}$. The signal strength $\mu_k = 3\sqrt{2 (\log m)/n}$ for $1 \leq  k \leq m_1$.}
~\\
\begin{tabular}{llllllll}
\hline
$s$ & $n$& $\alpha$ & 0.05 &0.10 &0.15 &0.20 &0.25 \\
\hline
2&100&FDP&0.042 &0.087 &0.125 &0.179 &0.221 \\
&&Power&1.000 &1.000 &1.000 &1.000 &1.000 \\
&200&FDP&0.049 &0.102 &0.144 &0.187 &0.234 \\
&&Power&1.000 &1.000 &1.000 &1.000 &1.000 \\
5&100&FDP&0.026 &0.049 &0.089 &0.120 &0.156 \\
&&Power&0.999 &1.000 &1.000 &1.000 &1.000 \\
&200&FDP&0.040 &0.069 &0.089 &0.132 &0.184 \\
&&Power&1.000 &1.000 &1.000 &1.000 &1.000 \\
10&100&FDP&0.011 &0.014 &0.022 &0.041 &0.054 \\
&&Power&0.991 &0.995 &0.999 &0.999 &0.999 \\
&200&FDP&0.040&0.069 &0.102 &0.131 &0.166 \\
&&Power&1.000 &1.000 &1.000 &1.000 &1.000 \\
\hline
\end{tabular}
\label{table:multiple4}
\end{table}


\begin{thebibliography}{9}

\bibitem[Arlot, Blanchard and Roquain(2010)]{ABR2010}
{\sc Arlot, S., Blanchard, G.} and {\sc Roquain, E.} (2010).
Some nonasymptotic results on resampling in high dimension. I. Confidence regions.
{\it Ann. Statist.} \textbf{38} 51--82.

\bibitem[Audibert and Catoni(2011)]{AC2011}
{\sc Audibert, J.-Y.} and {\sc Catoni, O.} (2011).
Robust linear least squares regression.
{\it Ann. Statist.} \textbf{39} 2766--2794.

\bibitem[Barras, Scaillet and Wermers(2010)]{BSW2010}
        {\sc Barras, L., Scaillet, O.} and {\sc Wermers, R.} (2010).
        False discoveries in mutual fund performance: Measuring luck in estimated alphas.
        {\it J. Finance} {\bf 65} 179--216.

\bibitem[Benjamini and Hochberg(1995)]{BH1995}	
	{\sc Benjamini, Y.} and {\sc Hochberg, Y.} (1995).
	Controlling the false discovery rate: A practical and powerful approach to multiple testing.
	{\it J. R. Stat. Soc. Ser. B. Stat. Methodol.} {\bf 57} 289--300.

\bibitem[Berk and Green(2004)]{BG2004}
	{\sc Berk, J.\,B.} and {\sc Green, R.\,C.} (2004).
	Mutual fund flows and performance in rational markets.
	{\it J. Polit. Econ.} {\bf 112} 1269--1295.

\bibitem[Brownlees, Joly and Lugosi(2015)]{BJL2015}
{\sc Brownlees, C.}, {\sc Joly, E.} and {\sc Lugosi, G.} (2015).
Empirical risk minimization for heavy-tailed losses.
{\it Ann. Statist.} \textbf{43} 2507--2536.

\bibitem[{Catoni(2012)}]{C2012}
{\sc Catoni, O.} (2012).
Challenging the empirical mean and empirical variance: A deviation study.
\textit{Ann. Inst. Henri Poincar\'e Probab. Stat.} \textbf{48} 1148--1185.

\bibitem[{Catoni and Giulini(2017)}]{CG2017}
{\sc Catoni, O.} and {\sc Giulini, L.} (2017).
Dimension-free PAC-Bayesian bounds for matrices, vectors, and linear least squares regression.
Available at \href{https://arxiv.org/abs/1712.02747}{arXiv:1712.02747}.


\bibitem[Chatterjee and Bose(2005)]{CB2005}
	{\sc Chatterjee, S.} and {\sc Bose, A.} (2005).
	Generalized bootstrap for estimating equations. {\em Ann. Statist.} {\bf 33} 414--436.
	
\bibitem[Chernozhukov, Chetverikov and Kato(2013)]{CCK2013}
	{\sc Chernozhukov, V., Chetverikov, D.} and {\sc Kato, K.} (2013).
	Gaussian approximations and multiplier bootstrap for maxima of sums of high-dimensional random vectors. {\em Ann. Statist.} {\bf 41} 2786--2819.

\bibitem[Chernozhukov, Chetverikov and Kato(2014)]{CCK2014}
	{\sc Chernozhukov, V., Chetverikov, D.} and {\sc Kato, K.} (2014).
	Gaussian approximation of suprema of empirical processes. {\em Ann. Statist.} {\bf 42} 1564--1597.

\bibitem[Delaigle, Hall and Jin(2011)]{DHJ2011}	
	{\sc Delaigle, A.}, {\sc Hall, P.} and {\sc Jin, J.} (2011).
 	Robustness and accuracy of methods for high dimensional data analysis based on Student's $t$-statistic.
 	{\it J. R. Stat. Soc. Ser. B. Stat. Methodol.} {\bf 73} 283--301.

\bibitem[Desai and Storey(2012)]{DS2012}
	{\sc Desai, K.\,H.} and {\sc Storey, J.\,D.} (2012).
	Cross-dimensional inference of dependent high-dimensional data.
	{\it J. Amer. Statist. Assoc.} {\bf 107} 135--151.

\bibitem[Devroye et~al.(2016)]{DLLO2016}
\textsc{Devroye, L.}, \textsc{Lerasle, M.}, \textsc{Lugosi, G.} and \textsc{Oliveira, R.\,I.} (2016).
Sub-Gaussian mean estimators.
{\it Ann. Statist.} \textbf{44} 2695--2725.

\bibitem[{Dudoit and van der Laan(2008)}]{DV2008}
Dudoit, S. and van der Laan, M.\,J. (2008). 
{\it Multiple Testing Procedures with Applications to Genomics.} 
Springer, New York.

\bibitem[Efron(2010)] {E2010}
      {\sc Efron, B.} (2010).
      {\em Large-Scale Inference: Empirical Bayes Methods for Estimation, Testing, and Prediction. Institute of Mathematical Statistics (IMS) Monographs} \textbf{1}. Cambridge Univ. Press, Cambridge.

\bibitem[Fama and French(1993)]{FF1993}
	{\sc Fama, E.\,F.} and {\sc French, K.\,R.} (1993).
	Common risk factors in the returns on stocks and bonds.
	{\it  J. Financial Econ.} {\bf 33} 3--56.
	
\bibitem[Fan, Hall and Yao(2007)]{FHY2007}	
	{\sc Fan, J.}, {\sc Hall, P.} and {\sc Yao, Q.} (2007).
	To how many simultaneous hypothesis tests can normal, Student's $t$ or bootstrap calibration be applied?
 	{\it J. Amer. Statist. Assoc.} {\bf 102} 1282--1288.

\bibitem[Fan, Han and Gu(2012)]{FHG2012}	
	{\sc Fan, J.}, {\sc Han, X.} and {\sc Gu, W.} (2012).
	Estimating false discovery proportion under arbitrary covariance dependence.
 	{\it J. Amer. Statist. Assoc.} {\bf 107} 1019--1035.


\bibitem[Fan, Li and Wang(2017)]{FLW2017}
\textsc{Fan, J.}, \textsc{Li, Q.} and \textsc{Wang, Y.} (2017).
Estimation of high dimensional mean regression in the absence of symmetry and light tail assumptions.
{\it J. R. Stat. Soc. Ser. B. Stat. Methodol.} {\bf 79} 247--265.

\bibitem[Fan, Liao and Yao(2015)]{FLY2015}
{\sc Fan, J., Liao, Y.} and {\sc Yao, J.} (2015).
Power enhancement in high-dimensional cross-sectional tests.
{\it Econometrica} {\bf 83} 1497--1541.

\bibitem[Friguet, Kloareg and Causeur(2009)]{FKC2009}
	{\sc Friguet, C., Kloareg, M.} and {\sc Causeur, D.} (2009).
	A factor model approach to multiple testing under dependence.
	{\it J. Amer. Statist. Assoc.} {\bf 104} 1406--1415.
	
\bibitem[Giulini(2017)]{G2017}
	{\sc Giulini, I.} (2017).
	Robust PCA and pairs of projections in a Hilbert space. 
	{\it Electron. J. Stat.} {\bf 11} 3903--3926.

\bibitem[Hahn, Kuelbs and Weiner(1990)]{HKW1990}
{\sc Hahn, M.\,G., Kuelbs, J.} and {\sc Weiner, D.\,C.} (1990).
The asymptotic joint distribution of self-normalized censored sums and sums of squares.
{\it   Ann. Probab.} {\bf 18} 1284--1341.

\bibitem[{Hsu and Sabato(2016)}]{HS2016}
{\sc Hsu, D.} and {\sc Sabato, S.} (2016).
Loss minimization and parameter estimation with heavy tails.
{\it J. Mach. Learn. Res.} {\bf 17}(18) 1--40.

\bibitem[{Huber(1964)}]{H1964}
{\sc Huber, P.\,J.} (1964).
{Robust estimation of a location parameter.}
{\it Ann. Math. Statist.} \textbf{35} 73--101.

\bibitem[{Huber and Ronchetti(2009)}]{HR2009}
{\sc Huber, P.\,J.} and {\sc Ronchetti, E.\,M.} (2009).
{\it Robust Statistics,} 2nd ed.
Wiley, New York.


\bibitem[Lan and Du(2019)]{LD2017}
{\sc Lan, W.} and {\sc Du, L.} (2019).
A factor-adjusted multiple testing procedure with application to mutual fund selection. 
{\it J. Bus. Econom. Statist.} {\bf 37} 147--157.

\bibitem[Lepski\u \i(1991)]{Lep1992}
{\sc Lepski\u \i, O.\,V.} (1991).
Asymptotically minimax adaptive estimation. I. Upper bounds. Optimally adaptive estimates.
{\it Teor. Veroyatn. Primen.} {\bf 36} 645--659.



\bibitem[Lintner(1965)]{L1965}
	{\sc Lintner, J.} (1965).
	The valuation of risk assets and the selection of risky investment in stock portfolios and capital budgets.
	{\it Rev. Econ. Stat.} {\bf 47} 13--37.

\bibitem[Liu and Shao(2014)]{LS2014}
	{\sc Liu, W.} and {\sc Shao, Q.-M.} (2014).
	Phase transition and regularized bootstrap in large-scale $t$-tests with false discovery rate control.
	{\it Ann. Statist.} {\bf 42} 2003--2025.

\bibitem[Lugosi and Mendelson(2019)]{LM2017}
{\sc Lugosi, G.} and {\sc Mendelson, S.} (2019).
Sub-Gaussian estimators of the mean of a random vector.
{\it Ann. Statist.} {\bf 47}  783--794.

\bibitem[Minsker(2015)]{M2015}
{\sc Minsker, S.} (2015)
Geometric median and robust estimation in Banach spaces.
{\em Bernoulli} {\bf 21} 2308--2335.

\bibitem[Minsker(2018)]{M2016}
{\sc Minsker, S.} (2018).
Sub-Gaussian estimators of the mean of a random matrix with heavy-tailed entries.
{\it Ann. Statist.} {\bf 46} 2871--2903. 


\bibitem[Qi and Sun(1999)]{QS1999}
{\sc Qi, L.} and {\sc Sun, D.} (1999).
A survey of some nonsmooth equations and smoothing Newton methods.
In {\it Progress in Optimization} 121--146.  Springer, Boston, MA.

\bibitem[Sharpe(1964)]{S1964}	
	{\sc Sharpe, W.\,F.} (1964).
	Capital asset prices: A theory of market equilibrium under conditions of risk.
	{\it J. Finance} {\bf 19} 425--442.

\bibitem[Spokoiny and Zhilova(2015)]{SZ2015}	
	{\sc Spokoiny, V.} and {\sc Zhilova, M.} (2015).
	Bootstrap confidence sets under model misspecification.
	{\it Ann. Statist.} {\bf 43} 2653--2675.

\bibitem[{Sun~{\it et al.}(2018)}]{SZF2016}
	{\sc Sun, Q., Zhou, W.-X.} and {\sc Fan, J.} (2018).
	Adaptive Huber regression.
	{\it J. Amer. Statist. Assoc.} To appear.
	Available at \href{https://arxiv.org/abs/1706.06991}{arXiv:1706.06991}.

\bibitem[{van der Vaart and Wellner(1996)}]{VW1996}
	{\sc van der Vaart, A.\,W.} and {\sc Wellner, J.}  (1996).
	{\it Weak Convergence and Empirical Processes: With Applications to Statistics.}
	Springer, New York.

\bibitem[{Vershynin(2018)}]{V2012}
	{\sc Vershynin, R.} (2018).
	{\it High-Dimensional Probability: An Introduction with Applications in Data Science}.
	Cambridge Univ. Press, Cambridge.

\bibitem[Wang et al.(2017)]{WZHO2017}
	{\sc Wang, J., Zhao, Q., Hastie, T.} and {\sc Owen, A.\,B.} (2017).
	Confounder adjustment in multiple hypothesis testing.
	{\it Ann. Statist.} {\bf 45} 1863--1894.

\bibitem[Wilks(1938)]{W1938}	
	{\sc Wilks, S.\,S.} (1938).
	The large-sample distribution of the likelihood ratio for testing composite hypotheses.
	{\it Ann. Math. Statist.} {\bf 9} 60--62.

\bibitem[Zhilova(2016)]{Z2016}
	{\sc Zhilova, M.} (2016).
	Non-classical Berry-Esseen inequality and accuracy of the weighted bootstrap. 
	Available at \href{https://arxiv.org/abs/1611.02686}{arXiv:1611.02686.}
	
\bibitem[Zhou et al.(2018)]{ZBFL2017}
	{\sc Zhou, W.-X., Bose, K., Fan, J.} and {\sc Liu, H.} (2018).
	A new perspective on robust $M$-estimation: Finite sample theory and applications to dependence-adjusted multiple testing.
	{\it Ann. Statist.} {\bf 46} 1904--1931.


\end{thebibliography}

\begin{thebibliography}{9}

\bibitem[Adamczak et al.(2011)]{Adam2011}
	{\sc Adamczak, R., Litvak, A.\,E., Pajor, A.} and {\sc Tomczak-Jaegermann, N.} (2011).
	Restricted isometry property of matrices with independent columns and neighborly polytopes by random sampling.
	{\it Constr. Approx.} {\bf 34} 61--88.

\bibitem[Ball(1993)]{B1993}
{\sc Ball, K.} (1993).
The reverse isoperimetric problem for Gaussian measure.
{\it Discrete Comput. Geom.} {\bf 10} 411--420.

\bibitem[Bentkus(2005)]{B2005}
{\sc Bentkus, V.} (2005).
A Lyapunov-type bound in $R^d$.
{\it Theory Probab. Appl.} {\bf 49} 311--323.

\bibitem[{Bousquet(2003)}]{B2003}
{\sc Bousquet, O.} (2003).
Concentration inequalities for sub-additive functions using the entropy method.
{\it In Stochastic Inequalities and Applications. Progress in Probability} {\bf 56} 213--247. Birkh\"auser, Basel.

\bibitem[{Bubeck, Cesa-Bianchi and Lugosi(2013)}]{BCL2013}
{\sc Bubeck, S.}, {\sc Cesa-Bianchi, N.} and {\sc Lugosi, G.} (2013).
Bandits with heavy tail.
\textit{IEEE Trans. Inform. Theory} \textbf{59} 7711--7717. 

\bibitem[{Catoni(2012)}]{C2012}
{\sc Catoni, O.} (2012).
Challenging the empirical mean and empirical variance: A deviation study.
\textit{Ann. Inst. Henri Poincar\'e Probab. Stat.} \textbf{48} 1148--1185.


\bibitem[de la Pe\~na, Lai and Shao(2009)]{DLS2009}
{\sc de la Pe\~na, V.\,H., Lai, T.\,L.} and {\sc Shao, Q.-M.} (2009).
{\it Self-Normalized Processes: Limit Theory and Statistical Applications}. Springer, Berlin.

\bibitem[De Luca and Facchinei and Kanzow(1996)]{DFK1996}
{\sc   De Luca, T.}, {\sc Facchinei, F.} and {\sc Kanzow, C.} (1996).
A semismooth equation approach to the solution of nonlinear complementarity problems.
{\it Math. Program.} {\bf 75} 407--439.


\bibitem[Devroye et~al.(2016)]{DLLO2016}
\textsc{Devroye, L.}, \textsc{Lerasle, M.}, \textsc{Lugosi, G.} and \textsc{Oliveira, R.\,I.} (2016).
Sub-{G}aussian mean estimators.
{\it Ann. Statist.} \textbf{44} 2695--2725.

\bibitem[Falk(1999)]{falk1999simple}	
	{\sc Falk, M.} (1999).
 	A simple approach to the generation of uniformly distributed random variables with prescribed correlations.
 	{\it Comm. Statist. Simulation Comput.} {\bf 28} 785--791.



\bibitem[Fan, Li and Wang(2017)]{FLW2017}	
	{\sc Fan, J.}, {\sc Li, Q.} and {\sc Wang, Y.} (2017).
 	Estimation of high dimensional mean regression in the absence of symmetry and light tail assumptions.
 	{\it J. R. Stat. Soc. Ser. B. Stat. Methodol.} {\bf 79} 247--265.

\bibitem[Fan et~al.(2018)]{FLSZ2015}
\textsc{Fan, J.}, \textsc{Liu, H.}, \textsc{Sun, Q.} and \textsc{Zhang, T.}
  (2018).
\newblock {I-LAMM} for sparse learning: Simultaneous control of algorithmic
  complexity and statistical error. {\it Ann. Statist.} \textbf{96} 1348--1360.


\bibitem[Hahn, Kuelbs and Weiner(1990)]{HKW1990}
{\sc Hahn, M.\,G., Kuelbs, J.} and {\sc Weiner, D.\,C.} (1990).
The asymptotic joint distribution of self-normalized censored sums and sums of squares.
{\it   Ann. Probab.} {\bf 18} 1284--1341.

\bibitem[Hiriart-Urruty(2001)]{HL2001}
{\sc Hiriart-Urruty,, J.\,B.} and {\sc Lemar\'{e}chal, C. } (2015).
{\it Fundamentals of Convex Analysis}.
Springer-Verlag.


\bibitem[{Ledoux and Talagrand(1991)}]{LT1991}
{\sc Ledoux, M.} and {\sc Talagrand, M.} (1991).
{\it Probability in Banach Spaces: Isoperimetry and Processes}.
Springer-Verlag, Berlin.

\bibitem[Liu and Shao(2014)]{LS2014}
	{\sc Liu, W.} and {\sc Shao, Q.-M.} (2014).
	Phase transition and regularized bootstrap in large-scale $t$-tests with false discovery rate control.
	{\it Ann. Statist.} {\bf 42} 2003--2025.



\bibitem[Lepski\u \i(1991)]{Lep1992}
{\sc Lepski\u \i, O.\,V.} (1991).
Asymptotically minimax adaptive estimation. I. Upper bounds. Optimally adaptive estimates.
{\it Teor. Veroyatn. Primen.} {\bf 36} 645--659.

\bibitem[{Loh and Wainwright(2015)}]{LW2015}
{\sc Loh, P.-L.} and {\sc Wainwright, M.\,J.} (2015).
Regularized $M$-estimators with nonconvexity: Statistical and algorithmic theory for local optima.
{\it J. Mach. Learn. Res.} {\bf 16} 559--616.


\bibitem[Minsker(2018)]{M2016}
{\sc Minsker, S.} (2018).
Sub-Gaussian estimators of the mean of a random matrix with heavy-tailed entries.
{\it Ann. Statist.} {\bf 46} 2871--2903. 


\bibitem[Pugh(2015)]{P2015}
{\sc Pugh, C.\,C.} (2015).
{\it Real Mathematical Analysis}, 2nd ed.
Springer-Verlag, New York.

\bibitem[Qi and Sun(1999)]{QS1999}
	{\sc Qi, L.} and {\sc Sun, D.} (1999).
	A survey of some nonsmooth equations and smoothing Newton methods.
	In {\it Progress in Optimization} 121--146. Springer, Boston, MA.

\bibitem[Spokoiny(2012)]{S2012}
    {\sc Spokoiny, V.} (2012).
	Parametric estimation. Finite sample theory.
	{\it Ann. Statist.} {\bf 40} 2877--2909.

\bibitem[Spokoiny(2013)]{S2013}
	{\sc Spokoiny, V.} (2013).
	Bernstein--von Mises theorem for growing parameter dimension. Preprint.
	Available at \href{https://arxiv.org/abs/1302.3430}{arXiv:1302.3430.}

\bibitem[Spokoiny and Zhilova(2015)]{SZ2015}	
	{\sc Spokoiny, V.} and {\sc Zhilova, M.} (2015).
	Bootstrap confidence sets under model misspecification.
	{\it Ann. Statist.} {\bf 43} 2653--2675.
	
\bibitem[Storey, Taylor and Siegmund(2004)]{STS2004}	
	{\sc Storey, J.\,D.}, {\sc Taylor, J.\,E.} and {\sc Siegmund, D.} (2004).
	Strong control, conservative point estimation and simultaneous conservative consistency of false discovery rate: A unified approach.
	{\it J. R. Stat. Soc. Ser. B. Stat. Methodol.} {\bf 66} 187--205.

\bibitem[{Vershynin(2018)}]{V2018}
	{\sc Vershynin, R.} (2018)
	{\it High-Dimensional Probability: An Introduction with Applications in Data Science.}
	Cambridge Univ. Press, Cambridge.

\bibitem[Zhou et al.(2018)]{ZBFL2017}
	{\sc Zhou, W.-X., Bose, K., Fan, J.} and {\sc Liu, H.} (2018).
	A new perspective on robust $M$-estimation: Finite sample theory and applications to dependence-adjusted multiple testing.
	 {\it Ann. Statist.} {\bf 46} 1904--1931.

\end{thebibliography}
\end{document}